\documentclass[preprint]{imsart}

\RequirePackage{fix-cm}

\usepackage{latexsym}
\usepackage{array}
\usepackage{color}
\usepackage{rotating}
\usepackage{epsfig, wrapfig}
\usepackage{graphicx, lscape, graphics, color} % removed subfig 
\usepackage{verbatim}
\usepackage{latexsym}
\usepackage{fancyvrb}
\usepackage{pdfpages}
\usepackage{fmtcount}
\usepackage{multirow}
\usepackage{url}
\usepackage[english]{babel}
\usepackage[utf8]{inputenc}
\usepackage{tikz} 
\usetikzlibrary{arrows, decorations.pathmorphing, backgrounds, fit, positioning, shapes.symbols, chains}
\usetikzlibrary{decorations.markings}
\usepackage{lscape}
\usepackage{epstopdf}
\usepackage{algpseudocode}
\usepackage{algorithm}
\usepackage{kotex}
\usepackage{soul}
\usepackage{enumerate}
\usepackage{colortbl}
\usepackage{booktabs}
\usepackage{appendix}
\usepackage{mathtools}
\usepackage{tikz} 
\usetikzlibrary{arrows, decorations.pathmorphing, backgrounds, fit, positioning, shapes.symbols, chains}
\usetikzlibrary{decorations.markings}
\usetikzlibrary{arrows.meta}

\DeclareMathOperator*{\argmin}{arg\,min}

\DeclarePairedDelimiter\floor{\lfloor}{\rfloor}

\newcommand{\cV}{\mathcal{V}}

\newcommand{\bbH}{\mathbb{H}}

\newcommand{\bbR}{\mathbb{R}}

\newcommand{\bV}{\mathbf{V}}

\newcommand{\bX}{\mathbf{X}}
\newcommand{\bY}{\mathbf{Y}}
\newcommand{\bZ}{\mathbf{Z}}

\def\Frechet{Fr\'{e}chet}

\def\av{\mathbf a}

\def\ev{\mathbf e}

\def\vv{\mathbf v}

\def\xv{\mathbf x}
\def\yv{\mathbf y}

\RequirePackage{amsthm, amsmath, amsfonts, amssymb}
\RequirePackage[authoryear]{natbib}%% uncomment this for author-year citations
\usepackage{hyperref}
\theoremstyle{plain}

\newtheorem{theorem}{Theorem}
\newtheorem{lemma}{Lemma}
\newtheorem{definition}{Definition}
\newtheorem{remark}{Remark}
\newtheorem{corollary}{Corollary}
\newtheorem{example}{Example}
\newtheorem{proposition}{Proposition}
\theoremstyle{remark}

\begin{document}

\begin{frontmatter}
\title{Huber means on Riemannian manifolds}
%\title{A sample article title with some additional note\thanksref{t1}}
\runauthor{Lee and Jung}
\runtitle{Huber means on Riemannian manifolds}
%\thankstext{T1}{A sample additional note to the title.}

\begin{aug}
%%%%%%%%%%%%%%%%%%%%%%%%%%%%%%%%%%%%%%%%%%%%%%%
%% Only one address is permitted per author. %%
%% Only division, organization and e-mail is %%
%% included in the address.                  %%
%% Additional information can be included in %%
%% the Acknowledgments section if necessary. %%
%% ORCID can be inserted by command:         %%
%% \orcid{0000-0000-0000-0000}               %%
%%%%%%%%%%%%%%%%%%%%%%%%%%%%%%%%%%%%%%%%%%%%%%%
\author[A]{\fnms{Jongmin}~\snm{Lee}\ead[label=e1]{jongmin.lee@pusan.ac.kr}\orcid{0000-0003-1723-4615}}
\and
\author[B]{\fnms{Sungkyu}~\snm{Jung}\ead[label=e2]{sungkyu@snu.ac.kr}\orcid{0000-0002-6023-8956}}
%%%%%%%%%%%%%%%%%%%%%%%%%%%%%%%%%%%%%%%%%%%%%%
%% Addresses                                %%
%%%%%%%%%%%%%%%%%%%%%%%%%%%%%%%%%%%%%%%%%%%%%%
\address[A]{Department of Statistics,
Pusan National University\printead[presep={,\ }]{e1}}

\address[B]{Department of Statistics and Institute for Data Innovation in Science,
Seoul National University\printead[presep={,\ }]{e2}}
\end{aug}

\begin{abstract}
This article introduces Huber means on Riemannian manifolds, providing a robust alternative to the Fr\'{e}chet mean by integrating elements of both $L_2$ and $L_1$ loss functions. The Huber means are designed to be highly resistant to outliers while maintaining efficiency, making it a valuable generalization of Huber's $M$-estimator for manifold-valued data. We comprehensively investigate the statistical and computational aspects of Huber means, demonstrating their utility in manifold-valued data analysis. Specifically, we establish nearly minimal conditions for ensuring the existence and uniqueness of the Huber mean and discuss regularity conditions for unbiasedness. The Huber means are consistent and enjoy the central limit theorem. Additionally, we propose a novel moment-based estimator for the limiting covariance matrix, which is used to construct a robust one-sample location test procedure and an approximate confidence region for location parameters. The Huber mean is shown to be highly robust and efficient in the presence of outliers or under heavy-tailed distributions. Specifically, it achieves a breakdown point of at least 0.5, the highest among all isometric equivariant estimators, and is more efficient than the Fr\'{e}chet mean under heavy-tailed distributions. Numerical examples on spheres and the space of symmetric positive-definite matrices further illustrate the efficiency and reliability of the proposed Huber means on Riemannian manifolds.
\end{abstract}



\end{frontmatter}
%%%%%%%%%%%%%%%%%%%%%%%%%%%%%%%%%%%%%%%%%%%%%%
%% Please use \tableofcontents for articles %%
%% with 50 pages and more                   %%
%%%%%%%%%%%%%%%%%%%%%%%%%%%%%%%%%%%%%%%%%%%%%%
\tableofcontents

\section{Introduction}

Locating the mean of data has long been a fundamental task in statistics, serving as the basis for numerous statistical inferences and computations. Conventional mean estimation methods often rely on the $L_2$ loss function, exemplified by the \Frechet\ mean for data lying in spaces more general than Euclidean vector spaces. However, these methods can be significantly influenced by outliers, especially when dealing with manifold-valued data, which are increasingly encountered in contemporary sciences.

In this article, we introduce the concept of Huber means on Riemannian manifolds. The Huber means are defined as the minimizers of the expected Huber loss with a prespecified robustification parameter $c\in [0, \infty]$. It offers a robust alternative to the \Frechet\ mean by combining elements of $L_2$ loss for distances smaller than $c$, thereby conserving efficiency, and the $L_1$ loss for larger distances, inducing robustness to the tail of the distribution. The cutoff parameter $c$ controls the balance between these loss functions. 
The Huber mean serves as a natural generalization of Huber's $M$-estimator \citep{huber1964robust} to the manifold setting, and can be viewed as a robust extension of the Fr\'{e}chet mean; see Figure~\ref{flow_chart}. 

 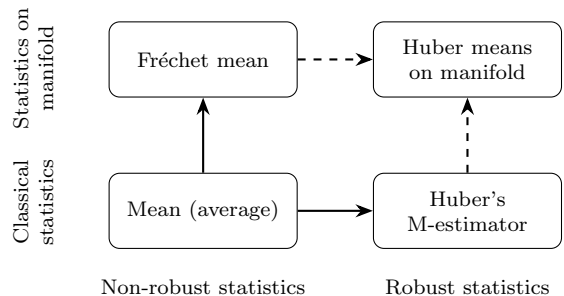
\begin{figure}[t]
     \centering
     \begin{tikzpicture}[
         node distance=1.4cm,
         nonrobust/.style={rectangle, draw, rounded corners, minimum width=2.5cm, minimum height=1cm, align=center},
         robust/.style={rectangle, draw, rounded corners, minimum width=2.5cm, minimum height=1cm, align=center},
         arrow/.style={-Stealth, thick},
         dashedarrow/.style={-Stealth, dashed, thick, black},
         dotarrow/.style={-Stealth, dashed, thick, black}
     ]

     % Nodes
     \node[nonrobust] (classical) {Mean (average)};
     \node[nonrobust, above=of classical] (manifold1) { Fr\'{e}chet mean };
     \node[robust, right=of manifold1] (manifold2) {Huber means \\ on manifold};
     \node[robust, right=of classical] (robust) {Huber's\\ M-estimator};

     % Arrows
     \draw[arrow] (classical) -- (robust);
     \draw[arrow] (classical) -- (manifold1);
     \draw[dashedarrow] (manifold1) -- (manifold2);
     \draw[dotarrow] (robust) -- (manifold2);

     % Labels
     \node[align=center, left=of manifold1, rotate=90, xshift=1.4cm] {Statistics on\\ manifold};
     \node[align=center, left=of classical, rotate=90, xshift=0.8cm] {Classical\\ statistics};
     \node[below=0.3cm of classical] {Non-robust statistics};
     \node[below=0.3cm of robust] {Robust statistics};

     \end{tikzpicture}
     \caption{A generalization of Huber means for data on Riemannian manifolds.}
     \label{flow_chart}
 \end{figure}

Our work comprehensively investigates the mathematical, statistical, and computational perspectives of Huber means for data on Riemannian manifolds. We begin this journey by verifying the conditions under which the Huber mean exists and is unique. While we believe these conditions are minimal and, in some cases, necessary, our requirement for ensuring a unique Huber mean involves bounding the support of the underlying distribution ($P_{X}$) for manifolds with positive sectional curvature. Conversely, for manifolds with nonpositive sectional curvature, we do not restrict on the support of $P_X$, which can be deemed as an improvement over the bounded support condition commonly used in the literature for the geometric median \citep{fletcher2009geometric, yang2010riemannian, afsari2011riemannian}.

The sample Huber mean is a natural estimator of the population Huber mean (see Definition~\ref{def:Huber}). We show that the population Huber means for any $c$ equal the population \Frechet\ mean $\mu \in M$ when $M$ is a Riemannian symmetric space and the distribution $P_X$ is geodesically symmetric about $\mu$. Under the same regularity conditions, the sample Huber mean is shown to be an \textit{unbiased} estimator of $\mu$. We also verify that the sample Huber mean set is a \textit{consistent} estimator of the population Huber mean set, even in cases where the Huber means are not unique. Furthermore, with appropriate regularity conditions, the \textit{central limit theorem} of the sample Huber mean is established.

We further propose a novel estimation strategy for the limiting covariance matrix of the sample Huber mean, which is crucial for constructing robust tests for location parameters. The limiting covariance matrix  $A_c = H_c^{-1} \Sigma_c H_c^{-1}$ is a combination of the expected Hessian matrix $H_c$ and the covariance matrix $\Sigma_c$ of the gradient of the Huber loss. Due to the challenging nature of explicitly writing out the Hessian matrix (see, for example, \cite{bhattacharya2008statistics, pennec2018barycentric,pennec2019curvature}), there has been scarce work on estimating $H_c$ even for the \Frechet\ mean case. Instead of directly estimating $H_c$, our approach involves evaluating a finite set of second-directional derivatives of the Huber loss function, from which the corresponding Hessian matrix can be constructed. After verifying the consistency of the proposed estimator of $A_c$, we utilize both the asymptotic normality and the estimator of $A_c$ to construct a robust one-sample location test procedure. The proposed test controls the type I error rate, and possesses an asymptotic power of 1 even under contiguous alternatives shrinking towards the null hypothesis.  

The robust nature of Huber means is further highlighted by their optimal breakdown point of 0.5, the highest possible among isometric-equivariant estimators on unbounded homogeneous spaces. This robustness is demonstrated through numerical examples on spheres and the space of symmetric positive-definite matrices, showing that Huber means remain efficient and reliable even in the presence of outliers or under heavy-tailed distribution. 

The efficiency of Huber means is compared with that of the \Frechet\ mean, under various situations. We demonstrate that the Huber means are more efficient than the \Frechet\ means when $P_X$ is heavy-tailed. Furthermore, it is verified that the robustification parameter $c$ can be chosen to achieve at least 95\% of relative efficiency over the \Frechet\ mean under Gaussian-type distributions. A practical guideline for choosing the value of $c$, ensuring 95\% relative efficiency, as used in \cite{holland1977robust}, is briefly discussed.

For Euclidean data, the (\Frechet) mean is simply the average of observations and thus is computationally cheap. For data on manifolds, however, the \Frechet\ mean often needs to be evaluated by iterative algorithms, thus the computational merit of the \Frechet\ mean diminishes. We develop a numerical recipe to compute the sample Huber and pseudo-Huber means, and show that the latter has a linear rate of convergence. As a comparison, both the subgradient algorithm of \cite{yang2010riemannian} and the Weiszfeld algorithm \citep{fletcher2009geometric}, commonly used for computing geometric median, are known to have sublinear convergence rates \citep{yang2010riemannian}.

A real dataset consisting of symmetric positive-definite matrices, obtained from a multivariate tensor-based morphometry study \citep{paquette2017ventricular}, is used to further demonstrate the robustness and efficiency of Huber means. In particular, the asymptotic normality of Huber means and approximate 95\% confidence regions for population location parameters are validated against resampling-based alternatives.

\subsection{Related work}

The Huber means, defined for a robustification parameter $c \in [0,\infty]$, include the celebrated \Frechet\ mean and the geometric median, which are sometimes referred to as the Riemannian $L_p$ center of mass \citep{afsari2011riemannian} (for $p = 2$ and $1$, respectively). See \cite{bhattacharya2005large, huckemann2011intrinsic, afsari2011riemannian, eltzner2019smeary, huckemann2021data,fletcher2009geometric, arnaudon2012medians} for statistical aspects of the \Frechet\ mean and the geometric median on non-Euclidean spaces. There is a vast collection of research endeavors on the statistical properties of the \Frechet\ mean and its generalizations. For the \Frechet\ mean on manifolds, its theoretical properties including the existence, uniqueness, consistency, asymptotic normality, and surprising peculiarities such as smeariness have been extensively studied; see e.g., \cite{huckemann2021data, eltzner2019smeary}. Some of our theoretical developments are rooted in these previous findings, building upon the established results and technicalities as used in \cite{huckemann2011inference, afsari2011riemannian, jung2025averaging}. 
The Huber mean is a special case of a \textit{generalized} \Frechet\ mean, which extends the classical \Frechet\ mean to various settings such as a largely unrestricted family of loss functions, and different spaces for the estimator and the sample. While the properties of the generalized \Frechet\ means under a general form of loss functions have been studied by \cite{huckemann2011inference, schotz2019convergence, schotz2022strong, park2023strong,evans2024limit}, we focus on the specific case of Huber loss. By doing so, we can leverage the advantages of the Huber loss function to derive a robust and efficient estimator.

% The Huber mean is a special case of a generalized \Frechet\ mean, a term coined by Huckemann \citep{huckemann2011intrinsic}. {\color{red}The generalization refers to two aspects: (i) extensions of the $L_{2}$ loss function, and (ii) the separation of sample and descriptor spaces.} %The generalized \Frechet\ mean and its theoretical properties, developed in \cite{huckemann2011inference, schotz2019convergence}, apply to various settings such as a largely unrestricted family of loss functions, and different spaces for the estimator and the sample. [JM: I think it is repetitive to the above.]
% In our work, however, we focus on a specific type of loss function defined for data on manifolds. By doing so, we can leverage the advantages of the Huber loss function to derive a robust and efficient estimator.
%

% Beyond the manifolds, statistical methodologies and related theories have been also developed actively on the (stratified) spaces and metric spaces %\citep{petersen2019frechet}. 
% Although Huber means are well-defined for metric spaces, we utilize the inherent structure of Riemannian manifolds to develop estimation procedures and theoretical properties. 

This work can be naturally extended to incorporate other choices of robustness-inducing loss functions, or general robust $M$-estimators \citep{huber2004robust}. While we focus on the Huber loss function, potential extensions and applications are further discussed in Section~\ref{sec:discuss}. For this work, we have adopted the criteria for robustness and efficiency from the literature on robust statistics, such as \cite{huber2004robust}. In classical robust statistics, Huber means on Euclidean spaces (or Huber's $M$-estimators) have originated from the fundamental problem of location and scale estimation \citep{huber1964robust}, but have found significant applications in robust regression problems \citep{fan2017estimation, jiang2019robust, zhou2018new}. In this respect, our work paves the way for comprehensive studies on Huber regression for manifold-valued response variables.

\section{Location estimation under Huber loss}\label{sec2}
\subsection{Setup and notations}\label{sec:setup}

Throughout this article, unless otherwise stated, $M$ is a smooth manifold without boundary, endowed with a Riemannian metric. The Riemannian metric naturally induces the Riemannian distance $d: M \times M \rightarrow [0, \infty)$, thereby making $(M, d)$ a metric space. We assume that $M$ is geodesically complete, connected, and separable. %and smooth Riemannian manifold without boundary, and is endowed with a Riemannian metric. The Riemannian metric naturally induces the Riemannian distance $d: M \times M \rightarrow [0, \infty)$. 

For each point $p \in M$, a distinct vector space called the tangent space $T_{p}M$ is assigned, and its elements are called tangent vectors. The Riemannian norm of $\mathbf{v} \in T_pM$, denoted by $\| \mathbf{v} \|_{p}$, is induced by the Riemannian metric at $p$. By the fundamental theorem of Riemannian geometry, there is a unique torsion-free connection compatible with the Riemannian metric \citep[Theorem 5.4,][]{lee2006riemannian}, known as the Levi-Civita connection, denoted  by $\nabla$. This connection generalizes the directional derivative in Euclidean spaces. A smooth curve $\gamma$ on $M$ is called a \emph{geodesic} if $\nabla_{\gamma'} \gamma'= \textbf{0}$. Given a fixed point $p \in M$ and a tangent vector $\mathbf{v} \in T_{p}M$, there exists a unique geodesic $\gamma_{p, \mathbf{v}}$ starting from $p$ in the direction of $\mathbf{v}$ with constant speed. The Riemannian exponential map at $p$, $\mbox{Exp}_{p}:T_{p}M \to M$, is then defined by $\mbox{Exp}_{p}(\mathbf{v})=\gamma_{p, \mathbf{v}}(1)$. 
By the Hopf-Rinow theorem (see, e.g., Theorem 2.8 of \cite{do1992riemannian}), 
since $M$ is assumed to be a complete metric space, the exponential map at $p$ is well-defined on the entire $T_{p}M$. Moreover, for any $q\in M$ there is at least one geodesic joining $p, q \in M$. The geodesic is called a minimal geodesic between these points when its length equals $d(p, q)$. The tangential cut locus at $p$, denoted by $C_p$, is defined as the set of tangent vectors in $T_{p}M$, where their image of $\mbox{Exp}_{p}$ no longer gives unique minimal geodesics on $M$. 
The exponential map at $p$, $\mbox{Exp}_{p}: T_{p}M \to M$, may not be injective. 
The injectivity domain at $p$, denoted by $D_p$, refers to the largest region around $\mathbf{0} \in T_{p}M$ where $\mbox{Exp}_{p}$ remains an injective function. The cut locus of $p$, $\mbox{Cut}(p)$, is defined as the closure of $\mbox{Exp}_{p}(C_p)$. 
%It implies that for any $q \in \mbox{Cut}(p)$ there are at least two minimal geodesics between these two points. 
%
%To define the inverse function of $\mbox{Exp}_{p}$ appropriately, we consider the restricted exponential map at $p$, $\mbox{Exp}_{p}: D_p \to M \setminus \mbox{Cut}(p)$, which is a diffeomorphism, and its inverse function
By restricting the domain of $\mbox{Exp}_{p}$ to $D_p$, the Riemannian logarithmic map at $p$ is defined as the inverse of the restricted exponential map; that is, $\mbox{Log}_{p}: M\setminus \mbox{Cut}(p) \to D_p$. The whole image of $\mbox{Log}_{p}$ equals the injectivity domain at $p$. The injectivity radius of $M$ is defined as $r_{\mbox{\tiny inj}}(M) = \inf_{p \in M} d(p, \mbox{Cut}(p))$. 
We denote by $\Delta$ (and $\delta$) the least upper bound (and the greatest lower bound, respectively) of sectional curvatures of $M$, and assume that $ -\infty < \delta \le \Delta < \infty$.
The analog of the Lebesgue measure to Riemannian manifolds is said to be the Riemannian measure (also known as the volume measure), denoted by $V$. Fortunately, it is well known that $\mbox{Cut}(p)$ is a Riemannian measure zero set. 
For further details on Riemannian geometry, we refer the reader to \cite{do1992riemannian, chavel2006riemannian,petersen2006riemannian,lee2006riemannian}.

 Let $P$ be the probability measure defined on a sample space, $X$ be an $M$-valued random variable defined on the sample space, and $P_{X}$ be the distribution (law) of $X$, i.e., $P_{X}(B)=P(X \in B)$ for all $B \in \mathcal{B}(M)$, where $\mathcal{B}(M)$ is the Borel $\sigma$-field of $M$ generated by the distance $d$. For any $x \in M$, $\delta_{x}$ stands for the Dirac measure at $x$. Given an $n$-tuple of deterministic points $(x_{1}, x_{2}, ..., x_{n}) \in M^{n}$, the empirical distribution is denoted by  $P_{n} = \frac{1}{n}\sum_{i=1}^{n} \delta_{x_{i}}$.

\subsection{Huber means on Riemannian manifolds} \label{subsec:Huber}

% The  Fr\'{e}chet mean of a distribution $P_X$ %on a metric space $(M,d)$ 
% is defined as follows.
% \begin{definition}
% The \textit{population Fr\'{e}chet mean set} with respect to $P_{X}$ is
% \[
% \mbox{argmin}_{m \in M} \int d^2(X, m) dP.
% \]
% Moreover, given $n$ deterministic observations $(x_{1}, x_{2}, ..., x_{n}) \in M^n$, the \textit{sample Fr\'{e}chet mean set} is
% \[
% \mbox{argmin}_{m \in M} \frac{1}{n} \sum_{i=1}^{n} d^2(x_{i}, m).
% \]
% \end{definition}

% The definition and properties of the Fr\'{e}chet mean have been crucial for statistics on manifolds, and more generally for non-Euclidean statistics  \citep{%petersen2019frechet, 
% huckemann2021data}. It is well-known that the Fr\'{e}chet mean possesses important statistical properties such as the consistency of the sample Fr\'{e}chet means with respect to their population counterparts %\citep{ziezold1977expected} 
% and the asymptotic normality of the sample Fr\'{e}chet means under suitable regularity conditions \citep{bhattacharya2005large}.
%Generalizing the $L_{2}$ loss function $d^2(\cdot, \cdot)$ used for the Fr\'{e}chet means, 
To define a robust descriptor of the central location of a distribution, we consider a distance-based loss function of the form $\rho_{c}\{d(\cdot, \cdot)\}$, where $\rho_{c}(\cdot)$ is the \textit{Huber loss} function with a cutoff parameter $c>0$.
The Huber loss function, introduced by \cite{huber1964robust}, combines elements of both $L_{2}$ and $L_{1}$ losses. For a cutoff parameter $c > 0$, the Huber loss function is defined for $x \ge 0$ as follows:
\[
\rho_{c}(x) = \begin{cases} x^2 & \mbox{if} ~ x \le c, \\ 
2c(x- \tfrac{c}{2}) &  \mbox{if} ~ x > c.
\end{cases}
\]
When $c \simeq 0$, the Huber loss closely resembles $2c$ times the $L_{1}$ loss, since $\rho_{c}(x) = 2cx - c^2 \simeq 2cx$ for small $c$. As $c \rightarrow \infty$, the Huber loss converges pointwise to the $L_{2}$ loss.

The Huber loss is continuously differentiable %($C^{1}$)
but is not twice differentiable at $x = c$, 
as illustrated in the right panel of Figure~\ref{comparison_loss}. 
This lack of smoothness often presents challenges in theoretical developments and calculations related to the Huber mean (see Definition~\ref{def:Huber} below). These issues motivate the adoption of \textit{pseudo-Huber loss} \citep{hartley2003multiple}, defined for $0 \le x < \infty$ as 
$\tilde{\rho}_{c}(x) = 2c^2\{\sqrt{1+(x/c)^2}-1 \}.$
This function serves as a smooth %($C^{\infty}$) 
approximation of the Huber loss function. Similar to the Huber loss function, the pseudo-Huber loss converges pointwise to the $L_{2}$ loss as $c \rightarrow \infty$, and closely approximates $2cx$ for $c \simeq 0$. The graphs of these loss functions are illustrated in Figure~\ref{comparison_loss}.

\begin{figure}[t!]
\center
\includegraphics[width =0.8\linewidth]{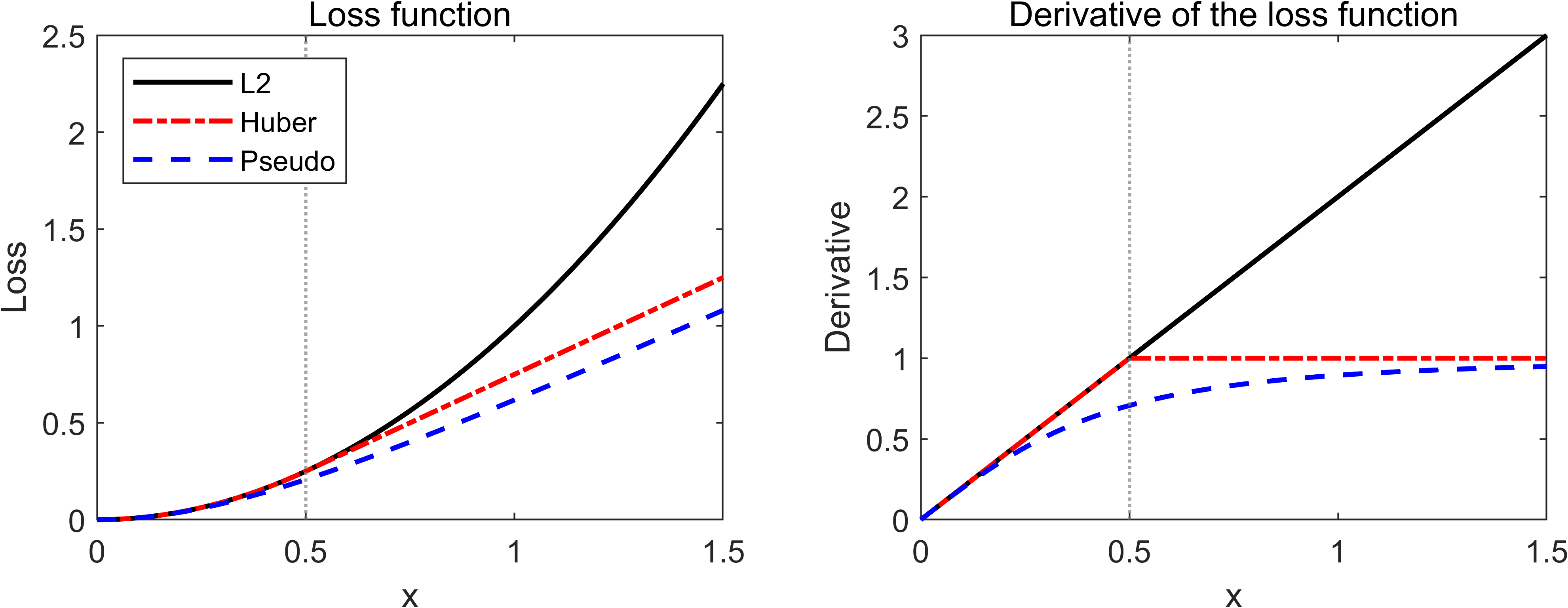} 
\caption{The $L_{2}$, Huber and pseudo-Huber loss functions for $c = 0.5$ and their derivatives.}
\label{comparison_loss}
\end{figure}

For a comprehensive study, we extend the definitions of (pseudo) Huber loss function to the limiting cases $c =0$ and $c = \infty$ by setting  
$\rho_{0}(x) = \tilde{\rho}_{0}(x) := L_{1}(x) = x$ and $\rho_{\infty}(x) = \tilde{\rho}_{\infty}(x) := L_{2}(x) = x^2$. 
The population (or sample) Huber mean is defined as any minimizer of the expected Huber loss with respect to $P_{X}$ (or $P_{n}$, respectively): 
\begin{definition}\label{def:Huber}
Given a prespecified constant $c \in [0, \infty]$, the \textit{population Huber mean set} with respect to $P_{X}$ is
\[
E^{(c)}:=\mbox{argmin}_{m \in M} F^{c}(m), ~ F^{c}(m) := \int \rho_{c}\{d(X, m)\} dP. 
\]
For given $n$ deterministic observations $(x_{1}, x_{2}, ..., x_{n}) \in M^{n}$, the \textit{sample Huber mean set} is %given by
\begin{equation}\label{eq:Huberfunctional}
E^{(c)}_{n}:=\mbox{argmin}_{m \in M} F^{c}_{n}(m), ~ F^{c}_{n}(m) := \frac{1}{n} \sum_{i=1}^{n} \rho_{c}\{d(x_{i}, m)\}.
\end{equation}
%where $E^{(c)}_{n}$ is often denoted by $E^{(c)}_{n}(x_{1}, x_{2}, ..., x_{n})$ for clarity.
\end{definition}

 The pseudo-Huber mean sets are defined analogously to Definition~\ref{def:Huber}, with $\rho_c$ substituted by $\tilde{\rho}_{c}$, and are denoted by $\tilde{E}^{(c)}$ and $\tilde{E}_n^{(c)}$. Note that the definition specifies the set of Huber means rather than a singular Huber mean, as there are instances where more than one minimizer of the expected Huber loss exists.

While the Huber mean for $0<c<\infty$ corresponds to the usual definition of Huber means, the Huber mean for $c = 0$ coincides with the geometric median. For the choice of $c = \infty$, the Huber mean equals the  \Frechet\ mean.  We have extended the definition of Huber means to encompass the two extremes to facilitate a comprehensive discussion on the properties of Huber means that are dependent on the choice of $c$. For example, it is demonstrated that the Huber means with varying $0 < c < \infty$ ``continuously" bridge the Fr\'{e}chet mean and geometric median (see Section A.2 for the precise meaning), and that the breakdown point of Huber means for $0 \le c < \infty$ is 0.5 (see Theorem~\ref{thm:robust:breakdown}). Furthermore, Huber means are maximum likelihood estimators of a location parameter under the class of certain distributions including the isotropic Gaussian-type and Laplace-type ones (see Section A.1 of the supplementary material for a proof).

We note that the Huber mean is defined for \textit{any} metric space $(M, d)$, not limited to Riemannian manifolds. However, for a thorough discussion on the theoretical, statistical, and computational properties of the Huber mean, it is beneficial to focus our examination specifically on Huber means defined on Riemannian manifolds. In the subsequent subsection, we introduce a gradient descent algorithm for computing sample Huber means, utilizing the Riemannian logarithmic map. The theoretical and statistical properties of the Huber mean are thoroughly investigated in Section \ref{sec:thms}.

\subsection{An algorithm for computing sample Huber means}

In this subsection, we present a gradient descent algorithm for computing a sample Huber mean for the cases $c \in (0,\infty)$. For the special case where $c = 0$, the corresponding Huber mean is the (sample) geometric median on Riemannian manifolds, and can be computed either by the Weiszfeld algorithm of \cite{fletcher2009geometric} or by the subgradient descent algorithm of \cite{yang2010riemannian}. For $c = \infty$, the corresponding Huber mean or, equivalently, the Fr\'{e}chet mean can be computed by a gradient descent algorithm \citep{ferreira2019gradient}. 
 
Let $c \in (0,\infty)$ be fixed. 
For a given set of data $(x_{1}, x_{2}, ..., x_{n}) \in M^{n}$, the sample Huber mean minimizes the objective function $F_n^c$ (see (\ref{eq:Huberfunctional})), which can be rewritten as
\begin{align}
    \label{eq:HuberObjective}
    F^{c}_{n}(m) = \frac{1}{n}\sum_{i=1}^{n}[d^2(m, x_i)\cdot 1_{d(m, x_i) \le c}+ \{2c \cdot d(m, x_i)-c^2\}\cdot 1_{d(m, x_i) > c}],
\end{align}
where ``$1_{\cdot}$'' stands for the indicator function. Then, the estimating equation is given by $\textsf{grad}F^{c}_{n}(m) = \frac{1}{n}\sum_{i=1}^{n} \textsf{grad}\rho_{c}\{d(m, x_{i})\} = \mathbf{0} \in T_{m} M$, where $\textsf{grad}$ denotes the Riemannian gradient evaluated at $m \in M$. %As detailed in \cite{pennec2018barycentric}, 
The gradients for the squared distance and the distance functions, necessary for this computation, are 
\begin{eqnarray}\label{grad}
\textsf{grad}d^2(m, x) &=& -2 \mbox{Log}_{m}(x) \quad \mbox{for any} ~ x \notin  \mbox{Cut}(m), ~ \mbox{and} \nonumber \\
\textsf{grad}d(m, x) &=& -\frac{\mbox{Log}_{m}(x)}{\|\mbox{Log}_{m}(x)\|_{m}} \quad \mbox{for any} ~ x \notin \mbox{Cut}(m) \cup \{m\},
\end{eqnarray}
where $\| \cdot \|_{m}$ denotes the Riemannian norm on the tangent space $T_mM$. (With an appropriate choice of the coordinate system on $T_mM$, $\| \cdot \|_{m}$ is equivalent to the Euclidean norm.)
Combining (\ref{eq:HuberObjective}) and (\ref{grad}), the negative gradient of $F^{c}_{n}(\cdot)$ at $m$ is given by 
\begin{equation}\label{Huber_grad}
-\textsf{grad}F^{c}_{n}(m) = \frac{2}{n}\sum_{i=1}^n \big[\mbox{Log}_{m}(x_i) \cdot 1_{d(m, x_i)\le c} + \frac{c \cdot \mbox{Log}_{m}(x_i)}{\| \mbox{Log}_{m}(x_i)\|_{m}} \cdot 1_{d(m, x_i) > c}\big] \in T_{m}M,
\end{equation}
under the assumption that $x_i \notin \mbox{Cut}(m)$ for all $1\le i\le n$. 
From a computational perspective, data points lying on $\mbox{Cut}(m)$ are excluded from the gradient calculation in  (\ref{Huber_grad}) and (\ref{pseudo_grad}). This exclusion is incorporated into the algorithm used for computing Huber means. 

The Riemannian gradient descent algorithm for minimizing $F_n^c$ begins with an initial candidate $m_0$ for the Huber mean, and iteratively updates the candidate along the direction of negative gradient. Since the gradient is a tangent vector, the exponential map is used to update the candidate. Our procedure for computing Huber means is summarized in Algorithm~\ref{alg:Huber}. 

%Note that the minimizer of $F^{c}_{n}$ attains the zero gradient vector. Based on (\ref{Huber_grad}), we can obtain the algorithm for locating the Huber mean for $c \in (0, \infty)$. 
%

For the case where the pseudo-Huber loss is used, the objective function evaluated at $m \in M$ is given by $$\tilde{F}^{c}_{n}(m)= \frac{1}{n}\sum_{i=1}^{n} \tilde{\rho}_{c}\{d(x_i, m)\}= \frac{2c^2}{n}\sum_{i=1}^{n} (\sqrt{1+d^2(x_i, m)/c^2}-1),$$ and the negative gradient of $\tilde{F}^{c}_{n}(\cdot)$ evaluated at $m$ is, for $m \notin \cup_{i=1}^n\mbox{Cut}(x_i)$,
%In the same way, we can find the algorithm based on the pseudo-Huber loss. Let $\tilde{F}^{c}_{n}(m)= \frac{1}{n}\sum_{i=1}^{n} \tilde{\rho}_{c}\{d(x_i, m)\}= \frac{2c^2}{n}\sum_{i=1}^{n} (\sqrt{1+d^2(x_i, m)/c^2}-1)$. The gradient of $-\tilde{F}^{c}_{n}(\cdot)$ at $m \in M$ is given by
\begin{eqnarray}\label{pseudo_grad}
-\textsf{grad} \tilde{F}^{c}_{n}(m) = \frac{2c^2}{n}\sum_{i=1}^{n}\frac{\mbox{Log}_{m}(x_i)}{\sqrt{1 + (\|\mbox{Log}_{m}(x_i)\|_{m}/c)^2}} \in T_{m}M. 
\end{eqnarray}

\begin{algorithm}[t]
  \caption{Computation for (pseudo) Huber mean for $c \in (0, \infty)$ %on manifolds
  }
  \label{alg:Huber}
  \begin{algorithmic}
  \State \textbf{Input:}  Data $(x_{1},\, x_{2},\, \ldots,\, x_{n}) \in M^{n}$, an initial point $m_0$, a step size $\alpha > 0$, and a threshold $>0$.

  \While{ ($\Delta m \ge \mbox{threshold}$) } 
    %\State \textcolor{red}{In each iteration, data points that lie within the cut locus of $m_k$
    %are excluded from the data solely during the current iteration.}
    \State $\Delta m = \frac{2}{n}\sum_{\{1 \le i \le n: ~ x_i \notin \mbox{\tiny Cut}(m_k)\}} \big[\mathrm{\text{Log}}_{m_k}(x_i) \cdot 1_{d(m_{k}, x_i) \le c}
    + \frac{c \cdot \text{Log}_{m_{k}}(x_i)}{\|\text{Log}_{m_{k}}(x_i)\|_{m_k}} \cdot 1_{d(m_{k}, x_i) > c}\big]$  \\
    (or $\Delta m = \frac{2c^2}{n}\sum_{\{1 \le i \le n: ~ x_i \notin \mbox{\tiny Cut}(m_k)\}} \mbox{Log}_{m_{k}}(x_i)/\sqrt{1+(\|\mbox{Log}_{m_{k}}(x_i)\|_{m_k}/c)^2}$ if the pseudo-Huber loss is used)
    \State  $m_{k+1} = \mathrm{\text{Exp}}_{m_{k}}(\alpha \Delta m)$%, ~ $k \leftarrow k+1$ \\
    \EndWhile  
  \end{algorithmic}
\end{algorithm}

A pertinent question is whether the algorithm described in Algorithm~\ref{alg:Huber} converges. The conditions for ensuring the convergence of Algorithm~\ref{alg:Huber} include limiting the data diameter in terms of the sectional curvatures of $M$ and carefully determining the step size $\alpha$. In Section A.3, we demonstrate that, under appropriate conditions, the iterate $m_k$ in the algorithm converges linearly to the sample pseudo-Huber mean when the pseudo-Huber loss is employed. Although it has been technically difficult to verify the convergence of Algorithm~\ref{alg:Huber} when the Huber loss is used, in all of our numerical experiments in Section~\ref{sec:examples}, the algorithm has converged in a few steps.

\section{Theoretical properties of Huber means}\label{sec:thms}
In this section, we scrutinize the mathematical, statistical, and computational properties of Huber means. Most results are stated only for Huber means with the Huber loss $\rho_{c}$, but the findings are valid when $\rho_{c}$ is replaced by $\tilde{\rho}_{c}$. 
The exception is Theorem~\ref{thm:robust:breakdown} for which we clearly state which loss function is applicable. Proofs of all findings, technical details, and auxiliary lemmas are given in Appendix.

%The only exceptions are \Cref{prop:alg:conv} and \Cref{thm:robust:breakdown} for which we clearly state which loss function is applicable. Proofs of all findings, technical details, and auxiliary lemmas are given in Section B of the supplementary material.

\subsection{Existence and uniqueness}\label{sec:exist:unique}

The conditions for the uniqueness and existence of the Riemannian $L_{p}$ center of mass, including the Fr\'{e}chet mean and geometric median, are thoroughly provided in \cite{afsari2011riemannian}; see also \cite{sturm2003probability,huckemann2021data} for Fr\'{e}chet means. We adopt the arguments used in \cite{afsari2011riemannian} to extend these results to the Huber means for $c \in (0,\infty)$.

To ensure the existence of Huber means, we impose an integrability condition: 

\indent \underline{(A1)} For some $m \in M$, $F^c(m) = \int \rho_{c}\{d(X, m)\}dP < \infty$. 

It can be shown that, for any $c \in [0, \infty]$, Assumption (A1) is equivalent to the condition that $\int \rho_{c}\{d(X, m)\}dP < \infty$ for \textit{any} $m \in M$.
If $P_X$ has a finite variance, i.e., $\int d^2(X, m)dP < \infty$ for an $m\in M$, then (A1) holds for any $c \in [0, \infty]$. If the pseudo-Huber mean is of interest, (A1) is interpreted as $\tilde{F}^c(m) = \int \tilde{\rho}_{c}\{d(X, m)\}dP < \infty$ for some $m \in M$. For a more relaxed condition on $P_{X}$, Assumption (A1) can be replaced by $\int [\rho_{c}\{d(X, m)\} - \rho_c\{d(X, m_0)\}]dP < \infty$ for some $m, m_0\in M$. This enables us to define the Huber mean when $P_{X}$ is highly heavy-tailed. In this case, $E^{(c)}$ in Definition~\ref{def:Huber} should be defined as $\mbox{argmin}_{m \in M} \int [\rho_c\{d(X, m)\} - \rho_c\{d(X, m_0)\}]dP$, as in the trick used in \cite{sturm2003probability}. For ease of discussion, however, we leave it in the form of (A1).

 In the next result, we show and use the facts that for any choice of $c \in [0,\infty]$ the objective function $F^{c}$ is coercive, i.e., has its minimum on a compact set, and that $F^c$ is continuous  to ensure that the Huber mean set is non-empty.

\begin{theorem}[Existence of the population Huber means]\label{thm:exist}
For a given $c \in [0, \infty]$, assume that $P_{X}$ satisfies Assumption (A1). Then, the population Huber mean exists, i.e., $E^{(c)} \neq \phi$.
\end{theorem}

For any data set with finite sample size $n$, the sample (pseudo) Huber mean exists for any sample size $n$ and $c \in [0,\infty]$; that is, $E_{n}^{(c)} \neq \phi$. This is observed by substituting $P_{X}$ in Theorem~\ref{thm:exist} with $P_{n}$.

It is well known that the geometric median, which minimizes the sum of absolute deviations, may not be unique in Euclidean spaces. Similarly, the Huber means for $c < \infty$  on Euclidean spaces are not necessarily unique. Additionally,  the Fr\'{e}chet mean may lack uniqueness for general manifolds with non-zero curvature. These complexities raise an important question: Under what conditions can the uniqueness of the Huber mean be guaranteed?

To address this question, we would ideally hope for the convexity of the objective function  $F^{c}$ on $M$. Unfortunately, it is known that any non-constant continuous function on a compact manifold with no boundary, e.g., the unit sphere $S^k$, cannot be convex \citep{yau1974non}. Our strategy is to ensure that the Huber mean is located in a strongly convex subset of $M$, and to show that $F^{c}$ is strictly convex on the region. As previously done in \cite{yang2010riemannian, afsari2011riemannian}, we impose a condition on the support of $P_{X}$ to ensure both strong convexity of the domain and the strict convexity of $F^c$.

Recall that $\Delta ~ (< \infty)$ stands for the least upper bound of the sectional curvatures of $M$, and $r_{\mbox{\tiny inj}}(M)$ for the injectivity radius of $M$.
Throughout, we use the convention that ${1}/{\sqrt{\Delta}}$ is treated as $\infty$ in the case of $\Delta \le 0$. 

\indent \underline{(A2)} For the prespecified $c \in [0, \infty]$, there exists $p_{0} \in M$ such that 
$\textsf{supp}(P_{X}) \subseteq B_{r_0}(p_0)$, where $\textsf{supp}(P_{X})$ denotes the support of $P_{X}$,  and 
\begin{equation}\label{eq:r_0}
    r_0 = \begin{cases}
    \frac{1}{2}\min\{ \frac{\pi}{2\sqrt{\Delta}}, r_{\mbox{\tiny inj}}(M)  \} & \mbox{if} ~  0 \le c < \frac{\pi}{\sqrt{\Delta}},\\
     \frac{1}{2}\min\{\frac{\pi}{\sqrt{\Delta}}, r_{\mbox{\tiny inj}}(M) \} & \mbox{if} ~  \frac{\pi}{\sqrt{\Delta}} \le c \le \infty. \\
\end{cases}
\end{equation} 

%It is important to note that for the case where $r_0 = \infty$, the condition (A2) poses no restriction on the support of $P_X$. 
We mention that if $c \ge \pi / \sqrt{\Delta}$ and (A2) is satisfied, then the corresponding Huber mean set $E^{(c)}$ is exactly the Fr\'{e}chet mean set: $E^{(c)} = E^{(\infty)}$ for all $\pi / \sqrt{\Delta} \le c \le \infty$. 
For cases where $0 \le c < \pi/\sqrt{\Delta}$, we require the radius $r_0$ to be less than or equal to $\pi/(4\sqrt{\Delta})$, rather than $\pi/(2\sqrt{\Delta})$, which was used in \cite{afsari2011riemannian}. In these cases, the Huber loss function is still convex, but is not strictly convex (due to the presence of distance function $x \mapsto d(x, m)$ in the loss). The smaller upper bound $\pi/(4\sqrt{\Delta})$ of the radius $r_0$ is imposed to ensure strict convexity of the objective function $F^c$. When the pseudo-Huber loss function is used, $r_0$ in Assumption (A2) can be replaced by $r_0 = \tfrac{1}{2}\min\{\tfrac{\pi}{\sqrt{\Delta}}, r_{\mbox{\tiny inj}}(M) \}$, since %the pseudo-Huber loss function 
$m \mapsto \tilde{\rho}_{c} \{d(x,m)\}$ is strictly convex on $B_{r_0}(p_0)$ for any $x \in B_{r_0}(p_0)$ and $c \in (0, \infty)$.

\begin{theorem}[Uniqueness of population Huber means]\label{thm:unique:pop} 
\begin{enumerate}
     \item[(a)]
For a prespecified constant $c \in [0,\infty]$, suppose that $P_{X}$ satisfies Assumptions (A1) and (A2). 
  Suppose further that  either of the following holds: 
          \begin{enumerate}
              \item[(i)] $P_X$ does not degenerate to any single geodesic.
              \item[(ii)] $P_X$ degenerates to a single geodesic $\gamma_X$,  and  for any $m_0 \in E^{(c)}$ it holds that $P(d(X, m_0) \le c+\epsilon ) > 0$ for all $\epsilon > 0$.
          \end{enumerate} 
Then, the population Huber mean with respect to $P_{X}$ is unique. 
     \item[(b)]
For a prespecified constant $c \in (0,\infty)$, suppose $P_{X}$ satisfies $\int_M \tilde{\rho}_{c} \{d(X,m)\} dP < \infty$ for some $m$, and that the support of $P_X$ lies in an open ball of radius $r_0 = \frac{1}{2}\min\{\frac{\pi}{\sqrt{\Delta}}, r_{\mbox{\tiny inj}}(M) \}$. Under these conditions, the population pseudo-Huber mean with respect to $P_{X}$ is unique.
\end{enumerate}
\end{theorem}  

We emphasize that, for the case where $r_0$ is infinite, Theorem~\ref{thm:unique:pop} essentially requires only the integrability of the loss function, while allowing unbounded support for the distribution.  
As a comparison, in \cite{yang2010riemannian,afsari2011riemannian}, the authors require a bounded support condition to ensure the uniqueness of the geometric median, even on manifolds with nonpositive sectional curvature. In this respect, %we believe that 
Part (a) of Theorem~\ref{thm:unique:pop} is a significant improvement over the bounded support condition for the case $c = 0$.

Our result ensures uniqueness without any restrictions on the support of $P_X$ for manifolds with nonpositive sectional curvature, including $\bbR^{k}$, the space of symmetric positive-definite matrices $\mbox{Sym}^{+}(k)$ endowed with log-Euclidean or affine-invariant Riemannian distances \citep{pennec2006riemannian}, and the hyperbolic space $\bbH^{k}$.

%The theorem imposes no restrictions on the support of $P_X$ for manifolds with nonpositive sectional curvature, such as Euclidean space $\bbR^{k}$, the space of symmetric positive-definite matrices $\mbox{Sym}^{+}(k)$ endowed with log-Euclidean or affine-invariant Riemannian distances \citep{pennec2006riemannian}, and the hyperbolic space $\bbH^{k}$.  

For manifolds with positive sectional curvature, it is noteworthy that the radius of the open ball $r_0 = \tfrac{1}{2}\min\{ \tfrac{\pi}{2\sqrt{\Delta}}, r_{\mbox{\tiny inj}}(M)  \}$ in Assumption (A2) for smaller $c \le \pi/\sqrt{\Delta}$ coincides with the previous findings on the support condition for the case $c = 0$ \citep{yang2010riemannian,afsari2011riemannian}. For the $k$-dimensional unit sphere $S^{k}$ (with the usual metric, i.e., $\Delta = 1$), the Huber mean for $c \in [0,\pi)$ is unique if $\textsf{supp}(P_{X})$ lies in an open ball of radius $\pi/4$; for $c \ge \pi$, 
the Huber mean coincides with the Fr\'{e}chet mean, which is unique if the support lies in an open $(\pi/2)$-ball. For smaller $c$, we do not believe that the radius $r_0$ in the support condition (A2) is the largest possible. Nevertheless, there are nontrivial examples where the Huber mean is not unique when the condition (A2) is violated; see Section B.3 for examples. 
On the other hand, the uniqueness of pseudo-Huber means for any $c \in (0, \infty)$ is guaranteed for any distribution supported on an open hemisphere. 

We now explain the somewhat unusual requirement (ii) of Part (a) in Theorem~\ref{thm:unique:pop}. This special case, where $P_X$ degenerates to a single geodesic, corresponds to the typical situation of scalar-valued Huber means. For this scenario, the requirement of non-negligible density or mass near any $m_0 \in E^{(c)}$ is needed to prevent situations where Huber loss function $F^c$ is \textit{flat} at any Huber mean. An illustrative example of a situation where requirement (ii) is violated is given in Figure \ref{fig:unique_c}. In this example, there exists a closed geodesic segment such that every element in the segment is Huber mean. This occurs due to the presence of $d(x,m)$ in the loss function, whose sum or integral might be flat along the geodesic segment. Requirement (ii) is imposed to prevent such situations. We note that, for $c > 0$, the condition can be simplified to $P(d(X, m_0) \le c) > 0$ for any $m_0 \in E^{(c)}$. 

\begin{figure}[t]
\centering
\includegraphics[scale=0.4]{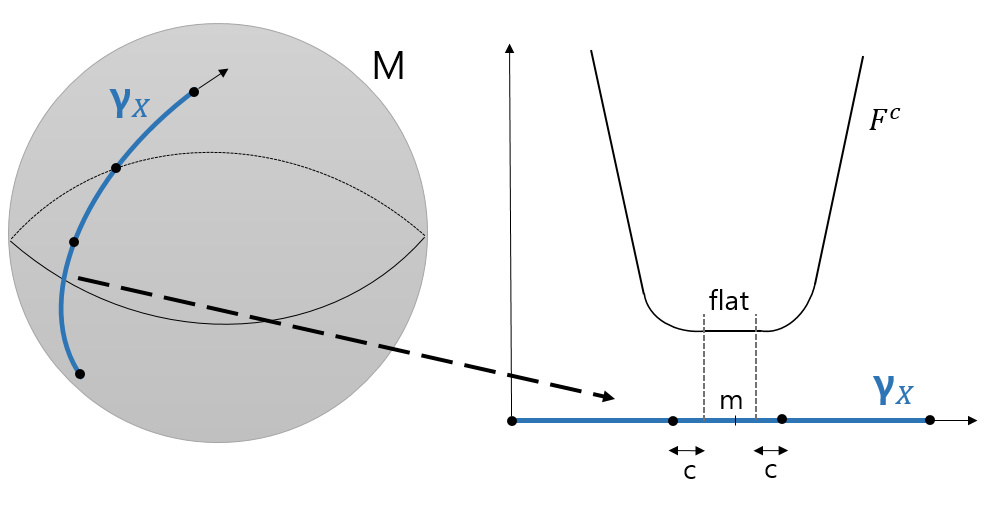}
\caption{An example of non-unique Huber means for which the requirement (ii) of Theorem~\ref{thm:unique:pop} is violated. (Left) $P_{X}$ is equally distributed on four black dots on a geodesic $\gamma_X$. (Right) The graph of $F^{c}$ parametrized by the arc length of $\gamma_{X}$. In this situation, there is a point $m \in E^{(c)} \subset \gamma_{X}$ near the center of $\gamma_{X}$ such that $P(d(X, m) \le c) = 0$. The minimum of $F^c$ occurs at every point of a closed segment within $\gamma_X$.
}
\label{fig:unique_c}
\end{figure}

%Given $n$ deterministic observations, $(x_{1}, x_{2}, ..., x_{n}) \in M^{n}$, by replacing $P_{X}$ with $P_{n}$ in \Cref{thm:unique:pop}, we can obtain sufficient conditions for the uniqueness of sample Huber means. %as follows. 

% % \begin{corollary}[Uniqueness of sample Huber means]\label{cor:unique:sam}
% % \begin{enumerate}
% % \item[(a)]
% % Let $c \in [0,\infty]$ be pre-specified. Given $n$ deterministic observations $(x_{1}, x_{2}, ..., x_{n}) \in M^{n}$, suppose that the sample points $x_{1}, x_{2}, ..., x_{n}$ lie in an open ball of radius $r_0$ as specified in (\ref{eq:r_0}), and that either of the following holds: 
% % \begin{enumerate}
% %  \item[(i)] The sample points are in general position (i.e., they do not lie on a single geodesic).
% %   \item[(ii)] The sample points lie in a single geodesic $\gamma_X$ and the sample size $n$ is an odd number. 
% %   \item[(iii)] The sample points lie in a single geodesic $\gamma_X$, $n = 2k$ for some $k \ge 1$, and for the order statistics $(x_{(1)}, x_{(2)}, ..., x_{(n)})$ of $(x_{1}, x_{2}, ..., x_{n})$ along the geodesic $\gamma_{X}$ it holds that $d(x_{(k)}, x_{(k+1)})\le 2c$. 
% % \end{enumerate} 
% % Then, the sample Huber mean is unique. 
% % \item[(b)]
% % Let $c \in (0, \infty)$ be pre-specified. Suppose that the sample points $x_{1}, x_{2}, ..., x_{n}$ lie in an open ball of radius $r_0 = \frac{1}{2}\min\{\frac{\pi}{\sqrt{\Delta}}, r_{\mbox{\tiny inj}}(M) \}$. Then, the sample pseudo-Huber mean is unique.
% % \end{enumerate}
% \end{corollary}

\subsection{Unbiasedness of sample Huber means}\label{sec:unbiased}

On Euclidean spaces $(M = \bbR^k$), the sample mean, i.e., the sample \Frechet\ mean, is an unbiased estimator of the population mean vector. For general manifolds, however, the sample \Frechet\ mean is not necessarily unbiased with respect to the population \Frechet\ mean, as demonstrated in  \cite{pennec2019curvature}. We say an estimator $\hat\theta_n$ (based on a random sample of size $n$) of $\theta \in M$ is \emph{unbiased} if the population \Frechet\ mean of $\hat\theta_n$ equals $\theta$. While  asymptotic unbiasedness of the sample \Frechet\ mean has been investigated \citep{schotz2019convergence,pennec2019curvature}, 
to our knowledge, scarce work has addressed regularity conditions under which the finite-sample unbiasedness of \Frechet\ means is exactly guaranteed. This obscurity raises a relevant question: Under what regularity conditions can the finite-sample unbiasedness of Huber means be ensured? 
 
%Using the Taylor expansion on Riemannian manifolds, in a finite-sample setting, Pennec derived the bias of the \Frechet\ mean in terms of the sample size $n$ and Riemannian curvature tensors \citep[Theorem 7,][]{pennec2019curvature}, which is concerned with the primary goal of this subsection, unbiasedness.  
%For the cases where $M$ is a Hadamard manifold (a simply connected and complete manifold with nonpositive curvature), it is shown that the sample \Frechet\ mean is asymptotically unbiased under certain regularity conditions \citep[Corollary 4,][]{schotz2019convergence}.(For the definition of asymptotic unbiasedness, refer to Section B.4.1 of the supplementary material). }  

We address this question for the case where $M$ is a \textit{Riemmannian symmetric space}. A smooth and connected Riemannian manifold is said to be a (globally) Riemannian symmetric space if, for each $\mu \in M$, there exists a unique isometry $s_\mu$, called geodesic symmetry, that reflects the manifold through $\mu$ (so that $s_\mu(\mu) = \mu$ and $s_\mu$ reverses all geodesics through $\mu$). Examples of Riemannian symmetric spaces include $\bbR^k$ with $s_\mu(\mu +m) = \mu -m$, and $S^k \subset \bbR^{k+1}$ with $s_\mu(m) = (I_{k+1} - 2 \mu \mu^T)m$. (see \cite{cornea2017regression} for more examples).%\cite{gorodski2021introduction} 

An appropriate notion of symmetric distributions on $M$ is now introduced.

\begin{definition}
The distribution $P_{X}$ on a Riemannian symmetric space $M$ is said to be \emph{geodesically symmetric} about $\mu \in M$ if the distribution of $s_\mu (X)$ is also $P_X$.
\end{definition}

To exemplify a class of geodesically symmetric distributions on $M$, we suppose $P_{X}$ is absolutely continuous with respect to $V$. Consider a class of symmetric density functions: 
\begin{equation}\label{eq:mahala:dist} 
f_{\mu, \Sigma, g}(x)= C \cdot g\{d_\Sigma(x,\mu)\},
\end{equation}
for $x \in B_{r_{\mbox{\tiny inj}}}(\mu)$,
where $r_{\mbox{\tiny inj}} ~(=r_{\mbox{\tiny inj}}(M))$ refers to the injectivity radius of $M$ (which may be infinite). 
Here, $\Sigma$ is a positive-semidefinite matrix, $d_\Sigma(x,\mu) = \mbox{Log}_{\mu}(x)^{T}\Sigma^{-1}\mbox{Log}_{\mu}(x)$ and $g:[0, \infty) \rightarrow [0, \infty)$ is any measurable function satisfying $\int_{B_{r_{\mbox{\tiny inj}}}(\mu)} g\{d_\Sigma(x,\mu)\} dV(x) \in (0, \infty)$. Any distribution with the density (\ref{eq:mahala:dist}) is geodesically symmetric about $\mu$, since $f_{\mu, \Sigma, g}(x) = f_{\mu, \Sigma, g}\{s_{\mu}(x)\}$ for any $x \in B_{r_{\mbox{\tiny inj}}}(\mu)$.

% \begin{example}
% Suppose that the distribution $P_{X}$ on a Riemannian symmetric space $M$ is absolutely continuous with respect to $V$, and it holds that $\frac{dP_{X}}{dV}=f_{\mu, \Sigma, \rho}$ as specified in (\ref{eq:mahala:dist}). Then, $P_{X}$ is geodesically symmetric about $\mu$, since it holds that $f_{\mu, \Sigma, \rho}(x) = f_{\mu, \Sigma, \rho}\{s_{\mu}(x)\}$ for any $x \in B_{r_{\mbox{\tiny inj}}}(\mu)$. 
% \end{example}

We present a set of regularity conditions for the finite-sample unbiasedness of Huber means.

\indent (C1) The following are satisfied: 
\begin{itemize}
\item[(i)] $M$ is a simply connected and complete Riemannian symmetric space.
\item[(ii)] $P_{X}$ has a finite variance, and does not degenerate to any single geodesic.
\item[(iii)] For a given $\mu \in M$, $P_{X}$ is geodesically symmetric about $\mu$, and satisfies $\textsf{supp}(P_{X}) \subseteq B_{r}(\mu)$, where $r = \frac{1}{2}\min\{\frac{\pi}{2\sqrt{\Delta}}, r_{\mbox{\tiny inj}}(M)\}$ (which may be infinite). 
\end{itemize}

If Conditions (ii) and (iii) above hold, then the Huber means are guaranteed to be unique (cf. Theorem \ref{thm:unique:pop}). 
The next result shows that, under the regularity conditions, the population Huber means are the same for any $c \in [0, \infty]$. By utilizing this, it is further shown that the sample Huber mean for any $c$ is unbiased under the same conditions.

\begin{proposition}\label{prop:unbiased}
If Condition (C1) is satisfied, then for a given $c \in [0, \infty]$ the following hold:
\begin{itemize} 
\item[(a)] The population Huber mean for $c$ is unique and equals $\mu$.  
\item[(b)] For every sample size $n \ge 1$, the sample Huber mean for $c$ is an unbiased estimator of $\mu$.
\end{itemize}
\end{proposition}

\subsection{Strong consistency and asymptotic normality for Huber means}\label{sec:asymptotics}

In this subsection, we explore important large-sample properties of sample Huber means on Riemannian manifolds in relation to their population counterparts. First, we demonstrate that the sample Huber mean set $E^{(c)}_{n}$ for $c \in [0, \infty]$ is \textit{strongly consistent} with $E^{(c)}$
%in the sense of Bhattacharya--Patrangenaru \citep{bhattacharya2005large} 
(see \cite{huckemann2011intrinsic} for the choice of terminology). Given $n$ random observations $X_{1}, X_{2}, ..., X_{n} \overset{i.i.d.}{\sim} P_{X}$, the sample Huber mean set $E^{(c)}_{n}= E^{(c)}_{n}(X_{1}, X_{2}, ..., X_{n})$ is a random closed set.

\begin{theorem}[Strong consistency]\label{thm:slln} 
For a given $c \in [0, \infty]$, if $P_{X}$ satisfies Assumption (A1), then with probability 1,
\[
\lim_{n\rightarrow \infty}\sup_{m \in E^{(c)}_{n}} d(m, E^{(c)})=0,
\]
where $d(m, E^{(c)}):=\inf_{p \in E^{(c)}} d(m, p)$. 
\end{theorem}

We next establish a central limit theorem for Huber means. 
%To this end, we 
Assume the following for a prespecified $c \in (0,\infty]$:\\
\indent \underline{(A3)} The population Huber mean for $c$ with respect to $P_{X}$ is unique, and so is the sample Huber mean with probability 1 for every sample size $n$. 

Under Assumption (A3), we write the unique population and sample Huber means by $m^{c}_{0}$ and $m^{c}_{n}$, respectively. In a local coordinate chart $(\phi_{m^{c}_{0}}, U)$ centered at $m_{0}^{c}$ (i.e., $\phi_{m_{0}^{c}}(m_{0}^{c})= \mathbf{0} \in \bbR^{k}$), we define $\Sigma_{c}(\mathbf{x}):=\textsf{Cov}[\textsf{grad}\rho_{c}\{d(X, \phi_{m_{0}^{c}}^{-1}(\mathbf{x})\}]$  and $H_{c}(\mathbf{x})= E[\textsf{Hess}\rho_{c}\{d(X, \phi_{m_{0}^{c}}^{-1}(\mathbf{x}))\}]$ where the $\textsf{grad}$ and $\textsf{Hess}$ denote the gradient and Hessian matrix (with respect to the standard basis of $\mathbb{R}^k$ induced by the chart), respectively. We denote $\Sigma_{c}:=\Sigma_{c}(\mathbf{0})$ and $H_{c}:=H_{c}(\mathbf{0})$, and impose a set of regularity conditions, referred to as Assumption (A4), on the distribution $P_X$. These conditions are stated in Section B.1. There, we also provide an example under which the conditions are satisfied. 
 %If $P_{X}$ is absolutely continuous with respect to the $k$-dimensional Hausdorff measure, and $P_{X}$ further satisfies Assumptions (A1) and (A2), then Assumption (A3) is satisfied. It is noteworthy that Assumption (A3) does not impose any bounded support condition on $P_{X}$. 
% It is generally challenging to verify the asymptotic normality of an estimator on manifolds, due to their nonlinearity. To overcome the difficulty, a ``linearization" of manifolds by utilizing local coordinate charts has been typically applied in the related literature \citep{bhattacharya2003large, bhattacharya2005large, huckemann2011inference, jung2025averaging}. 
%Recall that $m^{c}_{0}$ is the unique population Huber mean. 

Since the sample Huber mean represents the central location of data on manifolds, its asymptotic normality qualifies to be termed a central limit theorem.
\begin{theorem}[Central limit theorem]\label{thm:clt}
For a given $c \in (0, \infty]$, suppose that $P_{X}$ satisfies Assumptions (A1), (A3), and (A4). Then, \\
(a) $m^{c}_{n} \rightarrow m^{c}_{0}$ almost surely as $n \rightarrow \infty$, and \\
(b) $\sqrt{n}\phi_{m^{c}_{0}}(m^{c}_{n}) \rightarrow N_{k}(\mathbf{0}, H_{c}^{-1}\Sigma_{c} H_{c}^{-1})$ in distribution as $n \rightarrow \infty$.
\end{theorem}

%The sample pseudo-Huber mean is also consistent with the population pseudo-Huber mean, and its limiting distribution is also Gaussian. 

In the next subsections, we discuss an estimation strategy for the limiting covariance matrix $H_c^{-1}\Sigma_{c}H_c^{-1}$ that appeared in the asymptotic normal distribution for the sample Huber mean, and use it in a one-sample location test procedure based on Huber means. We hereafter identify a local neighborhood of a point $m$ via the Riemannian logarithm map $\mbox{Log}_m$, allowing computations of gradients and Hessians in the tangent space $T_mM$ equipped with the Riemannian metric.

\subsection{Limiting covariance estimation for the sample Huber mean}\label{subsec:est:cov}

The limiting covariance matrix $H_c^{-1}\Sigma_{c}H_c^{-1}$ in the central limit theorem (Theorem~\ref{thm:clt}) plays important roles in designing hypothesis testing procedures and constructing confidence regions (cf. Sections \ref{subsec:application:hypo} and \ref{subsec:TBM}), as well as in comparing the efficiencies of estimators (Section~\ref{sec:RE}). However, estimating the limiting covariance matrix for manifold-valued estimators has not been an easy task for two main reasons. First, the form of the Hessian matrix $H_c$, or the Hessian of the loss function, depends not only on the loss function but also on which manifold $M$ the estimator is defined. Therefore, explicitly writing out for $H_c$, even for the case $c  =\infty$ (i.e., for the \Frechet\ mean) is complicated, and has been done only for manifolds with simple structures \citep{pennec2018barycentric}. 
Second, the covariance matrix of the (Huber) mean $\mbox{Log}_{m_0^c}(m_n^c)$ is not simply a scaled covariance matrix of a single random variable $\mbox{Log}_{m_0^c}(X_1)$. It is so only if the Hessian matrix is proportional to identity. For example, in the case of \Frechet\ mean on the Euclidean space $M = \mathbb{R}^k$, we have $H_\infty = 2I_k$ and $\textsf{Cov}(\bar{X}_n) = \tfrac{1}{n}\textsf{Cov}(X_1)$. %However, one cannot expect such a relation (i.e., $\textsf{Cov}\{\phi_{m_0^\infty}(m_n^\infty)\} = C \cdot \textsf{Cov}\{\phi_{m_0^\infty}(X_1)\}$ for a $C>0$) to hold for general manifolds.
However, one cannot expect such a relation (i.e., $\textsf{Cov}\{\mbox{Log}_{m_0^\infty}(m_n^\infty)\} = C \cdot \textsf{Cov}\{\mbox{Log}_{m_0^\infty}(X_1)\}$ for a $C>0$) to hold for general manifolds, especially for manifolds with non-zero curvature, as observed in \cite{pennec2019curvature}.

%\textcolor{red}{Meanwhile, the phenomenon may be closely related to the curvature of manifolds. In this regard, the ratio between the empirical variance on the Fréchet means $\textsf{tr}\{\textsf{Cov}(m^{\infty}_{n})\}$ and $\textsf{tr}\{\mbox{Cov}(X_1)\}/n$ (e.g., the modulation of \Frechet\ mean is 1 in Euclidean cases), has been called modulation, a term coined by Pennec \citep{pennec2019curvature}. The author showed that the modulation of \Frechet\ mean is greater than 1 in manifolds with positive curvature and lower than 1 in manifolds with negative curvature. We believe that the modulation of Huber means on manifolds is an intriguing future topic to be investigated.}
%

%In this subsection, We propose a novel moment-based estimation procedure for $\Sigma_c$ and $H_c$ for any $c \in (0,\infty]$. To the best of our knowledge, our approach for estimating the Hessian matrix $H_c$ for both Huber means and Fr\'{e}chet means, has never been considered in the literature. For the case of $c = \infty$, our moment-based estimator of $\Sigma_c$ has been used in some previous works \citep{hotz2015intrinsic, eltzner2019smeary, jung2025averaging}, but it has been misused as an estimator for $H^{-1}_{\infty}\Sigma_{\infty}H^{-1}_{\infty}$.

While the variance for gradient of Huber loss, $\Sigma_{c}$, is naturally estimated by a moment estimator
\begin{equation}
    \label{eq:Sigma_c_est_main}
    \hat{\Sigma}_{c} := \frac{4}{n}\sum_{i=1}^n \left( \frac{\| \mbox{Log}_{m}(X_i) \| \wedge c}{\| \mbox{Log}_{m}(X_i) \|} \right)^2 \mbox{Log}_{m}(X_i) \mbox{Log}_{m}(X_i)^T,
\end{equation} 
 there is no simple expression for the expected Hessian $H_c$. Instead of estimating $H_c$ directly, we propose to estimate $\vv^T H_c \vv$ for a set of vectors $\vv$, from which one can construct the estimator for $H_c$. This estimation strategy can be explained in a metaphor: If we know the cross-section of an object in each direction, we may infer the overall shape of the object. Our closed-form expression for the moment-based estimator of $
 \vv^T H_c \vv$ and the detailed algorithmic procedure for the estimation of $H_{c}$ are described in Section B.6.1.
 
The multiplicative factor 4 that appears in (\ref{eq:Sigma_c_est_main}) is an artifact of using the form of squared distances in $\rho_c$, rather than using half of the squared distance. A detailed procedure for the estimation of $H_{c}$ is described in Section B.6.1 of the supplementary material. The estimation strategy can be explained in a metaphor: If we know the cross-section of an object in each direction, we may infer the overall shape of the object. 
Write  $\hat{H}_c$ for the estimator of $H_c$. The proposed estimators are then consistent, as established below.

\begin{theorem} \label{thm:cov:est}
Let $c \in (0, \infty]$ be prespecified and assume that $P_{X}$ satisfies Assumptions (A1)--(A4) for the normal coordinate chart $\phi_{m^{c}_{0}} = \mbox{Log}_{m^{c}_{0}}$ centered at $m^{c}_0$. Suppose further that $P_X$ is absolutely continuous with respect to $V$ and the support of $P_X$ lies in $B_{r_{\mbox{\tiny cx}}}(m^{c}_{0})$. Then, $\hat\Sigma_c$, $\hat{H}_c$ and $\hat{H}_{c}^{-1} \hat\Sigma_{c} \hat{H}_{c}^{-1}$ are consistent estimators of $\Sigma_c$, $H_c$ and $H_{c}^{-1}\Sigma_{c}H_{c}^{-1}$, respectively, as $n \to \infty$.
\end{theorem}

\subsection{An application to hypothesis testing}\label{subsec:application:hypo}

The asymptotic normality of sample Huber means on the tangent space at the true Huber mean, established in Theorem~\ref{thm:clt}, and the consistent limiting covariance estimator, derived in Section~\ref{subsec:est:cov}, enable us to devise a one-sample location test procedure. 
For simplicity, we assume that the population Huber means are the same for all $c \in (0, \infty]$, and are given by a common point $m_0 \in M$. 

Consider the following hypotheses: For a given $\tilde{m}_0 \in M$, $$\textbf{H}_{0}: m_0 = \tilde{m}_{0} \quad \mbox{vs} \quad \textbf{H}_{1}: m_{0} \neq \tilde{m}_{0}.$$ 
To test the hypotheses, we propose to use a (Hotelling-like) test statistics 

$$
T_{n} = n \mbox{Log}_{\tilde{m}_0}(m^{c}_{n})^{T} \hat{A}^{-1}_{c} \mbox{Log}_{\tilde{m}_0}(m^{c}_{n}),    
$$
where $m^c_n$ is the sample Huber mean, and $\hat{A}_c := \hat{H}_c^{-1} \hat\Sigma_c \hat{H}_c^{-1}$ (Section~\ref{subsec:est:cov}). At the significance level $\alpha \in (0,1)$, the test rejects the null hypothesis if $T_{n} \ge \chi^{2}_{k, \alpha}$ where $\chi^{2}_{k, \alpha}$ stands for the upper $\alpha$ quantile of $\chi^{2}_{k}$.
%
%Due to \Cref{thm:cov:est}, $\hat{A}_{c}$ is a consistent estimator of $A_{c} = H^{-1}_{c}\Sigma_{c}H^{-1}_{c}$, i.e., $\hat{A}_{c} = A_{c} + o_{P}(1)$ as $n \rightarrow \infty$. 
An application of Theorems \ref{thm:clt} and \ref{thm:cov:est} verifies that the type I error of the proposed test is controlled for large sample sizes.

\begin{corollary}\label{cor:hypotest}
For a given $c \in (0,\infty]$, suppose that the conditions of Theorem~\ref{thm:cov:est} are all satisfied with $m_0^c = \tilde{m}_0$. Then, given any $\alpha \in (0, 1)$, $\lim_{n\to\infty} P(T_n \ge \chi^2_{k,\alpha} \mid \textbf{H}_{0}) = \alpha$, i.e., the type I error rate of the proposed test converges to $\alpha$ as $n \rightarrow \infty$. 
\end{corollary}

To inspect the asymptotic power of the proposed test, we consider a sequence of contiguous alternative hypotheses: 
\begin{equation}
    \label{eq:alternatives}
    \textbf{H}_{0}: m_0 = \tilde{m}_{0} \quad \mbox{vs} \quad \textbf{H}_{1(n)}: m_{0} = \tilde{m}_{n},
\end{equation} 
where $\tilde{m}_{n}$ satisfies $d(\tilde{m}_{0}, \tilde{m}_{n}) = O(n^{-\beta})$ as $n\rightarrow \infty$ for a constant $0 < \beta < 1/2$.
Our next result, a corollary of Theorems \ref{thm:clt} and \ref{thm:cov:est}, states that the proposed level $\alpha$ test possesses an asymptotic power of 1 even under the contiguous alternatives (\ref{eq:alternatives}) shrinking towards $\mathbf{H}_{0}$ at a rate $d(\tilde{m}_0, \tilde{m}_{n}) = O(n^{-\beta})$ as $n\rightarrow \infty$.

\begin{corollary}\label{cor:stat:power}
For a given $c \in (0,\infty]$, suppose that the conditions of Theorem~\ref{thm:cov:est} are all satisfied with $m_0^c = \tilde{m}_n$ for each $n$. Then, $\lim_{n\to \infty} P(T_{n} \ge \chi^{2}_{k, \alpha} \mid \textbf{H}_{1(n)}) = 1$ for any $\alpha \in (0, 1)$. 
\end{corollary}

We confirm numerically that the proposed Huber mean-based test outperforms the \Frechet\ mean-based test under contamination or heavy tails. See Section C.1 for details.

In our experiment, when the underlying distribution is contaminated or has a long tail, the proposed Huber mean-based test is more powerful than the \Frechet\ mean-based test. For details on the numerical performance of the proposed test, see Section C.1.

\subsection{Relative efficiency and choice of the robustification parameter}\label{sec:RE}
The problem for choosing the robustification parameter $c$ of the Huber loss has been extensively studied in regression settings \citep[cf.][]{holland1977robust, zhou2018new}. In these settings, 
several criteria have been used to choose the parameter $c$, which are either derived from asymptotic theory \citep{holland1977robust}, finite-sample theory \citep{zhou2018new}, or data-dependent methods \citep{jiang2019robust}. 
To our knowledge, no work has addressed the selection of $c$ for Huber means on general metric spaces, including the usual multivariate Euclidean space $\bbR^{k}$ for $k \ge 2$. 

We discuss a strategy for choosing the value of $c$ based on the asymptotic results in Section~\ref{sec:asymptotics}. Specifically, we propose selecting $c$ such that the asymptotic relative efficiency (suitably defined for manifold-valued estimators) of the Huber mean $m_n^c$, relative to the Fr\'{e}chet mean, is at least 95\%. It will be verified that $c$ can be tuned so that the Huber mean is highly efficient even under Gaussian-type distributions on manifolds.

To evaluate the performance of $M$-valued estimators, the notion of relative efficiency is employed. For $n$ observations $X_{1}, X_{2}, ..., X_{n} \overset{i.i.d.}{\sim} P_{X}$ on $M$, suppose that $m_n = m_{n}(X_{1}, X_{2}, ..., X_{n})$ and $m'_n = m'_{n}(X_{1}, X_{2}, ..., X_{n})$ are unbiased estimators of the parameter $\mu = \mu(P_X) \in M$. For a local coordinate chart $\phi$ defined on an open set $U \subset M$ containing both $m'_{n}$ and $m_{n}$, the \textit{relative efficiency} (RE) of $m_{n}$ relative to $m'_{n}$ is defined as 
\begin{equation}\label{def:RE}
\textsf{RE}_{\phi}(m_{n}, m'_{n}) := \frac{ \textsf{tr}\{\textsf{Cov}(\phi(m'_{n}))\}}{\textsf{tr}\{\textsf{Cov}(\phi(m_{n}))\}}, 
\end{equation} 
where $\textsf{tr}(A)$ denotes the trace of the matrix $A$. We say that $m_{n}$ is more efficient than $m'_{n}$ if $\textsf{RE}_{\phi}(m_{n}, m'_{n}) > 1$. Since the covariance matrix $\textsf{Cov}(\phi(m_{n}))$ for a manifold-valued estimator $m_n$ is in general challenging to grasp, the \textit{asymptotic relative efficiency} (ARE) can be used to compare the performances of $m_n$ and $m_n'$ in a large-sample setting. The ARE of $m_n$  relative to $m'_n$ is defined as the limit of $\textsf{RE}_{\phi}(m_{n}, m'_{n}) $ as the sample size $n$ increases. That is, 
$\textsf{ARE}_{\phi}(m_{n}, m'_{n}):= \lim_{n \rightarrow \infty} \textsf{RE}_{\phi}(m_{n}, m'_{n})$.

Assume that for a location parameter $\mu = \mu(P_{X}) \in M$, the population Huber means are equal to $\mu$; that is, $\mu=m_{0}^c$  for any $c \in (0, \infty]$. Denote $\phi$ by a coordinate chart of $M$ centered at $\mu$. For two robustification parameters $c_1, c_2 \in (0, \infty]$, suppose that $m^{c}_{n}$ for $c=c_{1} ~ \mbox{and} ~ c_{2}$ are unbiased estimators of the parameter $\mu$; see Proposition~\ref{prop:unbiased}(b). In addition, we assume that the regularity  conditions for the CLT (Theorem~\ref{thm:clt}) hold. Under these settings, we compare the efficiencies of the Huber means for $c_1$ and $c_2$ by the ARE:
\begin{equation}\label{def:ARE}
\textsf{ARE}_{\phi}(m^{c_{1}}_{n}, m^{c_{2}}_{n}) = \lim_{n \rightarrow \infty} \frac{\textsf{tr}\{\textsf{Cov}(\phi(m^{c_{2}}_{n}))\}}{\textsf{tr}\{\textsf{Cov}(\phi(m^{c_{1}}_{n}))\}} = \frac{\textsf{tr}(H_{c_{2}}^{-1}\Sigma_{c_{2}} H_{c_{2}}^{-1})}{\textsf{tr}(H_{c_{1}}^{-1}\Sigma_{c_{1}} H_{c_{1}}^{-1})},
\end{equation}
where the matrices $H_c$ and $\Sigma_c$ are given in Theorem~\ref{thm:clt}. 
See Section B.8.1 for the conditions and the proof of Equation (\ref{def:ARE}). 
%, and the last equality is satisfied since the technical conditions hold (see Section B.8.1 of the supplementary material for the conditions and the proof of Equation (\ref{def:ARE})). 
We say that $m^{c_{1}}_{n}$ is asymptotically more efficient than $m^{c_{2}}_{n}$ if $\textsf{ARE}_{\phi}(m^{c_{1}}_{n}, m^{c_2}_{n}) > 1$. 

It is well known that under the Gaussian assumption for the case $M = \mathbb{R}$, the sample mean $m_n^\infty$ is more efficient than Huber means $m_n^c$ for $c<\infty$. As there is no free lunch, to be resistant to unpredicted outliers (cf. Section~\ref{sec:robust}), the cost of efficiency, under Gaussian distributions, should be paid off \citep{huber2004robust}.
For general manifolds, the Gaussian assumption may be replaced by a Gaussian-type distribution on $M$ \citep{pennec2006intrinsic, said2022gaussian}, defined by the density function %defined below The %isotropic 
%Gaussian-type distribution on $M$ 
%with location parameter $\mu \in M$ and scale parameter $\sigma>0$ has the density function  
\begin{equation}\label{eq:Gaussian-type_distn}
f^{G}_{\mu, \sigma}(x) = C\exp\{-d^2(x, \mu)/2\sigma^2\} \quad \mbox{for any} ~ x \in M,
\end{equation}
where $\mu \in M$ and $\sigma>0$ are the location and scale parameters, and $C$ is a normalizing constant. %\footnote{\textcolor{red}{It can be computed by either the approximated or exact formulas presented in \citep[Sections 6.4 and 6.6,][]{pennec2006intrinsic} and \citep[Section 2,][]{said2022gaussian}, respectively.}} 
We propose to select the parameter $c$ of the Huber loss function as the smallest $c$ such that the asymptotic efficiency of $m_n^c$ is at least 95\% when compared to the Fr\'{e}chet mean under the Gaussian-type distribution on $M$:
% The proposition below supports that such an attempt is worthwhile.  
%
% \begin{proposition}\label{prop:c:worthwhile}
% Suppose that $P_{X}$ satisfies Assumptions (A1), (A3), and (A4) for all $c \in (0, \infty]$, and $(\ref{def:ARE})$ always holds. Then, there is a constant $c_{0} \in (0, \infty)$ satisfying $\textsf{ARE}_{\phi}(m^{c_{0}}_{n}, m^{\infty}_{n}) \ge 0.95$.
% \end{proposition}
%
% \Cref{prop:c:worthwhile} states that, for \textit{any} underlying distribution, the robustification parameter $c$ can be tuned so that the sample Huber mean is at least 95\% efficient relative to the Fr\'{e}chet mean. 
%Define 
\begin{equation}\label{eq:c_hat}
    \hat{c} = \inf\{c \in (0, \infty): \textsf{ARE}_{\phi}(m^{c}_{n}, m^{\infty}_{n}) \ge 0.95\},
\end{equation}
in which the ARE is computed under the Gaussian-type distribution. %(for a detailed rationale of (\ref{eq:c_hat}), see Section B.8.2 of the supplementary material). [The conditions listed in Section B.8.2 are not satisfied for the Gaussian-type distribution!] 

Note that $\rho_{c}\{d(m_{1}, m_{2})\} = \sigma^2 \cdot \rho_{c/\sigma} \{d(m_{1}, m_{2})/\sigma\}$ for any $m_{1}, m_{2} \in M$, $c \in (0, \infty)$, and for any scale parameter $\sigma > 0$.  
With this in mind, since Huber means are minimizers of the expected Huber loss, it is reasonable to further parameterize $c$ to be $c = \kappa \sigma$ with $\kappa \ge 0$. For the case $M=\bbR$, $\hat{c} = 1.345\sigma$ has been the traditional choice for $c$, ensuring the ARE of 95\% against the sample mean \citep{holland1977robust, huber2004robust};
See the top left panel of Figure~\ref{fig:cutoff}. %(See Section B.8.3 of the supplementary material for the evaluation of this figure.)

\begin{figure}[!t]
\center
\includegraphics[width = 0.8\linewidth]{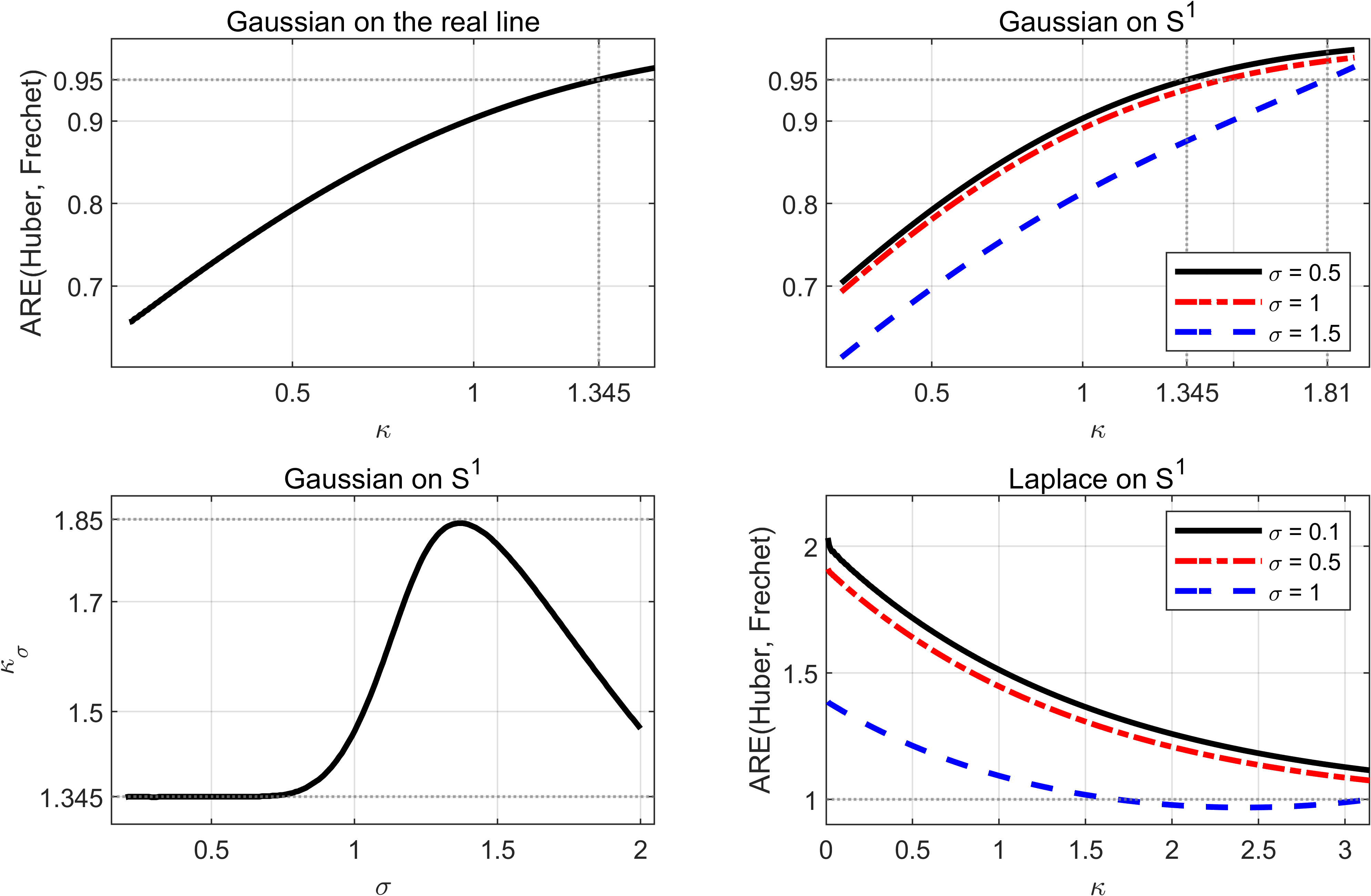}
\caption{(Top left) The ARE of the Huber mean over the \Frechet\ mean for varying $\kappa = c / \sigma$, under the Gaussian distribution on the real line. (Top right) The ARE versus $\kappa$ under the Gaussian-type distribution on $S^1$. 
(Bottom left) The ``optimal" $\kappa_\sigma$ versus $\sigma$ under the Gaussian-type distribution on $S^1$. (Bottom right) The ARE versus $\kappa$ under the Laplace-type distribution on $S^1$.
}
\label{fig:cutoff}
\end{figure}

We exemplify the choice of $c$ by $\hat{c}$ (\ref{eq:c_hat}) in the case where $M = S^1$. The unit circle $S^{1}$ is a manifold that has many statistical applications in computational biology and geostatistics, such as protein folding angles and wind directions \citep[e.g.,][]{huckemann2021data}. 
For bounded manifolds such as $S^1$, the Gaussian-type distribution (\ref{eq:Gaussian-type_distn}) has different shapes for different scale parameters. For instance, when parameterized by the signed angle and for small $\sigma$, the density $f^{G}_{\mu, \sigma}(\theta)$ closely resembles the Gaussian density functions. On the other hand, for large $\sigma$, the distribution is close to the uniform distribution on $S^1$. Thus, the meaning of the scale parameter $\sigma$ for Gaussian-type distribution on $S^{1}$ does not correspond to the standard deviation of the distribution.  
Nevertheless, we parameterize $c = \kappa \sigma$, and evaluate the value of $\kappa$ that provides 95\% ARE of the Huber mean, for several choices of $\sigma$ in (\ref{eq:Gaussian-type_distn}). This choice of $\kappa$ depends on $\sigma$, and we write the chosen value by $\kappa_\sigma$.

%In the case of $S^{1}$, there needs some care. 
%Although the meaning of scale $\sigma$ in $S^{1}$ is ambiguous unlike $\mathbb{R}$, along the same line with \Cref{0.95efficiency_R}, we decompose $\hat{c} = (\hat{c}/\sigma) \cdot \sigma =: \kappa \sigma$, and now estimate $\kappa$ in $S^1$ for several values of $\sigma$. We often denote $\kappa$ by $\kappa_{\sigma}$, because the values of $\kappa$ depend on those of $\sigma$. 

\begin{example}\label{0.95efficiency_S}
We parameterize the unit circle $S^{1}=\{(x,y) \in \bbR^2: x^2 + y^2 = 1\}$ by $\theta \in [-\pi, \pi)$. For a scale parameter $\sigma>0$, consider the Gaussian-type distribution $f^G_{\mu,\sigma}$, where the location parameter is set to be $\mu = (1,0) \in S^1$ (or equivalently by $\theta_\mu = 0$). For the open set $U = \{(\cos(\theta), \sin(\theta)): -\pi < \theta < \pi\} \subset S^1$ containing $\mu$, we choose to use the coordinate chart $\phi_{\mu}:U \rightarrow (-\pi, \pi)$, defined by $\phi_\mu((x,y)) = {\rm atan2}(y, x)$. For $X \sim f^{G}_{\mu, \sigma}$, we have 
\begin{equation}\label{eq:1d-hessianandcovariances}
\begin{aligned}
    \Sigma_{c} & =4\big[E\{d^2(X, \mu) \cdot 1_{d(X, \mu) \le c}\} + c^2P(d(X, \mu) > c)\big]  \\
    & = 4\sigma^2\{ E (Z_\sigma^2\cdot 1_{ |Z_\sigma| \le \kappa}) + \kappa^2 P(|Z_\sigma| > \kappa)\}, \\
    H_{c} & = 2P(d(X, \mu) < c) = 2P(|Z_\sigma| < \kappa),  \\
    \Sigma_{\infty} & = 4\sigma^2 \textsf{Var}(Z_\sigma) ~ \mbox{ and  } ~ H_{\infty} = 2,
\end{aligned}
\end{equation} 
where $Z_\sigma$ follows the standard Gaussian distribution, truncated to lie on $(-\pi/\sigma, \pi/\sigma)$. Numerically evaluating the above, and by (\ref{def:ARE}), we have evaluated $\textsf{ARE}_{\phi_{\mu}}(m^{\kappa \sigma}_{n}, m^{\infty}_{n})$ over a fine grid for $\kappa$ and several choices of $\sigma > 0$. In Figure~\ref{fig:cutoff} (top right panel), for each choice of $\sigma = 0.5, 1, 1.5$, the ARE as a function of $\kappa$ is plotted. 
For the smaller scale parameter $\sigma = 0.5$, choosing $\kappa = 1.345 ~ (=\kappa_{0.5})$ provides exactly 95\% ARE. This is consistent with the real line case, as the underlying distribution is nearly Gaussian. %This is consistent with Example \ref{0.95efficiency_R}, as the underlying distribution is nearly Gaussian. 
For larger scale parameters, the underlying distribution has smaller kurtosis (that is, less tailedness), and 95\% ARE is achieved at higher values of $\kappa$. In the bottom left panel of Figure~\ref{fig:cutoff}, we also depict the relationship between the scale parameter $\sigma$ and $\kappa_{\sigma}$ that provides 95\% ARE of the Huber mean. The relationship is nonlinear, but regardless of the value of the scale parameter, choosing $\kappa \in [1.345, 1.85]$ seems appropriate. 
\end{example}

From Example \ref{0.95efficiency_S},  
the recommended choice of $c$ is $c = 1.345\hat\sigma$, where $\hat\sigma$ is a robust estimate of the scale parameter. For $\hat\sigma$, one may use the median absolute deviation (MAD) divided by 0.6745 \citep{birch1980some}, 
where MAD is the median of the absolute deviations $d(x_{i}, \hat{\mu})$ from an initial location estimator $\hat\mu$ (e.g., the geometric median of $x_{1}, x_{2}, ..., x_{n}$). Numerical experiments on $S^2$ and $S^3$, reported in Section B.8.3, confirm that setting $c = 1.345 \hat{\sigma}$ yields a suitable choice.
% In addition, numerical evaluations for selecting $c$ under Gaussian-type distributions %with $\sigma = 0.1, 0.5, 1$
% on $S^2$ and $S^3$ are performed in Section B.8.3 of the supplementary material. Results indicate that the choice of $c = 1.345 \hat{\sigma}$ is suitable for both $S^2$ and $S^3$. 

We now demonstrate that the Huber means can be more efficient than \Frechet\ means, when $P_X$ is heavy-tailed. It is known that Huber means are efficient under the Laplace distribution on the real line (see Section B.8.3 for our proof). As a simple example on a manifold, consider the (truncated) Laplace distribution on $S^{1}$, whose density is given by, for $\mu \in M$ and $\sigma>0$, 
$f^{L}_{\mu, \sigma}(x) = C\exp\{-d(x, \mu)/\sigma\}$, 
% \begin{equation}\label{eq:Laplace-type_distn}
% f^{L}_{\mu, \sigma}(x) = C\exp\{-d(x, \mu)/\sigma\} \quad \mbox{for any} ~ x \in M,
% \end{equation}
where $C$ is a normalizing constant.

%To compare the efficiency of Huber means with Frechet means, we consider an isotropic Laplace-type distribution on $S^{1}$. The Laplace-type distribution on $M$ with location parameter $\mu \in M$, scale parameter $\sigma>0$ has the density function 
% \begin{equation}\label{eq:Laplace-type_distn}
% f^{L}_{\mu, \sigma}(x) = C\exp\{-d(x, \mu)/\sigma\} \quad \mbox{for any} ~ x \in M,
% \end{equation}
%where $C$ is a normalizing constant.

\begin{example}\label{circle:laplace}
Consider the sample Huber means on $S^1$ and suppose that $P_X$ is the Laplace-type distribution on $S^1$ with the location parameter $\mu = (1,0)$ and $\sigma > 0$. For $X \sim f^{L}_{\mu, \sigma}$, 
the corresponding expressions for $\Sigma_c$ and $H_c$ are (\ref{eq:1d-hessianandcovariances}) with $Z_\sigma$ replaced by $L_\sigma$ (following a truncated Laplace distribution). Utilizing these, 
%expressions and (\ref{def:ARE}), 
the ARE of $m^{\kappa \sigma}_{n}$ over the Fr\'{e}chet mean is evaluated over a fine grid of $\kappa$, and is displayed at the bottom right panel of Figure~\ref{fig:cutoff}. Overall, the Huber means are asymptotically more efficient than the Fr\'{e}chet mean for most combinations of $\kappa$ and $\sigma$.
\end{example}

% \begin{figure}[!ht]
% \center
% \includegraphics[scale=0.25]{figures/cutoff(circle)_Laplace.pdf}
% \caption{Asymptotic efficiency of Huber means $m_n^c$ relative to the Fr\'{e}chet mean on $S^{1}$, where $c = \kappa\sigma$, and
% $\sigma=0.01$ (black), $0.1$ (red), $1.0$ (blue).}
% \label{fig:cutoff:Laplace}
% \end{figure}

\subsection{Robustness of sample Huber means}\label{sec:robust}
%The breakdown point of an estimator is a popular measure of robustness. 
In this subsection, 
we evaluate the breakdown point of the sample Huber mean, demonstrating its high robustness. Throughout, the sample Huber mean is assumed unique.

Let $\bX= (x_{1}, x_{2}, ..., x_{n}) \in M^{n}$ be an $n$-tuple of deterministic data points on $M$. For $k \ge 1$, we write $\bY_{k}= (y_{1}, y_{2}, ..., y_{n}) \in M^{n}$ for arbitrarily corrupted $\bX$, consisting of $k$ arbitrary points on $M$ and ($n-k$) points from $\bX$. Denote by $m^{c}_{n}(\bX)$ the sample Huber mean of $\bX$. The breakdown point of the sample Huber mean at $\bX$ is given by $\epsilon^{*}(m^{c}_{n}, \bX) = \min_{1\le k\le n}\{\frac{k}{n}: \sup_{\bY_{k}} d(m^{c}_{n}(\bX), m^{c}_{n}(\bY_{k})) = \infty \}$,
where the supremum is taken over all possible $\bY_{k}$. The higher the breakdown point is, the more resistant the Huber mean is to outliers. We show below that, for any sample size $n$, the breakdown point of the sample Huber mean for $c \in [0, \infty)$ is at least 0.5 on general Riemannian manifolds. 

\begin{theorem}[Breakdown point]\label{thm:robust:breakdown}
 
Let $\bX= (x_{1}, x_{2}, ..., x_{n})$ be an $n$-tuple of observations on $M$. Then, for any $c \in [0, \infty)$, $\epsilon^{*}(m^{c}_{n}, \bX) \ge \floor{\frac{n+1}{2}}/n$ where $\floor{\cdot}$ denotes the floor function.
\end{theorem}

%We conjecture that the breakdown point of pseudo-Huber means is also at least 0.5. %The breakdown point of the geometric median (i.e. the $c = 0$ case) is also given in \cite{lopuhaa1991breakdown, fletcher2009geometric}.

If $M$ is bounded, then the sample Huber mean, as well as the \Frechet\ mean, never breaks down, i.e., $\epsilon^{*}(m^{c}_{n}, \bX) =1$. It is thus natural to investigate the breakdown point of Huber means when $M$ is unbounded. We do this not only for the Huber mean, but also for all \emph{isometric-equivariant} estimators, when $M$ is \emph{homogeneous}. 

A map $\psi:M \rightarrow M$ is an isometry on $M$ if $d\{\psi(x),\psi(y)\} = d(x,y)$ for any $x, y \in M$, where $d(\cdot,\cdot)$ is the distance function on $M$. Examples of isometric transformations include translation, reflection, and rotation for $M=\bbR^{k}$;  rotation and reflection for $S^k$; and the matrix inversion, uniform scaling, and rotation for $(\mbox{Sym}^{+}(k), d_{\mbox{\tiny AI}})$ where $d_{\mbox{\tiny AI}}$ denotes the affine-invariant Riemannian distance on $\mbox{Sym}^{+}(k)$ \citep{pennec2006riemannian}. A manifold $M$ is said to be \emph{homogeneous} if its isometry group acts transitively; that is, for any $p,q\in M$, there exists an isometry $\psi$ such that $q = \psi(p)$. Examples of homogeneous manifolds include $\bbR^{k}$, $S^{k}, \bbH^{k}, ~ \mbox{and} ~ \mbox{Sym}^{+}(k)$ with their canonical Riemannian metrics. 

The notion of isometric equivariance is defined as follows.

\begin{definition}
For a fixed positive integer $n$, let $\bX = (x_{1}, x_{2}, ..., x_{n})$ be an $n$-tuple of observations on $M$. An estimator $m:M^{n} \rightarrow M, \bX \mapsto m(\bX)$ is said to be \emph{isometric-equivariant} if, for any isometry $\psi:M \rightarrow M$, $\psi\{m(\bX)\} = m\{\psi(\bX)\}$ where $\psi(\bX):= \big(\psi(x_{1}), \psi(x_{2}), ..., \psi(x_{n})\big)$.
\end{definition}

We note that for $\mathbb{R}^k$-valued multivariate statistics, one may seek affine equivariance of estimators \citep[e.g.,][]{scealy2021analogues}. However, the notion of affine equivariance is generally not defined for manifolds, and we have instead adopted the related concept of isometric equivariance.

We establish an upper bound on the breakdown point of isometric-equivariant estimators on $M$ that is unbounded and homogeneous.

\begin{theorem}\label{thm:breakdown:0.5}
Suppose $M$ is unbounded and homogeneous. Let $\mathbf{X}=(x_{1}, x_{2}, ..., x_{n})$ be an $n$-tuple of observations on $M$. For any $M$-valued isometric-equivariant estimator, say $m_{n}$, it holds that $\epsilon^{*}(m_{n}, \mathbf{X}) \le \floor{\frac{n+1}{2}}/n$.
\end{theorem}

For any $c\in[0, \infty]$ and for any $M$, the sample Huber mean is isometric-equivariant (see Section B.9 for proof). For any unbounded homogeneous manifold, Theorems \ref{thm:robust:breakdown} and \ref{thm:breakdown:0.5} postulate that the Huber mean possesses a breakdown point of 0.5, which is the highest possible breakdown point among all isometric-equivariant estimators.

When $M$ is a Hadamard manifold, the Huber means for $c < \infty$ achieve the highest breakdown point of 0.5. In contrast, the \Frechet\ mean has a breakdown point of $1/n$ for any unbounded $M$. See Section B.9.3 for details.

\section{Numerical Examples}\label{sec:examples}

In this section, we provide numerical examples for utilizing sample Huber means. The numerical performances of the covariance estimator and the one-sample location test, proposed in Sections \ref{subsec:est:cov} and \ref{subsec:application:hypo}, are satisfactory, as demonstrated in Section C.1.

\subsection{Robustness and efficiency of Huber means}

We demonstrate that the Huber means are resistant to outliers on manifolds, and can be more efficient than the \Frechet\ mean. In the first example, we consider simulated data sets on the two-dimensional unit sphere, displayed in Figure~\ref{fig:vmf_sphere}. The directional data are generated from a mixture of von Mises-Fisher distribution
\[
X \sim 0.9 \cdot \textsf{vMF}(\mu_{1}, \kappa) + 0.1 \cdot \textsf{vMF}(\mu_{2}, 20 \kappa), 
\]
where $\mu_{1}=(1, 0, 0)$ and $\mu_{2}=(0, 0, 1)$, for several choices of concentration parameter $\kappa$. The data represent the situation where 10\% of data %$\textsf{vMF}(\mu_{1}, \kappa)$ data 
are contaminated. 
While the \Frechet\ mean is influenced by the group of outliers (near $\mu_2$), the Huber mean provides an accurate and robust estimation of $\mu_1$. 

\begin{figure}[t]
    \centering
    \includegraphics[width = \linewidth]{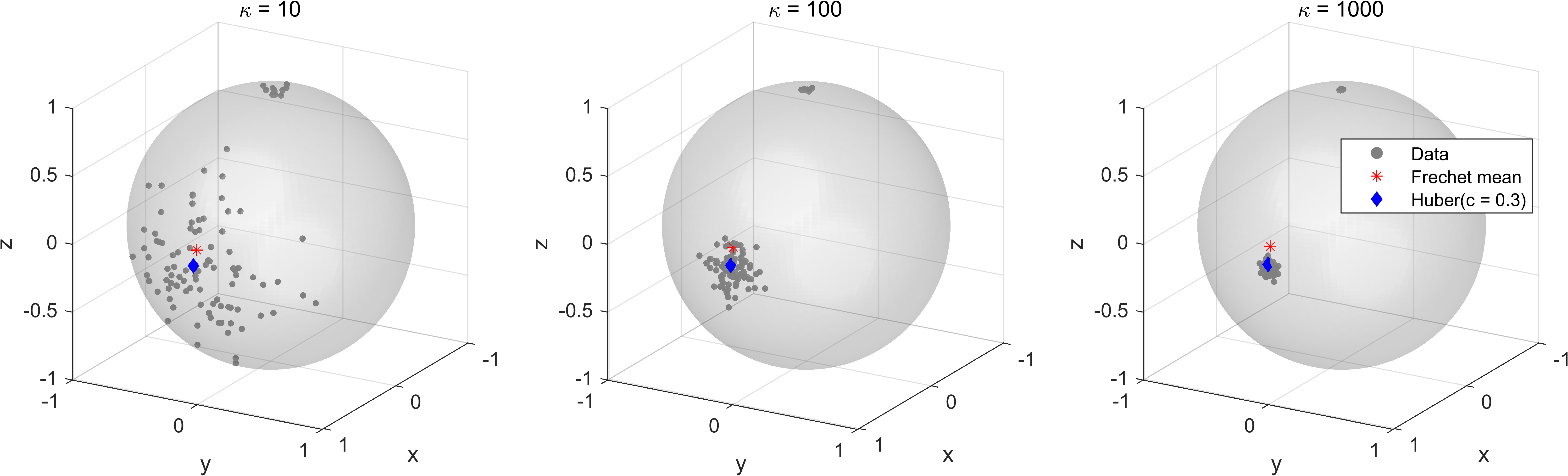}
    \caption{Toy data example on the unit sphere. The Huber mean (shown as blue diamond) is resistant to outliers (near the north pole), while the Fr\'{e}chet mean (red asterisk) is susceptible to outliers.}
    \label{fig:vmf_sphere}
\end{figure}

Our second example involves the unbounded manifold $\mbox{Sym}^{+}(2)$ with the canonical Riemannian metric, sometimes referred to as the affine-invariant geometric framework for $\mbox{Sym}^{+}(k)$ \citep{pennec2006riemannian}. The toy data are sampled from a lognormal distribution \citep{schwartzman2016}
on $\mbox{Sym}^{+}(2)$ with a mean of $I_2$. As done for the directional data example, 10\% of the data are replaced by gross outliers near $\mbox{Exp}_{I_{2}} \big(\begin{smallmatrix}
 10 &  5\sqrt{2} \\
  5\sqrt{2} &  10
 \end{smallmatrix}\big)$. This toy data set is displayed on the coordinates of $T_{I_2} \mbox{Sym}^+(2)$ in Figure~\ref{fig:mvn_spd}. There, the outliers are outside of the figure frame, so we have only indicated the direction of outlier locations. while the \Frechet\ mean is affected by the gross outliers, the Huber means as well as the geometric median remain resistant. Additionally, the Huber means continuously approach the true ``mean" at $I_2$ as the parameter $c$ decreases. 

\begin{figure}[t]
    \centering
    \includegraphics[width = 0.7\linewidth]{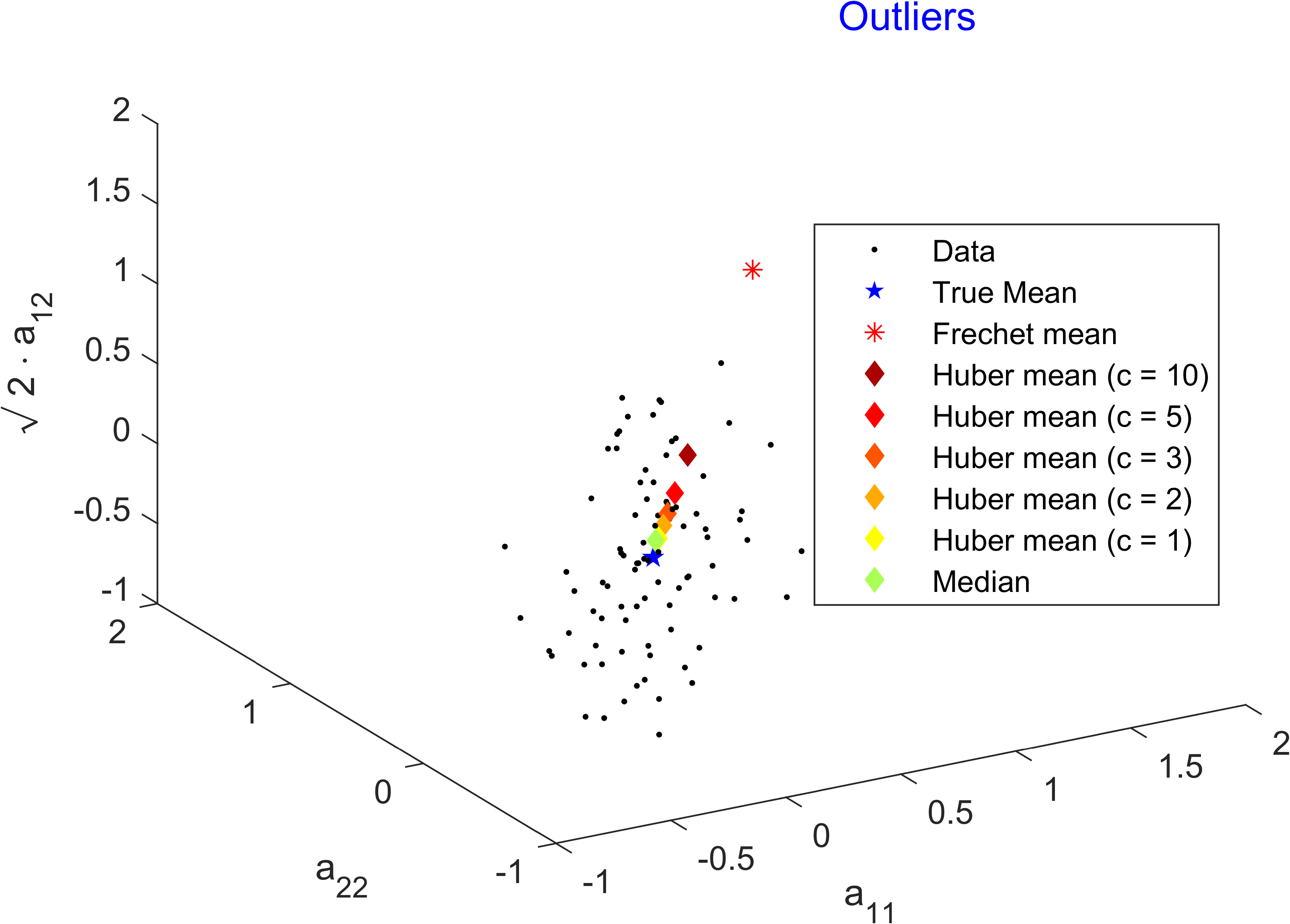}
    \centering
    \caption{Toy data example on $\mbox{Sym}^+(2)$. Huber means are closer to the uncontaminated true mean than the \Frechet\ mean, and they continuously approach the \Frechet\ mean as $c$ increases.}
    \label{fig:mvn_spd}
\end{figure}

To further support our claim that Huber means are more efficient than the \Frechet\ mean when the underlying distribution is contaminated or has heavy tails, we compared the mean squared errors (MSE) of Huber means with those of \Frechet\ means and geometric medians. For this, we take $m_0 = I_2$ as the true mean, and define the following:  
  $\mbox{MSE}(\hat{m}) := E d^2(\hat{m}, m_0)$, $\mbox{Bias}(\hat{m}) := d(E(\hat{m}), m_0)$ and $\textsf{Var}(\hat{m}) := Ed^2(\hat{m}, E(\hat{m}))$, where $E(\hat{m})$ is the \Frechet\ mean of the estimator $\hat{m}$. The first three columns of Table~\ref{table:num_result} compare (numerically evaluated) MSE, bias, and variances of the location estimators, when the underlying distribution is the contaminated lognormal distribution. In this case, 10\% of outliers force the \Frechet\ mean to be located far from the truth $I_2$, thus \Frechet\ means exhibit large bias, resulting in the inflated MSE. On the other hand, both Huber means and pseudo-Huber means show significantly smaller MSE. 
  
  \begin{table}[t]
\centering
\caption{The bias, variance and mean squared errors of Huber means, compared with those of \Frechet\ means, computed for the sample size $n = 100$, based on $1,000$ simulations. \label{table:num_result}}
\begin{tabular}{c||ccc|ccc}
  \multicolumn{1}{c}{}   & \multicolumn{3}{c}{lognormal with outliers} & \multicolumn{3}{c}{log-Laplace}\\ 
Type & Bias & Variance &MSE & Bias & Variance &MSE \\ \hline
Fr\'{e}chet mean & 1.728 & 0.003 & 2.990 & 0.004 & 0.045 & 0.045\\ 
Huber mean ($c=1$) & 0.106 & 0.003 & 0.014 & 0.007 & 0.035 & 0.035\\  
Pseudo-Huber mean ($c=1$) & 0.130 & 0.003 & 0.020 & 0.007 & 0.035 & 0.035 \\ 
Geometric median & 0.075 & 0.004 & 0.010 & 0.008 & 0.035 & 0.035
\end{tabular}
\end{table}

In addition, we also have evaluated the bias, variance, and MSE under a log-Laplace distribution centered at $m_0 = I_2$, where each coordinate of 
$ \mbox{Log}_{m_0}(X) $ is independently sampled from the standard Laplace distribution. As can be inspected from the latter three columns of Table~\ref{table:num_result}, the variance of Huber means is smaller than that of Fr\'{e}chet means. While all estimators are seemingly unbiased, the Huber mean appears to be more efficient than the Fr\'{e}chet mean with $\textsf{RE}_{\phi}(m^{1}_{n}, m^{\infty}_{n}) \approx \frac{0.0454}{0.0352} \approx 1.287 > 1$. 

In all examples, the geometric median performed slightly better than or comparably to the Huber means. However, only the Huber means enjoy asymptotic normality, enabling valid statistical inference. The asymptotic normality of the Huber means will be utilized in the real data analysis presented in the next subsection.

\subsection{An application to multivariate tensor-based morphometry
%A real data example from multivariate tensor-based morphometry
}\label{subsec:TBM}

To demonstrate the usefulness of the proposed Huber means on manifolds, we consider a real data set from a multivariate tensor-based morphometry, consisting of $2 \times 2$ symmetric positive-definite (SPD) matrices. The data are obtained from a particular vertex of morphometry tensors obtained from registered lateral ventricles of $n = 36$ infants \citep{paquette2017ventricular}. Among those $2 \times 2$ SPD matrices, there is a seemingly outlying observation, as displayed in Figure~\ref{fig:morphometry}. For displaying points and regions in the manifold ${\mbox{Sym}^+(2)}$, we used the orthogonal coordinates of the tangent space at $m_n^{c}$, the (sample) Huber mean for $c = 0.805$. A point $\xv = (x_1,x_2,x_3)^T$ in the figure corresponds to an SPD matrix 
$\mbox{Exp}_{m_n^{0.805}}\big(\begin{smallmatrix}
 x_{1} & x_{3}/\sqrt{2} \\
 x_{3}/\sqrt{2} &  x_2
 \end{smallmatrix}\big)$. 
 We chose the robustification parameter $c = 1.345\hat\sigma$, where $\hat\sigma$ is a robust estimate of $\sigma$ in the normal model of $N_3(\mathbf{0},\sigma^2 I_3)$ on the tangent space at the median, and the value 1.345 was adopted from our recommendation in Section \ref{sec:RE}.

\begin{figure}[t]
\centering
\includegraphics[width = \linewidth]{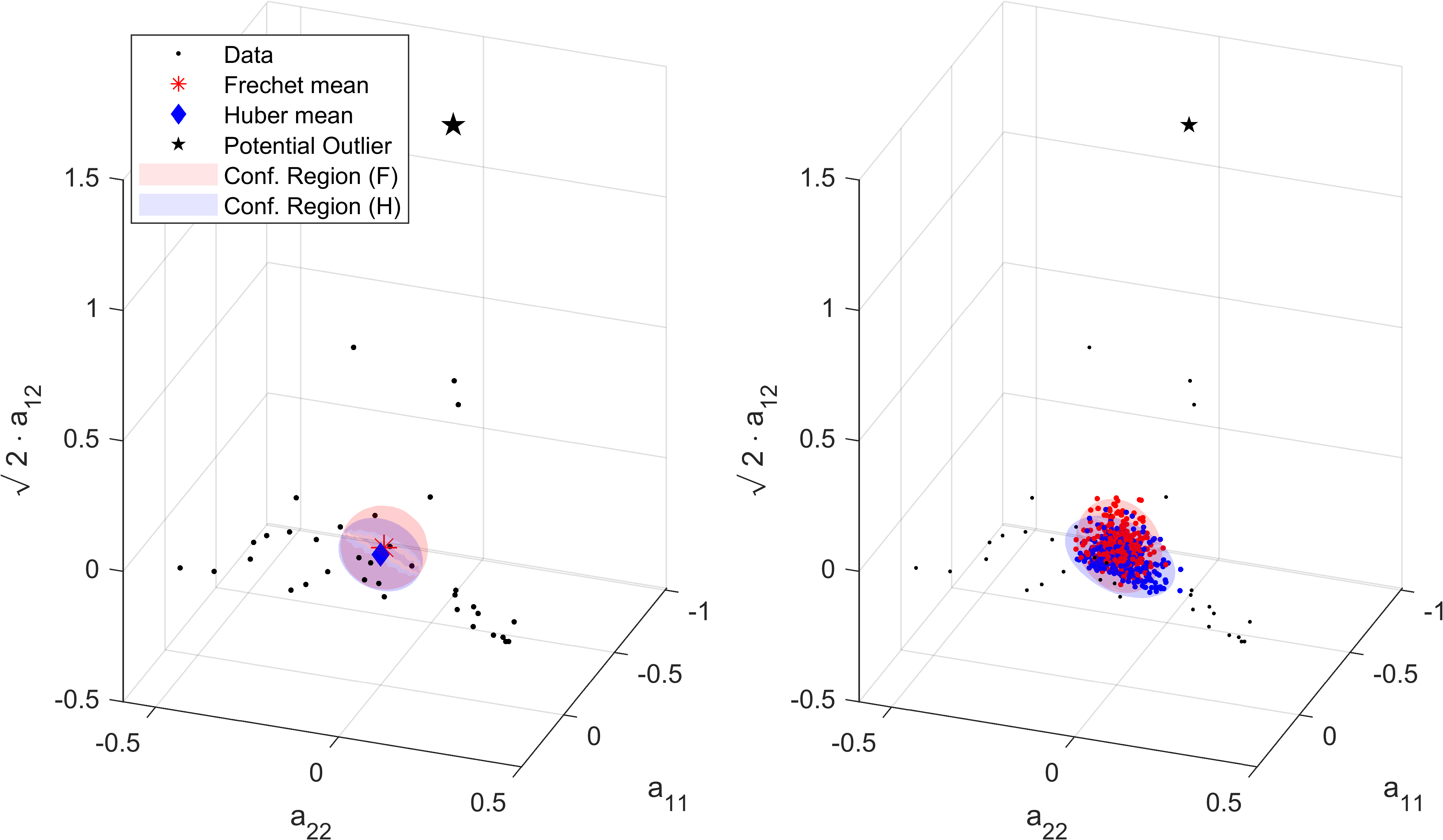} 
\caption{Huber means and their approximate confidence regions for a real sample of $2\times 2$ SPD matrices. (Left) Data and a potential outlier, overlaid with the sample \Frechet\ mean and the sample Huber mean for $c = 0.805$, and their 95\% confidence regions computed from the limiting covariance estimates (\ref{subsec:est:cov}). (Right) The same data with bootstrap replicates of the sample \Frechet\ and Huber means (red and blue dots, respectively), as well as bootstrap-approximated confidence regions.}
\label{fig:morphometry}
\end{figure}

In Figure~\ref{fig:morphometry}, we confirmed that the Huber mean
%for $c = 0.805$ 
is less influenced by the potential outlier, when compared to the Fr\'{e}chet mean. As theoretically demonstrated in Section~\ref{sec:robust}, 
if the outlier were positioned farther from other data points, then the Fr\'{e}chet mean would break down, while the Huber mean would not. To demonstrate the robustness of Huber means in the presence of outliers, we have computed approximate confidence regions for the Huber mean (for $c = 0.805$) and for the \Frechet\ mean. To build these confidence regions, we utilize the asymptotic normality (Theorem~\ref{thm:clt}) together with the moment-based limiting covariance estimator $\hat{A}_c$. %(cf. Section \ref{subsec:est:cov}). 
Specifically, the 95\% confidence region on the tangent space $T_{m^{c}_n}\mbox{Sym}^+(2) \cong \mathbb{R}^3$ is given by 
\begin{equation}
    \label{eq:CR}\mbox{CR}(m_0^c) := \{x \in \mathbb{R}^3: x^{T} (n^{-1}\hat{A}_c)^{-1} x \le \chi^{2}_{3, 0.05}\},
\end{equation} 
where $\chi^{2}_{3, 0.05}$ stands for the upper 5$\%$ quantile of the $\chi^{2}_{3}$ distribution. 
By visually inspecting these confidence regions in the left panel of Figure~\ref{fig:morphometry}, we observe that $\mbox{CR}(m_0^c)$, the confidence region for the Huber mean (for $c = 0.805$) resembles the (scaled) covariance of the data without the outlier. Heuristically, this is because the gradients used in computing $\hat\Sigma_{c}$ are clipped at $c$, thus effectively compensating for the presence of outliers. On the other hand, the confidence region $\mbox{CR}(m_0^\infty)$ for the \Frechet\ mean  is heavily affected by the potential outlier,

The next questions we address are whether the sampling distribution of Huber means is approximately normal and also whether the approximate covariance matrix $n^{-1}\hat{A}_c$ is an acceptable estimator of $\textsf{Cov}(\mbox{Log}_{m_0}m_n^c)$. This is particularly important given the small sample size ($n = 36$). To this end, we obtained a nonparametric bootstrap estimator of $\textsf{Cov}(\mbox{Log}_{m_0}m_n^c)$ for $c = 0.805$ and $\infty$. Specifically, for each of $B = 300$ bootstrap samples of size $n$, we computed the Huber means. For each bootstrap replicate of the Huber mean, denoted as $\breve{m}_n^{c}(b)$, we compute  $\mbox{Log}_{m_n^c} \breve{m}_n^{c}(b)$ for $b = 1, 2, \ldots, B$. The right panel of Figure~\ref{fig:morphometry} displays these bootstrap replicates, from which we can glimpse the sampling distribution of $\mbox{Log}_{m_0}m_n^c$. Even with the small sample size, we observe that the bootstrap-approximate distribution of $\mbox{Log}_{m_0}m_n^c$ (i.e., the empirical distribution of the bootstrap replicates  $\mathcal{L_{\rm boot}}(m_n^c):= \{ \mbox{Log}_{m_n^c} \breve{m}_n^{c}(b): b= 1, 2, \ldots, B\}$) is approximately normal. (See Section C.2 for normal QQ plots.)

The figure is also overlaid with the parametric (bootstrap-approximated) 95\% confidence regions for the \Frechet\ and Huber means, which are given by (\ref{eq:CR}) with $n^{-1}\hat{A}_c$ replaced by the empirical covariance matrix of the bootstrap replicates $\mathcal{L_{\rm boot}}(m_n^c)$. %It can be observed that 
The bootstrap-approximated confidence regions (on the right panel) and the covariance-estimator-based confidence regions (left) are close to each other. We therefore conclude that for this dataset, $n^{-1}\hat{A}_c$ serves as a reliable estimator of the covariance for the Huber mean $m_n^c$.

\section{Discussion}\label{sec:discuss}

 %including several unsolved problems identified during our investigation, some of which are discussed below.

% \textbf{Breakdown point for the \Frechet\ mean}. The sample Huber means for any $c \in [0, \infty)$ achieves the breakdown point of 0.5, which is optimal among isometric-equivariant estimators on unbounded Riemannian symmetric spaces. While we conjecture that the breakdown point for the sample \Frechet\ mean on unbounded manifolds is $\tfrac{1}{n}$, our work does not confirm this conjecture.

% \textbf{Limiting distribution for the sample geometric median}. 
% In \Cref{sec:asymptotics}, we have demonstrated that the sample Huber mean is asymptotically normal for the cases $c \in (0,\infty]$. This does not include the sample Huber mean for $c = 0$ (i.e., the geometric median). For $M=\bbR^k$, it is established that the geometric median is asymptotically normal under certain technical conditions (see Theorem 2 in \cite{mottonen2010asymptotic}). Given this evidence, we surmise that the geometric median is also asymptotically normal on manifolds under \emph{reliable} regularity conditions.

In this article, we have introduced Huber means on Riemannian manifolds, and have comprehensively investigated their mathematical, statistical, and computational properties. Huber means are highly robust to outliers and are more efficient than the \Frechet\ mean under heavy-tailed distributions. Additionally, to leverage the asymptotic normality of the sample Huber mean, we proposed a novel estimation strategy for the limiting covariance matrix and used it to construct a robust test for location parameters. We believe that this work opens many interesting research directions.

Affine equivariance, for example, is a desirable property for location estimators in $\bbR^{k}$ \citep{lopuhaa1991breakdown, scealy2021analogues}. However, achieving the affine equivariance for manifold-valued estimators has been quite challenging. Instead, we have used the isometric equivariance for location estimators in Section~\ref{sec:robust}. Nevertheless, as \cite{scealy2021analogues} explored analogues of affine-equivariant geometric medians for spherical data, we believe that developing affine-equivariant Huber means on Riemannian symmetric spaces is a worthwhile pursuit.

From the perspective of robust statistics, there are at least two interesting research directions. The first is to appropriately define and calculate the 
\textit{influence function} for Huber means on manifolds. While the breakdown point is a useful and intuitive notion of robustness, the influence function captures the central stability aspect of robustness more comprehensively. Another immediate research direction is the generalization of Huber means to \emph{$M$-estimators} for manifold-valued data. For example, a class of robust $M$-estimators for location can be defined as a minimizer of $m \mapsto E\rho\{d(X, m)/\sigma\}$, where $\rho$ has a bounded derivative. This class includes those utilizing (pseudo) Huber loss, Tukey's bisquare loss, and Hampel loss, among others. In line with the literature on $M$-estimators \citep{huber1964robust, holland1977robust, zhou2018new}, one may investigate general $M$-estimators along with robust scale estimations on manifolds.

Similarly, robust dimension reduction and regression estimators can be developed by utilizing the Huber loss or other robustness-inducing loss functions. Most existing works on the dimension reduction for manifold-valued data have used $L_2$ loss functions. (see e.g. \cite{jung2012analysis,pennec2018barycentric}). 
The \Frechet\ regression frameworks for manifold-valued or metric space-valued responses \citep{cornea2017regression, petersen2019frechet} can also be generalized to, for instance, Huber regression for random objects.

\section*{Acknowledgement}  
Sungkyu Jung was supported by the National Research Foundation of Korea (NRF) grants funded by the Korea government (MSIT) (RS-2023-00301976 and RS-2024-00333399).

% \noindent {\it Conflict of interest}: The authors do not declare any conflict of interest regarding this work.

% \section*{Funding}
% Sungkyu Jung was supported by the National Research Foundation of Korea (NRF) grants funded by the Korea government (MSIT) (RS-2023-00301976 and RS-2024-00333399). 

% \section*{Data availability}
% The multivariate tensor-based morphometry data used in this study are not publicly available. However, researchers may request access by contacting the authors of the original study \citep{paquette2017ventricular}, subject to their approval. 

\begin{appendix}

\section{Additional theoretical properties of Huber means}

\subsection{Huber means as maximum likelihood estimators}

As an analogue on Riemannian manifolds of the classical location-scale model in $\bbR$, \citep[e.g.,][]{lax1985robust, rousseeuw2002robust}, we 
consider the following location-scale model:
\begin{equation}\label{distr:huber:mfd}
f_{\mu, \sigma, c}(x) = C_{\mu, \sigma, c} \exp\{-\rho_{c}\{d(x, \mu)/\sigma)\}\},
\end{equation}
where $c \in [0, \infty]$ and $C_{\mu, \sigma, c}=(\int_{M} \exp\{-\rho_{c}\{d(x, \mu)/\sigma)\})^{-1}$ is the normalizing constant, and $V$ stands for volume measure (Riemannian measure) on $M$. The class $\{f_{\mu, \sigma, c}\}_{\mu, \sigma, c}$ belongs to distance-decay distributions, which includes Gaussian-type ($c=\infty$) and Laplace-type ($c=0$) distributions.
\begin{proposition}\label{prop:MLE}
For a given $c \in [0, \infty]$, suppose that $P_{X}$ is absolutely continuous with respect to $V$, and $\frac{d P_{X}}{dV} = f_{\mu, 1, c}$ as specified in (\ref{distr:huber:mfd}) with scale $\sigma=1$. Then, the sample Huber means for $c$ are maximum likelihood estimators of the location parameter $\mu= \mu(P_{X}) ~(= \textsf{mode}(P_{X}))$.
\begin{proof}[Proof of Proposition~\ref{prop:MLE}]
Let $(x_{1}, x_{2}, ..., x_{n})$ be an $n$-tuple of deterministic observations on $M$. The log-likelihood function of $\mu \in M$ is denoted by
\[
\ell(\mu) = n \log(C_{\mu, \sigma, c}) - \sum_{i=1}^{n} \rho_{c}\{d(\mu, x_i)\}. 
\]
Accordingly, the sample Huber means for $c$ maximize the log-likelihood of $\mu$, $\ell(\mu)$.
\end{proof}
\end{proposition}

\subsection{Huber means bridge the Fr\'{e}chet mean and geometric median}\label{sec:bridge}

The Huber means continuously connect the Fr\'{e}chet mean and geometric median. To firmly show this, we utilize a notion of the limit of sets:
\begin{definition}\label{def_outer_limit}
Let $E, E_{1}, E_{2}, ...$ be a sequence of subsets of $M$. The sequence $(E_{n})_{n \ge 1}$ is said to converge to $E$ in \textit{outer limit} if
\[
\cap_{n=1}^{\infty}\overline{\cup_{k=n}^{\infty}E_{k}} \subseteq E.
\]
\end{definition}

Lemma~\ref{lem:subseq} characterizes the points in $\cap_{n=1}^{\infty}\overline{\cup_{k=n}^{\infty}E_{k}}$.
\begin{lemma}\label{lem:subseq}
For a point $p \in M$, $p \in \cap_{n=1}^{\infty}\overline{\cup_{k=n}^{\infty}E_{k}}$ if and only if there exist an increasing sequence $k_{n} ~ (k_{n} \ge n)$ and a sequence of points $p_{n} \in E_{k_{n}}$ such that $p_{n} \underset{n \rightarrow \infty}{\rightarrow} p$ (i.e., $d(p_{n}, p) \underset{n \rightarrow \infty}{\rightarrow} 0$).
\end{lemma}
\begin{proof}[Proof of Lemma~\ref{lem:subseq}]
($\Rightarrow$) Denote by $A'$ the set of accumulation points of $A \subseteq M$. It is known that $\overline{A}=A \cup A'$ for any set $A \subset M$. Define $A_{n}=\cup_{k=n}^{\infty}E_k$. Then, $p \in \cap_{n=1}^{\infty}\overline{A_{n}}$ implies $p \in \overline{A_{n}}=A_{n} \cup A'_{n}$ for all $n\ge 1$, thereby meaning that
\[
p \in A_{n} ~ \mbox{infinitely many} ~ n \quad \mbox{or} \quad p \in A'_{n} ~ \mbox{infinitely many}~ n.
\]
When the former case occurs, let $p_{n}=p$ for all $n \ge 1$. When the later case occurs, suppose that $p \in A'_{n_{k}}$ for any $k \ge 1$. By the definition of accumulation point,
\begin{eqnarray*}
\exists p_{1}, k_{1}~(\ge n_1) \quad &\mbox{s.t.}& \quad p_1 \in E_{k_1},~ d(p_1,p)<1, \\
\exists p_{2}, k_{2}~(\ge n_2, k_1) \quad &\mbox{s.t.}& \quad p_2 \in E_{k_2}, ~ d(p_2,p)<\frac{1}{2}, \\
\exists p_{3}, k_{3}~(\ge n_3, k_2) \quad &\mbox{s.t.}& \quad p_3 \in E_{k_3}, ~ d(p_3,p)<\frac{1}{3}, \\
&\vdots& \\
\exists p_{\ell}, k_{\ell}~(\ge n_\ell, k_{\ell-1}) \quad &\mbox{s.t.}& \quad p_\ell \in E_{k_\ell}, ~ d(p_\ell,p)<\frac{1}{\ell}.
\end{eqnarray*}
Replacing the $\ell$ by $n$, we obtain that $p_{n} \in E_{k_{n}}$ for any $n \ge 1$, and that $d(p_{n}, p) \underset{n \rightarrow \infty}{\rightarrow} 0$.  

($\Leftarrow$) If $p_{n}=p$ infinitely many $n$, then $p \in A_{m} \subset \overline{A_{m}}$ for any integer $m \ge 1$, meaning that $p \in \cap_{m=1}^{\infty} \overline{A_{m}}$. If $p_{n} \neq p$ all but finitely many $n$, then $p \in A'_{m} \subseteq \overline{A_{m}}$ for any integer $m\ge 1$, meaning that $p \in \cap_{m=1}^{\infty} \overline{A_{m}}$.

Combining these facts, we complete the proof.    
\end{proof}

The main ingredient for establishing the convergence of the Huber mean set (in the outer limit sense) is the following observation about the convergence of Huber loss as $c \rightarrow +0$ and $\infty$.
\begin{lemma}\label{lem:unif:conv}
For any $x \in [0, \infty)$, $\rho_{c}(x)$ converges pointwise to $L_{2}(x)$ as $c \rightarrow \infty$. Similarly, $\frac{\rho_{c}(x)}{2c}$ converges pointwise to $L_{1}(x)$ as $c \rightarrow +0$. The above statement holds when $\rho_{c}(x)$ is replaced by $\tilde{\rho}_{c}$.
\end{lemma}

\begin{proof}[Proof of Lemma~\ref{lem:unif:conv}]
For any $0 \le x < c$, $\rho_{c}(x) = L_{2}(x)$. That is, for any $x \in [0, \infty)$, $\rho_{c}(x)$ converges pointwise to $L_{2}(x)$ as $c \rightarrow \infty$. Moreover, we obtain that for any $x \in [0, \infty)$ 
\[
|\tilde{\rho}_{c}(x) - L_{2}(x)| = \frac{x^4}{(c+\sqrt{c^2+x^2})^2} \underset{c \rightarrow \infty}{\rightarrow} 0.
\]
Note that, for any $0 \le c \le x$,  $\frac{\rho_{c}(x)}{2c}=x - \frac{c}{2}$. Accordingly, for any $x \in [0, \infty)$ $\frac{\rho_{c}(x)}{2c}$ converges pointwise to $L_{1}(x)$ as $c\rightarrow +0$. Also, $\frac{\tilde{\rho}_{c}(x)}{2c}$ converges pointwise to $L_{1}(x)$ as $c\rightarrow +0$ for any $x \in [0, \infty)$ as follows:
\begin{eqnarray*}
|\frac{\tilde{\rho}_{c}(x)}{2c} - x| &=& |c(\sqrt{1+x^2/c^2} - 1) - x| \\
&=& |\sqrt{c^2 + x^2} - (c + x)|  \\
&=& \frac{2cx}{\sqrt{c^2+x^2}+c+x} \\
&\underset{c \rightarrow +0}{\rightarrow}& 0.
\end{eqnarray*}
\end{proof}

The collection of Huber mean sets bridges the Fr\'{e}chet mean set and geometric median set, which is established below. 

\begin{theorem}\label{thm:set:conv}
Suppose that $P_{X}$ has a finite variance. Then $E^{(c)}$ converges to $E^{(0)} ~ (\mbox{or} ~ E^{(\infty)})$ in outer limit as $c \rightarrow 0 ~ (\mbox{or} ~ c \rightarrow \infty$, respectively) where $E^{(0)}, E^{(\infty)}$ denote the sets of geometric median and Fr\'{e}chet mean, respectively.
\begin{proof}[Proof of Theorem~\ref{thm:set:conv}]
Note that $E^{(0)}$ and $E^{(\infty)}$ are non-empty by Theorem~\ref{thm:exist}. We shall show that, for any sequence $c_{k}$ with $c_{k} \underset{k \rightarrow \infty}{\rightarrow} \infty$, $E^{(c_{k})}$ converges to $E^{(\infty)}$ in outer limit; that is,
 \begin{equation}\label{outer_limit}
 \cap_{n=1}^{\infty} \overline{\cup_{k=n}^{\infty} E^{(c_{k})}} \subseteq E^{(\infty)}.
 \end{equation}
 For a given sequence $c_{n}$ as above, let $p \in \cap_{n=1}^{\infty} \overline{\cup_{k=n}^{\infty} E^{(c_{k})}}$. By Lemma~\ref{lem:subseq} there exist a subsequence $c_{k_\ell}$ and a sequence of points $p_{\ell} \in E^{(c_{k_{\ell}})}$ such that $p_{\ell} \underset{\ell \rightarrow \infty}{\rightarrow} p$. By the definition of $E^{(c_{k_\ell})}$, for any $m \in M$
 \begin{equation}\label{int_limit}
 \int \rho_{c_{k_\ell}}\{d(X, m)\} dP \ge \int \rho_{c_{k_\ell}}\{d(X, p_{\ell})\} dP.
 \end{equation} 
 Since $c_{k_{\ell}} \underset{\ell \rightarrow \infty}{\rightarrow} \infty$, for each $x \in M$ and for any sufficiently large $\ell$ 
 \begin{eqnarray*}
 |\rho_{c_{k_{\ell}}}\{d(x, p_{\ell})\} - d^2(x, p)| 
 &\le& |\rho_{c_{k_{\ell}}}\{d(x, p_{\ell})\} - \rho_{c_{k_{\ell}}}\{d(x, p)\}| \\ 
 & & + |\rho_{c_{k_{\ell}}}\{d(x, p)\} - d^2(x, p)| \\
 &=& |d^2(x, p_{\ell}) - d^2(x, p))| \\
 & & + |\rho_{c_{k_{\ell}}}\{d(x, p)\} - d^2(x, p)| \\
 & \underset{\ell \rightarrow \infty}{\rightarrow}& 0.
 \end{eqnarray*}
 Hence, $\rho_{c_{k_\ell}}\{d(x, p_{\ell})\} \underset{\ell \rightarrow \infty}{\rightarrow} d^2(x, p)$ for any $x \in M$. Taking $\lim_{\ell \rightarrow \infty}$ to (\ref{int_limit}), we obtain, for any $m \in M$
 \[
 \int d^2(X, m) dP \ge \int d^2(X, p) dP,  
 \]
 where the limit and integration in (\ref{int_limit}) can be interchanged by Lebesgue's dominated convergence theorem together with finite variance condition. (Notice that $\rho_{c_{k_{\ell}}}\{d(x, m)\} \le d^2(x, m), \rho_{c_{k_{\ell}}}\{d(x, p_{k_{\ell}})\} \le d^2(x, p) + 1$ for each $x \in M$ and any sufficiently large $\ell$.)
 It in turn means that $p \in E^{(\infty)}$. Therefore, (\ref{outer_limit}) holds.

Nextly, we shall show that, for any sequence $c_{k}$ with $c_{k} \underset{k \rightarrow \infty}{\rightarrow} +0$, $E^{(c_{k})}$ converges to $E^{(0)}$ in outer limit; that is,
 \begin{equation}\label{outer_limit2}
 \cap_{n=1}^{\infty} \overline{\cup_{k=n}^{\infty} E^{(c_{k})}} \subseteq E^{(0)}.
 \end{equation}
 For a given sequence $c_{k}$ as above, assume that $p$ is in the left-hand side of (\ref{outer_limit2}). Thanks to Lemma~\ref{lem:subseq}, there exists a subsequence $c_{k_\ell}$ and a sequence of points $p_{{\ell}} \in E^{(c_{k_{\ell}})}$ such that $p_{\ell} \underset{\ell \rightarrow \infty}{\rightarrow} p$. For any $m \in M$, we have 
 \begin{equation}\label{int_limit2}
 \int \rho_{c_{k_\ell}}\{d(X, m)\} dP \ge \int \rho_{c_{k_\ell}}\{d(X, p_{\ell})\} dP.
 \end{equation}
 Note that $\frac{\rho_{c}}{2c}(\cdot)$ converges pointwise to $L_{1}(\cdot)$ as $c\rightarrow +0$ by Lemma~\ref{lem:unif:conv}. With this in mind, (\ref{int_limit2}) is transformed into
 \begin{equation}\label{int_limit3}
 \int \frac{\rho_{c_{k_\ell}}\{d(X, m)\}}{2c_{k_\ell}}dP \ge \int \frac{\rho_{c_{k_\ell}}\{d(X, p_{\ell})\}}{2c_{k_\ell}} dP.
\end{equation} 
For each $x \in M$, observe that $$\frac{\rho_{c_{k_\ell}}\{d(x, m)\}}{2c_{k_{\ell}}} \le d(x, m) ~ \mbox{for all} ~ \ell, \quad \mbox{and} \quad \frac{\rho_{c_{k_\ell}}\{d(x, p_{\ell})\}}{2c_{k_{\ell}}} \le d(x, p) + 1$$ 
for any sufficiently large $\ell$. Hence, we can apply Lebesgue's dominated convergence theorem to (\ref{int_limit3}). By taking $\lim_{\ell \rightarrow \infty}$ to (\ref{int_limit3}), we obtain that, for any $m \in M$  
\[
\int d(X, m) dP \ge \int d(X, p) dP.
\]
Hence we get $p \in E^{(0)}$, thereby leading to (\ref{outer_limit2}). It completes the proof.
\end{proof}
\end{theorem}

The sample version of Theorem~\ref{thm:set:conv} immediately follows by replacing $P_{X}$ with $P_{n}$. Given $n$ deterministic observations, $(x_{1}, x_{2}, ..., x_{n}) \in M^{n}$, we denote $E^{(c)}_{n}$ by $E^{(c)}_{n}(x_{1}, x_{2}, ..., x_{n})$ for clarity.
\begin{corollary}\label{cor:conv:sam}
Let $n$ be a fixed positive integer. $E_{n}^{(c)}(x_{1}, x_{2}, ..., x_{n})$ converges to $E_{n}^{(0)}(x_{1}, x_{2}, ..., x_{n})$ (or $E^{(\infty)}_{n}(x_{1}, x_{2}, ..., x_{n})$) in outer limit as $c \rightarrow +0$ (or $c \rightarrow \infty$, respectively).
\end{corollary}

By Theorem~\ref{thm:set:conv} and Corollary~\ref{cor:conv:sam}, the Huber mean bridges the Fr\'{e}chet mean and geometric median in the sense that each of the Huber means for $c$ becomes closer to the geometric median (and Fr\'{e}chet mean) as $c\rightarrow +0 ~ (\mbox{or}~ c \rightarrow \infty$, respectively). Note that Theorem~\ref{thm:set:conv} and Corollary~\ref{cor:conv:sam} are written in terms of the limit of sets, defined in Definition~\ref{def_outer_limit}. If the Huber means are unique, then the results can be described more familiarly. To this end, suppose that, for all $c\in [0, \infty]$, each of $E^{(c)}$ is singleton. For each $c \in [0, \infty]$, we denote $E^{(c)}=\{m^{c}_{0} \}$. Now, the following corollary of Theorem~\ref{thm:set:conv} is written in the way of sequences of points on $M$. 
\begin{corollary}\label{cor:meaning:huber}
\begin{itemize}
\item[(a)] Suppose that there exist $p_0 \in M$ and $r > 0$ such that $\textsf{supp}(P_{X}) \subseteq B_{r}(p_0)$ and $r \le \frac{1}{2}\min\{\frac{\pi}{2\sqrt{\Delta}},r_{\mbox{\tiny inj}}(M)\}$ (which may be infinite), and that $P_{X}$ does not degenerate to any single geodesic. If $(c_{k})_{k \ge 1}$ is a nonnegative sequence satisfying $c_{k} \underset{k \rightarrow \infty}{\rightarrow} \infty$, then $m^{c_k}_{0} \underset{k \rightarrow \infty}{\rightarrow} m^{\infty}_0$. If $(c_{k})_{k \ge 1}$ is a nonnegative sequence satisfying $c_{k} \underset{k \rightarrow \infty}{\rightarrow} 0$, then $m^{c_{k}}_{0} \underset{k \rightarrow \infty}{\rightarrow} m^{0}_0$.
\item[(b)] Let $n$ be a fixed positive integer. Given an $n$-tuple of deterministic observations $(x_{1}, x_{2}, ..., x_{n}) \in M^{n}$, suppose that there exist $p_0 \in M$ and $r > 0$ such that $x_{1}, x_{2}, ..., x_{n}$ lie in $B_{r}(p_0)$ and $r \le \frac{1}{2}\min\{\frac{\pi}{2\sqrt{\Delta}}, r_{\mbox{\tiny inj}}(M)\}$ (which may be infinite), and that $x_{1}, x_{2}, ..., x_{n}$ do not lie in any single geodesic. If $(c_{k})_{k \ge 1}$ is a nonnegative sequence satisfying $c_{k} \underset{k \rightarrow \infty}{\rightarrow} \infty$, then $m^{c_k}_{n} \underset{k \rightarrow \infty}{\rightarrow} m^{\infty}_n$. If $(c_{k})_{k \ge 1}$ is a nonnegative sequence satisfying $c_{k} \underset{k \rightarrow \infty}{\rightarrow} 0$, then $m^{c_k}_{n} \underset{k \rightarrow \infty}{\rightarrow} m^{0}_n$.
\end{itemize}
\end{corollary}
\begin{proof}[Proof of Corollary\ref{cor:meaning:huber}] We begin with a proof for Part(a). By Lemma~\ref{lem:contain:ball}, $m^{\infty}_{0}$ and $m^{c_{k}}_{0}$'s entirely lie in $B_{r}(p_0)$ where $p_{0} \in M$ and $r > 0$ are specified in Assumption (A2). By Theorem~\ref{thm:set:conv}, (\ref{outer_limit}) holds. Since $E^{(\infty)}$ and $E^{(c_{k})}$'s are singleton, $\cap_{n=1}^{\infty} \overline{\cup_{k=n}^{\infty} \{m^{c_{k}}_{0} \}} \subseteq \{m^{\infty}_{0}\}$. As $m^{c_{k}}_{0}$'s lie in $\overline{B_{r}(p_0)}$ which is sequentially compact by Hopf-Rinow theorem, there exists an accumulation point of $(m^{c_{k}}_{0})_{k \ge 1}$ in $\overline{B_{r}(p_0)}$, thereby meaning that $\phi \neq \cap_{n=1}^{\infty} \overline{\cup_{k=n}^{\infty} \{m^{c_{k}}_{0} \}} = \{m^{\infty}_{0}\}$. Due to Lemma~\ref{lem:subseq}, every convergent subsequence of $(m_{0}^{c_{k}})_{k \ge 1}$ converges to $m^{\infty}_{0}$. Since $(m_{0}^{c_{k}})_{k \ge 1}$ is bounded, it holds that $m^{c_{k}}_{0} \underset{k \rightarrow \infty}{\rightarrow} m^{\infty}_{0}$. In the same way, for any nonnegative sequence satisfying $c_{k} \rightarrow 0$, an argument similar to above leads that $m^{c_{k}}_{0} \underset{k \rightarrow \infty}{\rightarrow} m^{0}_0$.

Proof for Part(b). An argument similar to Part(a) together with Corollary~\ref{cor:conv:sam} leads to the desired result. 
\end{proof}

Before closing this section, by combining Corollary~\ref{cor:meaning:huber}, Theorem~\ref{thm:robust:breakdown}, and Theorem~\ref{thm:breakdown:0.5},  we remark that the (pseudo) Huber mean ``continuously" bridges the $L_{1}$- and $L_{2}$ centers of mass while preserving the highest robustness (on Riemannian symmetric spaces). As a result, the Huber mean has attractive properties of the $L_{1}$- and $L_{2}$ centers evenly. For example, the sample Huber mean is highly efficient and robust to outliers at the same time (Proposition~\ref{prop:c:worthwhile:supp} and Theorem~\ref{thm:robust:breakdown}); the pseudo-Huber mean can be computed by an algorithm that has a linear rate of convergence.

\subsection{Conditions for ensuring the convergence of estimating algorithm}\label{sec:conv:alg}

To ensure the convergence of Algorithm~\ref{alg:Huber}, a recent result from the literature of optimizations on Riemannian manifolds \citep{ferreira2019gradient} can be used. We do this for the case where the pseudo-Huber loss for $c \in (0, \infty)$ is used. 
The Lipshitz continuity of the gradient of the objective function $\tilde{F}_{n}^{c}$ plays an important role in ensuring the convergence. We say the \textit{gradient Lipshitz continuity} of $\tilde{F}_{n}^{c}$ is satisfied on a strongly convex set $M_1 ~ (\subset M)$ if the following condition holds for some constant $L < \infty$, for all $m_{1}, m_{2} \in M_1$:  
\begin{equation}\label{eq:Lipshitz:grad}
\|\textsf{grad} \tilde{F}_{n}^{c}(m_{1})- \Gamma_{m_{2} \rightarrow m_{1}}\{\textsf{grad}\tilde{F}_{n}^{c}(m_{2})\}\|_{m_1} \le L \cdot d(m_{1}, m_{2}), 
\end{equation}
where $\Gamma_{m_{2} \rightarrow m_{1}}$ denotes the parallel transport from $T_{m_{2}}M$ to $T_{m_{1}}M$ along the length-minimizing geodesic from $m_{2}$ to $m_{1}$. 
To guarantee (\ref{eq:Lipshitz:grad}), we suppose that sectional curvatures of $M$ are lower bounded by some finite constant $\delta ~ (> -\infty)$, as assumed in Section~\ref{sec:setup}. 
Since the pseudo-Huber loss is twice-continuously differentiable, the Riemannian Hessian operator of $\tilde{F}^{c}_{n}$ at $p\in M$, $\textsf{Hess}\tilde{F}^{c}_{n}(p)[\cdot]: T_{p}M \rightarrow T_{p}M, v \mapsto \textsf{Hess}\tilde{F}^{c}_{n}(p)[v] ~ (= \nabla_{v} \textsf{grad}\tilde{F}^{c}_{n}(p))$, is well-defined for any $p \in M$. Thus, the map 
$p \mapsto \|\textsf{Hess}\tilde{F}_{n}^{c}(p)\|_{\textrm{op}}$, where $\|\cdot\|_{\textrm{op}}$ denotes the operator norm, is continuous. 

We suppose that the data points $x_{1}, x_{2}, ..., x_{n}$ are contained in an open ball of finite radius $r' \le \tfrac{1}{2}\min\{\tfrac{\pi}{\sqrt{\Delta}}, r_{\mbox{\tiny inj}}(M) \}$, denoted by $B_{r'}(p_0)$ for some $p_0 \in M$. Then, for $L_{c}:= \sup_{p \in \overline{B_{r'}(p_0)}} \|\textsf{Hess}\tilde{F}_{n}^{c}(p)\|_{\textrm{op}} = \max_{p \in \overline{B_{r'}(p_0)}} \|\textsf{Hess}\tilde{F}_{n}^{c}(p)\|_{\textrm{op}}$, we have $L_{c} < \infty$ as $\overline{B_{r'}(p_0)}$ is compact by Hopf-Rinow theorem. Consequently, under the same assumption, the gradient Lipshitz continuity (\ref{eq:Lipshitz:grad}) on $B_{r'}(p_0)$ holds for $L=L_{c}$, to which we apply the result of \cite{ferreira2019gradient}, to obtain the following.

\begin{proposition}[Convergence of Algorithm~\ref{alg:Huber}]\label{prop:alg:conv}
For a given $c \in (0,\infty)$, suppose that $\tilde{\rho}_{c}$ is used. Given an $n$-tuple of deterministic observations $(x_{1}, x_{2}, ..., x_{n}) \in M^{n}$, suppose that $x_{1}, x_{2}, ..., x_{n}$ lie in $B_{r'}(p_0)$, and that for any iteration $k \ge 1$ of Algorithm~\ref{alg:Huber}, $m_{k}$ lies in $B_{r'}(p_0)$ where $r'$, $p_0$ are specified above. If $\alpha \in (0, 1/L_{c}]$, then $m_{k} \underset{k\rightarrow \infty}{\rightarrow} m^{c}_{n}$, and the convergence rate is linear.

\begin{proof}[Proof of Proposition~\ref{prop:alg:conv}]
For any $m_{1}, m_{2} \in \overline{B_{r'}(p_0)}$, by the mean-value theorem, we have
\[
\|\frac{\textsf{grad}\tilde{F}_{n}^{c}(m_{1}) - \Gamma_{m_{2} \to m_{1}}\{\textsf{grad} \tilde{F}_{n}^{c}(m_{2})\}}{d(m_{1}, m_{2})}\|_{m_1} \le \sup_{p \in \overline{B_{r'}(p_0)}} \|\textsf{Hess}\tilde{F}_{n}^{c}(p)\|_{\textrm{op}} = L_{c},  
\]
where the norm on the left denotes the Euclidean norm in $T_{x}M \cong \mathbb{R}^{k}$ and the norm on the right, $\|\cdot\|_{\textrm{op}}$, stands for the operator norm. Hence, the gradient Lipshitz continuity of $\tilde{F}_{n}^{c}$ holds on $B_{r'}(p_0)$. By applying \citep[Theorem 3,][]{ferreira2019gradient},\footnote{This is the exact point where the lower boundedness of sectional curvatures ($\delta > -\infty$) is required.} we obtain that $m_{k} \underset{k \rightarrow \infty}{\rightarrow} m^{c}_{n}$ and $\tilde{F}_{n}^{c}(m_{k}) - \tilde{F}_{n}^{c}(m^{c}_{n}) = O(1/k)$. It means that the iteration complexity is $O(1/\epsilon)$ to achieve $\tilde{F}_{n}^{c}(m_{k}) - \tilde{F}_{n}^{c}(m^{c}_{n}) < \epsilon$. So, the convergence rate is linear.
\end{proof} 
\end{proposition}

We remark that if the gradient Lipshitz continuity (\ref{eq:Lipshitz:grad}) is true for $F_n^c$, then the statement of Proposition~\ref{prop:alg:conv} is true for the case where Huber loss is used. While we conjecture that (\ref{eq:Lipshitz:grad}) holds for $F_n^c$, it has been difficult to verify as $F_n^c$ is not twice differentiable. Proposition~\ref{prop:alg:conv} indicates that extra care is required to ensure the convergence of Algorithm~\ref{alg:Huber}. The proposition specifies that the step size $\alpha$ of Algorithm~\ref{alg:Huber} should be chosen such that $\alpha \in (0, 1/L_{c}]$. Since evaluating $L_{c}$ is generally challenging, we recommend using a smaller $\alpha$ for smaller values of $c$. Additionally, $\alpha$ should be sufficiently small to ensure that $m_{k}$ remains within $B_{r'}(p_0)$ for each iteration $k$.

Note that Algorithm~\ref{alg:Huber} employs a fixed step size $\alpha$. For manifolds with lower-bounded sectional curvature, utilizing a small, but fixed step size is sufficient in guaranteeing the convergence of the gradient descent algorithm. In cases where the manifold's sectional curvatures are not lower-bounded, one may modify Algorithm~\ref{alg:Huber} by employing varying step sizes, specifically the Armijo step sizes as used in Algorithm 3.1 of \cite{wang2021convergence}. With this modified algorithm, it follows that $m_{k} \underset{k \rightarrow \infty}{\rightarrow} m^{c}_{n}$ and a linear rate of convergence is achieved under some technical conditions \citep[see Theorem 3.6,][]{wang2021convergence}.

\section{Technical details, additional lemmas, corollaries, and propositions}\label{appendix:B}

\subsection{The regularity conditions (A4) for central limit theorem}

To verify the central limit theorem for sample Huber means (Theorem~\ref{thm:clt} in the main article Section~\ref{sec:asymptotics}), we assume the following regularity conditions. 

\underline{(A4)} The following hold for $P_X$.
 \begin{itemize}
     \item[(i)] $H_{c}(\mathbf{x})$ is continuous at $\mathbf{x}=\mathbf{0}$.
     \item[(ii)] The map $\mathbf{x} \mapsto \rho_{c}\{d(X, \phi^{-1}_{m^{c}_{0}}(\mathbf{x}))\}$
 is twice-continuously differentiable at $\mathbf{x}=\mathbf{0}$ with probability 1.
     \item[(iii)] Moreover, the 
 \textit{integrability condition} is satisfied: 
\begin{equation}\label{integrability1}
    \begin{cases} E\{\textsf{grad}d^2(X, \phi_{m_{0}^{c}}^{-1}(\mathbf{0}))\} ~ \mbox{exists},  \\
 \textsf{Cov}\{\textsf{grad}d^2(X, \phi_{m_{0}^{c}}^{-1}(\mathbf{0}))\} ~ \mbox{exists, and} \\
 E\{\textsf{Hess}d^2(X, \phi_{m_{0}^{c}}^{-1}(\mathbf{x}))\} ~ \mbox{exists for $\mathbf{x}$ near $\mathbf{0}$ \mbox{and is non-singular at} $\mathbf{x}=\mathbf{0}$}.
  \end{cases}
\end{equation}
 \end{itemize}
%   (i) $H_{c}(\mathbf{x})$ is continuous at $\mathbf{x}=\mathbf{0}$, and (ii) the map $\mathbf{x} \mapsto \rho_{c}\{d(X, \phi^{-1}_{m^{c}_{0}}(\mathbf{x}))\}$
%  is twice-continuously differentiable at $\mathbf{x}=\mathbf{0}$ with probability 1. (iii) Moreover, the 
%  \textit{integrability condition} is satisfied:
% %\begin{eqnarray}\label{integrability1}
% %&&
% \begin{equation}\label{integrability1}
%     \begin{cases} E\{\textsf{grad}d^2(X, \phi_{m_{0}^{c}}^{-1}(\mathbf{0}))\} ~ \mbox{exists},  \\
%  \textsf{Cov}\{\textsf{grad}d^2(X, \phi_{m_{0}^{c}}^{-1}(\mathbf{0}))\} ~ \mbox{exists, and} \\
%  E\{\textbf{H}d^2(X, \phi_{m_{0}^{c}}^{-1}(\mathbf{x}))\} ~ \mbox{exists for $\mathbf{x}$ near $\mathbf{0}$}.
%   \end{cases}
% \end{equation}

 We conjecture that for general manifold, the assumption of $P_X$ being absolutely continuous (together with unimodal density in the case of manifolds with positive curvature) and having a finite variance guarantees Assumption (A4), but we have only partial proof for the conjecture.

 An example under which the regularity conditions (A4) are satisfied is given below. 
 
 \begin{proposition}\label{example:regularity}
     For the case where $M = \mathbb{R}^k$, Assumption (A4) is satisfied if $P_X$ is absolutely continuous with respect to the Lebesgue measure on $\bbR^{k}$ with finite second moment (i.e., $E\| X \|^2 < \infty$). 
 \end{proposition} 

\begin{proof}[Proof of Proposition~\ref{example:regularity}]
    
Assume that the distribution of $\bbR^k$-valued random variable $X$, $P_X$, is absolutely continuous, and satisfies $\int_{\bbR^{k}} \|\mathbf{x}\| dP_{X}(\mathbf{x}) = E\|X\| < \infty$. Let $c \in (0,\infty)$ be fixed. 

(i) Note that $$H_c(\xv) = 2 E I_{k} \cdot 1_{\|X -\xv\| \le c} + E \left\{ \frac{2c}{\|X -\xv\|} \left( I_k -  \frac{(X -\xv ) (X -\xv )^T}{\|X -\xv\|^2} \right) \cdot 1_{\|X -\xv\| > c}
\right\}.$$ Since $P_X$ is absolutely continuous, as $\xv \to 0$
$$ 
E 1_{\|X -\xv\| \le c} = P(\|X -\xv\| \le c) \to  P(\|X\| \le c).
$$
Moreover, the integrand of the second term, $\frac{2c}{\|X -\xv\|} \left( I_k -  \frac{(X -\xv ) (X -\xv )^T}{\|X -\xv\|^2} \right)$ is continuous (with repect to $\xv$) and it is absolutely bounded by $2I_k$. Thus, by dominated convergence theorem, together with the fact that $P_X$ is absolutely continuous, the second term, as a function $\xv$ is continuous at $\xv = 0$. Thus, $H_c(\xv)$ is continuous at $\xv$.

(ii) Since the set of points at which the function is not $C^2$ is measure zero (cf. Sard's theorem), and $X$ is absolutely continuous, the function is $C^2$ with probability 1. 

(iii) Suppose that $M=\bbR^{k}$ with the Riemannian distance $d(\mathbf{x}, \mathbf{y}) = \|\mathbf{x}-\mathbf{y}\|$ for any $\mathbf{x}, \mathbf{y} \in \mathbb{R}^{k}$, where $\|\cdot\|$ refers to the standard Euclidean norm. Hence, note that $E\{\textsf{grad} d^2(X, \phi^{-1}_{m^{c}_{0}}(\mathbf{0})) \} = E\{\textsf{grad}|_{\mathbf{x}=\mathbf{0}} \|X - \mathbf{x} \|^2\} = E \{2(\mathbf{x} - X)|_{\mathbf{x} = \mathbf{0}}\} = -2EX$ exists because each component of $EX$ is not larger than $E\|X\| ~(\le (E\|X\|^2)^{1/2})$. Similarly, $\textsf{Cov}\{\textsf{grad} d^2(X, \phi^{-1}_{m^{c}_0}(\mathbf{0})) \} = \textsf{Cov}\{\textsf{grad}|_{\mathbf{x}= \mathbf{0}} \|X - \mathbf{x}\|^2\} = \textsf{Cov}\{2(\mathbf{x} - X)|_{\mathbf{x} = \mathbf{0}}\} = 4\textsf{Cov}(X)$ exists owing to $E\|X\|^2 < \infty$. Finally, we obtain that $E\{\textsf{Hess} d^2(X, \phi^{-1}_{m^{c}_0}(\mathbf{x})) \} = E\{\textsf{Hess} \|X - \mathbf{x} \|^2 \} = 2I_{k}$ exists for $\mathbf{x}$ near $\mathbf{0}$, and becomes non-singular at $\mathbf{x}=\mathbf{0}$. Therefore, the integrability condition (\ref{integrability1}) is valid. 

\end{proof}

\subsection{Proofs for Section \ref{sec:exist:unique}}

\subsubsection{Proofs of Theorem \ref{thm:exist} and its corollary}

We begin with establishing a type of \textit{coercivity} property of the objective function $F^{c}$ with respect to $P_{X}$. 

\begin{proposition}\label{prop:cpt}
For a given $c \in [0, \infty]$, suppose that $P_{X}$ satisfies Assumption (A1). There exists a compact set $K_{c} \subseteq M$ such that 
\begin{eqnarray}\label{coercive}
\centering
&& \inf_{m \in K_{c}} F^{c}(m) = \inf_{m \in M} F^{c}(m), ~ \mbox{and} ~ \nonumber \\
&& \mbox{if} ~ K_{c} \subsetneq M, ~  \inf_{m \in K_{c}} F^{c}(m) < \inf_{m \in M \setminus K_{c}} F^{c}(m). 
\end{eqnarray}
Therefore, $E^{(c)}\subset K_{c}$ and $E^{(c)}$ is compact.
\begin{proof}[Proof of Proposition~\ref{prop:cpt}]

The argument used below works for any form of the objective function $m \mapsto \int \rho \{ d(X,m)\} dP$, provided that $\rho: [0,\infty) \to [0,\infty)$ is convex, strictly increasing, and the value of the objective function is finite for any $m \in M$. Thus, the statement of Proposition~\ref{prop:cpt} is valid when the Huber mean is replaced by the pseudo-Huber mean.

We note first that the objective function $F^c(\cdot)$ is continuous on $M$. Thus the set of minimizers $E^{(c)}$ is closed. 
We divide the proof according to whether $M$ is bounded or not.

\textbf{Case I ($M$ is bounded).} Let $K_{c}= M$. Then, $K_{c}$ is compact by the Hopf-Rinow theorem. Since $M$ is connected and complete, every closed and bounded subset of $M$ is compact. In particular, $E^{(c)} \subset K_{c} = M$ is bounded, and is thus compact as asserted.    

\textbf{Case II ($M$ is unbounded).} We shall verify that $\{F^{c}(m): m \in M\}$ is not a singleton. Suppose that, there exists $C \ge 0$ such that $F^{c}(m)=C$ for every $m \in M$. Since $\rho_{c}:[0, \infty) \rightarrow [0, \infty)$ is  convex, it holds that $\rho_{c}(x) + \rho_{c}(y) \ge 2\rho_{c}(\frac{x+y}{2})$ for any $x, y \ge 0$. Due to the preceding inequality, for any $m_{1}, m_{2} \in M$
\begin{eqnarray*}
2C=F^{c}(m_{1})+F^{c}(m_2) &=& \int \rho_{c}\{d(X, m_{1})\} + \rho_{c}\{d(X, m_{2})\} dP \\
&\ge& 2 \int \rho_{c} \left\{\frac{d(X, m_{1}) + d(X, m_{2})}{2}\right\} dP \\
&\ge& 2 \int \rho_{c} \{ d(m_{1}, m_{2})/2 \}dP \\
&=& 2 \rho_{c} \{d(m_{1}, m_{2})/2\}.
\end{eqnarray*}
Accordingly, we get $C \ge \rho_{c}\{d(m_{1}, m_{2})/2\}$, which is a contradiction since $m_{1}, m_{2}$ can be chosen so that $\rho_{c}\{d(m_{1}, m_{2})/2\} > C$. Consequently $\{F^{c}(m): m\in M\}$ is not singleton. 

Since $\{F^{c}(m): m \in M \}$ has at least two elements, there exists $m_{1} \in M$ such that $F^{c}(m_{1}) > \inf_{m \in M} F^{c}(m) \ge 0$. Since $M=\cup_{r \ge 0} B_{r}(m_{1})$ and $P(X \in M)=1$, there exists $r_{1} > 0$ such that $P(X \in B_{r_{1}}(m_{1})) > 0$. Let $K':=B_{r_{1}}(m_1)$. Since $\frac{\inf_{m \in K'} F^{c}(m)}{P(X \in K')} \ge 0$ is satisfied and $\rho_{c}(\cdot)$ is strictly increasing, there exists $r_{2} ~(> r_{1})$ such that for any $r \ge r_{2}$, $$\rho_{c}(r - r_{1}) > \frac{\inf_{m \in K'}F^{c}(m)}{P(X \in K')}$$ holds. Then, $K'':=\overline{B_{r_{2}}(m_{1})} ~(\supset K')$ is compact due to the Hopf-Rinow theorem. Due to the monotonicity of $\rho_{c}(\cdot)$, for any $m_{2} \in M \setminus K''$
\begin{eqnarray*}
F^{c}(m_{2}) \ge \int_{K'} \rho_{c}\{d(X, m_{2})\}dP \ge \rho_{c}(r_{2} - r_{1})P(X \in K') > \inf_{m \in K'} F^{c}(m),  
\end{eqnarray*}
where the second inequality holds by $d(x, m_{2}) \ge d(m_{1}, m_{2}) - d(x, m_{1}) \ge r_{2} - r_{1}$ for any $x \in K'$. Accordingly, the minimizers of $F^{c}$ do not lie in $M\setminus K''$. Since $K' \subset K''$, the minimizers must lie in $K''$. Letting $K_{c}=K'' \subsetneq M$ gives 
\[
\inf_{m \in K_{c}} F^{c}(m) = \inf_{m \in M} F^{c}(m) < \inf_{m \in M \setminus K_{c}} F^{c}(m).
\]
We moreover obtain that $E^{(c)} \subset K_{c}$, where $E^{(c)}=\argmin_{m \in M} F^{c}(m)$. Since every closed subset of a compact set is compact, so is $E^{(c)}$.
\end{proof}
\end{proposition}

The coercivity property of $F^c$ states that any minimizer of $F^{c}$, i.e., any Huber mean, lies in a compact set. This does not guarantee that there exists a minimizer of $F^{c}$.

\begin{proof}[Proof of Theorem~\ref{thm:exist}]
For each $c \in (0, \infty)$, the (pseudo) Huber loss function $\rho_{c}$ has a Lipshitz constant $2c$. When $c=0$, moreover, $\rho_{0}$ has a Lipshitz constant $1$. To sum up, for any choice of $c \in [0, \infty)$, $\rho_{c}$ has a Lipshitz constant $2c \lor 1 := \max(2c, 1)>0$. For a given $\epsilon > 0$, suppose that $m_{1}, m_{2} \in M$ and $d(m_{1}, m_{2}) < \epsilon/(2c \lor 1)$. Then, 
\begin{eqnarray*}
|F^{c}(m_1) - F^{c}(m_2)| &=& \int |\rho_{c}\{d(X, m_{1})\} - \rho_{c}\{d(X, m_{2})\}| dP \\
 &\le& (2c \lor 1) \int |d(X, m_1) - d(X, m_2)|dP \\
 &\le& (2c \lor 1) \int d(m_{1}, m_{2})dP \\
 & < & (2c \lor 1) \cdot \epsilon/(2c \lor 1) \\
 &=& \epsilon. 
\end{eqnarray*}
Hence, $F^{c}$ is (uniformly) continuous when $c \in [0, \infty)$. In the case of $c=\infty$, %let $m \in M$ and $(m_{k})_{k \ge 1}$ be a sequence of points on $M$ that converges to $m$ (i.e., $d(m_{k}, m) \underset{k \rightarrow \infty}{\rightarrow} 0$). For each $x \in M$, due to the continuity of $d(x, \cdot): M \rightarrow [0, \infty)$, $d^2(x, m_{k}) \underset{k \rightarrow \infty}{\rightarrow} d^2(x, m)$. For each $x \in M$ and any sufficiently large $k$, $|d^2(x, m_{k}) - d^2(x, m)| \le d^2(x, m_{k}) + d^2(x, m) \le 2d^2(x, m) + 1$. By Lebesgue's dominated convergence theorem, $\lim_{k \rightarrow \infty}\int d^2(X, m_{k}) dP = \int d^2(X, m) dP$, leading to $F^{\infty}(m_{k}) \rightarrow F^{\infty}(m)$; $F^{\infty}$ is continuous at $m\in M$. Since $m \in M$ was arbitrary, $F^{\infty}$ is continuous on $M$. 
it is not difficult to show that $F^{\infty}$ is continuous on $M$, which is available upon request. Consequently, $F^{c}$ is continuous for any $c \in [0, \infty]$. 

By Proposition~\ref{prop:cpt}, there exists a compact $K_{c} \subseteq M$ such that (\ref{coercive}) holds. Since $F^{c}$ is continuous, $F^{c}(K_{c}):=\{F^{c}(m) : m \in K_{c} \} \subset M$ is compact, and thus the minimum of $F^{c}$ is attained in $K_{c}$. Therefore, the population Huber mean exists (i.e., $E^{(c)} \neq \phi$). 

\end{proof}

For any data set with finite sample size $n$, the sample (pseudo) Huber mean exists. This is observed by substituting $P_{X}$ in Theorem~\ref{thm:exist} with $P_{n}$.

\begin{corollary}[Existence of the sample Huber means]\label{cor:exist:sam}
 Let $c \in [0, \infty]$ and $n$ be any positive integer. Given $n$ observations $(x_{1}, x_{2}, ..., x_{n}) \in M^{n}$, the sample Huber mean exists; that is, $E_{n}^{(c)}(x_{1}, x_{2}, ..., x_{n}) \neq \phi$.
\begin{proof}[Proof of Corollary~\ref{cor:exist:sam}]
By replacing $P_{X}$ in Theorem~\ref{thm:exist} by $P_{n} ~(=\frac{1}{n}\sum_{i=1}^{n} \delta_{x_{i}})$, the sample Huber mean exists; that is, $E_{n}^{(c)}(x_{1}, x_{2}, ..., x_{n}) \neq \phi$ for any finite $n$ (note that $P_{n}$ clearly satisfies Assumption (A1)).
\end{proof}
\end{corollary}

Before closing this subsection, we remark that Assumption (A1) is equivalent to the condition that $\int \rho_{c}\{d(X, m)\}dP < \infty$ for \textit{any} $m \in M$. If $P_X$ has a finite variance, i.e., $\int d^2(X, m)dP < \infty$ for an $m\in M$, then (A1) holds for any $c \in [0, \infty]$. %If the pseudo-Huber mean is of interest, (A1) is interpreted as $\tilde{F}^c(m) = \int \tilde{\rho}_{c}\{d(X, m)\}dP < \infty$ for some $m \in M$. 
For a more relaxed condition on $P_{X}$, Assumption (A1) can be replaced by $\int [\rho_{c}\{d(X, m)\} - \rho_c\{d(X, m_0)\}]dP < \infty$ for some $m, m_0\in M$. This enables us to define the Huber mean even when $P_{X}$ is highly heavy-tailed. In this case, $E^{(c)}$ in Definition~\ref{def:Huber} should be defined as $\mbox{argmin}_{m \in M} \int [\rho_c\{d(X, m)\} - \rho_c\{d(X, m_0)\}]dP$, as used in \cite{sturm2003probability, schotz2022strong, schotz2025variance}. For ease of discussion, however, we have left it in the form of (A1).

\subsubsection{Proofs of Theorem \ref{thm:unique:pop} and its corollary}

% We now explain the somewhat unusual requirement (ii) of Part(a) in \Cref{thm:unique:pop}. This special case, where $P_X$ degenerates to a single geodesic, corresponds to the typical situation of scalar-valued Huber means. For this scenario, the requirement of non-negligible density or mass near any $m_0 \in E^{(c)}$ is needed to prevent situations where Huber loss function $F^c$ is \textit{flat} at any Huber mean. An illustrative example of a situation where requirement (ii) is violated is given in Figure \ref{fig:unique_c}. In this example, there exists a closed geodesic segment such that every element in the segment is Huber mean. This occurs due to the presence of $d(x,m)$ in the loss function, whose sum or integral might be flat along the geodesic segment. Requirement (ii) is imposed to prevent such situations. We note that, for $c > 0$, the condition can be simplified to $P(d(X, m_0) \le c) > 0$ for any $m_0 \in E^{(c)}$. 

% \begin{figure}[t]
% \centering
% \includegraphics[scale=0.4]{figures/unique_c.png}
% \caption{An example of non-unique Huber means for which the requirement (ii) of Theorem~\ref{thm:unique:pop} is violated. (Left) $P_{X}$ is equally distributed on four black dots on a geodesic $\gamma_X$. (Right) The graph of $F^{c}$ parametrized by the arc length of $\gamma_{X}$. In this situation, there is a point $m \in E^{(c)} \subset \gamma_{X}$ near the center of $\gamma_{X}$ such that $P(d(X, m) \le c) = 0$. The minimum of $F^c$ occurs at every point of a closed segment within $\gamma_X$.
% }
% \label{fig:unique_c}
% \end{figure}

The following lemma is used to ensure that the (pseudo) Huber mean set lies in a strongly convex ball. Lemma~\ref{lem:contain:ball} is stated for a general loss function $\rho$, and both the Huber loss function $\rho_c$ and the pseudo-Huber loss function $\tilde{\rho}_{c}$ for any $c \in [0,\infty]$ satisfy the required conditions.

\begin{lemma}\label{lem:contain:ball}
Let $\rho: [0, \infty) \to [0, \infty)$ be any continuous and strictly increasing function. Suppose that $P_{X}$ satisfies 
$\int \rho\{d(X,m)\} dP < \infty$
and that
$\textsf{supp}(P_{X}) \subseteq B_{r_0}(p_0),$ 
where $\textsf{supp}(P_{X})$ denotes the support of $P_{X}$,  and $r_0 = \frac{1}{2}\min\{\frac{\pi}{\sqrt{\Delta}}  , r_{\mbox{\tiny inj}}(M) \}$. (Recall that $1/\sqrt{\Delta}$ is treated as $\infty$ if $\Delta \le 0$.)
Then, there exist $p_0 \in M$ and a finite $0 <r \le r_0$ such that $E_{\rho} \subset B_r(p_0)$, where 
$E_{\rho} = \argmin_{m \in M} \int \rho\{d(X,m)\} dP$.
\end{lemma}

\begin{proof}[Proof of Lemma~\ref{lem:contain:ball}]

  We divide the proof into two situations depending on the value of $r_0$: Case I for $r_0 = \infty$ and Case II for $r_0 <\infty$. Let $F_\rho: M \to \mathbb{R}$ be defined as $F_\rho(m) = \int \rho\{d(X,m)\} dP$.

\textbf{Case I ($r_0 = \infty$)}. In this case, Assumption (A2) is not useful as it allows unbounded support for the distribution: $\textsf{supp}(P_{X}) \subseteq M$. On the other hand, since the objective function is integrable, there exists a compact set $K \subset M$ such that $E_{\rho} \subset  K = \overline{B_{r'}(p_0)}$ for some $r' <\infty$ and $p_0 \in M$. (See the proof of Proposition \ref{prop:cpt}.) Pick a small $\epsilon > 0$, and let $r = r' + \epsilon$. Then, $E_{\rho} \subset K \subset B_r(p_0)$, as required. 

\textbf{Case II ($r_0 <\infty$)}. Assumption (A2) leads that $\textsf{supp}(P_{X}) \subseteq B_{r_0}(p_0)$, where $r_0$ is given in (\ref{eq:r_0}) and $p_0$ is specified in the assumption. Set $r = r_0$. 
%In this case, $p_0$ in Assumption (A2) is strategically chosen ``near the center of $P_X$" and $r = r_0$ is the radius of the support. We still know from \Cref{prop:cpt} that there exist an $r' \in(0,\infty)$ and $p' \in M$ such that $E \subset B_{r'}(p')$, but it is very possible that $p' \neq p_0$ and $r < r'$. 
The rest of the proof closely follows the argument used for Theorem 2.1 of \cite{afsari2011riemannian}.

Firstly, we shall show that $E_{\rho} \subset B_{2r}(p_0)$. Suppose this is not true. There exists $m_0 \in E$ such that $d(p_0, m_0) \ge 2r$. In view of the triangle inequality, for any $x \in \textsf{supp}(P_{X})$, $d(p_0, x) < r \le d(x, m_0)$, which implies $\rho\{d(p_0, x)\} < \rho\{d(x, m_0)\}$ for any $x \in \textsf{supp}(P_{X})$. Accordingly, $F_\rho(p_0) < F_\rho(m_0)$, which is a contradiction. 

%\underline{Step II.} 
Secondly, we shall prove that $E_{\rho} \subset \overline{B_{r}(p_0)}$. Suppose that there exists $m_0 \in E_{\rho}$ such that $r < d(m_0, p_0) < 2r$. Let $m'$ be the intersection of the (length-minimizing) geodesic, joining $p_{0}$ and $m_0$, and the boundary of $B_{r}(p_0)$. Let $m''$ be the reflection point of $m_0$ with respect to $m'$ across the boundary of $B_{r}(p_0)$. Then, it holds that $m'' \in B_{r}(p_0)$ and $d(m_0, m')=d(m', m'')$. As shown in Section 2.2.1 of \cite{afsari2011riemannian}, for any $x \in \textsf{supp}(P_{X}) \subseteq B_{r}(p_0)$, $d(m'', x) < d(m_{0}, x)$. (The key idea is an application of the Toponogov comparison theorem,\footnote{The Toponogov comparison theorem provides information on the geodesic deviation in terms of sectional curvatures. This is the exact point where the upper boundedness of sectional curvatures ($\Delta < \infty$) is needed.} for details see, \citep[page. 420,][]{chavel2006riemannian} 
or \citep[Theorem 2.2,][]{afsari2011riemannian}.) Consequently, $\rho\{d(m'', x)\} < \rho\{d(m_{0}, x)\}$ for any $x \in \textsf{supp}(P_{X})$, and thus $F_\rho(m'') < F_\rho(m_{0})$, which is a contradiction.

%\underline{Step III.} 
Thirdly, we shall verify that $E_{\rho} \subset B_{r}(p_0)$. Suppose that there exists $m_0 \in E_{\rho}$ such that $m_0$ lies in the boundary of $B_{r}(p_0)$; that is, $d(p_0, m_0)=r$. For each $x \in \overline{B_{r}(p_0)}$, $\rho'\{d(x, m_0)\} >0$ and $\mbox{Log}_{m_{0}}(x)/\|\mbox{Log}_{m_{0}}(x)\|$ is an inward-pointing vector on the boundary of $B_{r}(p_0)$. Hence, $-\textsf{grad}F_\rho(m_0)= \int \rho^{'}\{d(X, m_0)\} \cdot \mbox{Log}_{m_0}(X)/\|\mbox{Log}_{m_0}(X)\|dP \neq \mathbf{0}$ is also an inward-pointing vector at the boundary of $B_{r}(p_0)$. Therefore, there exists $m \in B_{r}(p_0)$ such that $F_\rho(m) < F_\rho(m_0)$. This contradicts the assumption of the existence of a Huber mean at the boundary of $B_r(p_0)$. 
\end{proof}

It is well known that the regular convexity radius of $M$, denoted by $r_{\mbox{\tiny cx}}(M)$, is lower bounded by $\frac{1}{2} \min\{\frac{\pi}{\sqrt{\Delta}}, r_{\mbox{\tiny inj}}(M)\}$ \citep[cf.][]{afsari2013convergence}. Thus, for $r \in (0,r_0]$ in the conclusion of Lemma~\ref{lem:contain:ball}, the ball $B_r(p_0)$ is strongly convex.

\begin{proof}[Proof of Theorem~\ref{thm:unique:pop}]
%If $\textsf{supp}(P_{X})$ is singleton, the proof is obvious. Hence, assume that $\textsf{supp}(P_{X})$ is not singleton. 
For any $c \in [0,\infty]$, the conditions of Lemma~\ref{lem:contain:ball} are satisfied for either $\rho_c$ or $\tilde{\rho}_c$. Thus, the (pseudo) Huber mean set lies in $B_r(p_0)$, which is strongly convex. Here, both $r$ and $p_0$ are as specified in Lemma~\ref{lem:contain:ball}.

We begin with a proof for Part(a).

\textbf{Case I ($c \in [\tfrac{\pi}{\sqrt{\Delta}},\infty]$).} 
Since $r \le r_0 \le \frac{\pi}{2\sqrt{\Delta}}$, we have $\textsf{supp}(P_{X}) \subseteq B_{\frac{\pi}{2\sqrt{\Delta}}}(p_0)$ and any Huber mean $m \in E^{(c)}$ also lies in $B_{\frac{\pi}{2\sqrt{\Delta}}}(p_0)$. For any $x,m \in B_{\frac{\pi}{2\sqrt{\Delta}}}(p_0)$, $d(x,m) \le \pi/ \sqrt{\Delta} \le c$ and, in turn $$\rho_c\{d(x,m)\} = d^2(x, m).$$
Thus, for any $c \ge \tfrac{\pi}{\sqrt{\Delta}}$, the Huber mean set is exacly the Fr\'{e}chet mean set: $E^{(c)} = E^{(\infty)}$. Furthermore, due to the bounded support condition of $P_{X}$, Assumption (A1) leads that $\int d^2(X,m) dP =  \int \rho_c\{d(X,m)\} dP  < \infty $. Under this finite variance condition, the uniqueness of the Fr\'{e}chet mean is verified in, e.g., \citep[Theorem 2.1,][]{afsari2011riemannian} and \citep[Proposition 4.3,][]{sturm2003probability}.

\textbf{Case II ($c \in (0,\tfrac{\pi}{\sqrt{\Delta}})$).} 
Let $\gamma$ be any geodesic on $M$ such that $\gamma \cap B_r(p_0) \neq \emptyset$. (We abuse the notation for a geodesic $\gamma$, which either stands for a function from an interval to $M$ or the image of the function $\gamma$ in $M$.) Since $B_r(p_0)$ is strongly convex, there exist $-\infty\le t_0<t_1 \le \infty$ such that $\gamma((t_0,t_1)) = \gamma \cap B_r(p_0)$.
We collect several facts from \cite{karcher1977riemannian,petersen2006riemannian,fletcher2009geometric,afsari2011riemannian}. 
First of all, for any $x \in B_{r_0}(p_0)$, $\tfrac{\partial^2}{\partial t^2} d(x,\gamma(t))$ is defined for $x \neq \gamma(t)$, and is nonnegative. 
Furthermore, for $x \in B_{r_0}(p_0) \setminus \gamma$, 
$  \tfrac{\partial^2}{\partial t^2} d(x,\gamma(t)) > 0$. This can be seen from the Hessian comparison theorem\footnote{This is the exact point where the upper boundedness of sectional curvatures ($\Delta < \infty$) is required.} \citep[see, e.g.,][]{afsari2011riemannian}: Recall that $\Delta ~(< \infty)$ is the least upper bound of the sectional curvatures of $M$. Then, the Hessian comparison theorem states that for any $t \in (t_0,t_1)$, 
\begin{equation}
    \label{eq:HessianComparison} \tfrac{\partial^2}{\partial t^2} d(x,\gamma(t)) \ge \frac{\textsf{sn}_{\Delta}'(d(x,\gamma(t))}{\textsf{sn}_{\Delta}(d(x,\gamma(t))} \sin^2(\alpha),
\end{equation} 
where $\alpha$ is the angle formed at $\gamma(t)$ by the geodesic $\gamma$ and the unique length-minimizing geodesic from $x$ to $\gamma(t)$, and 
$\textsf{sn}_{\Delta}(s) = \tfrac{1}{\sqrt{\Delta} }\sin (s\sqrt{\Delta} )$ for $\Delta >0$, 
$s$ for $\Delta =0$, and $\tfrac{1}{\sqrt{|\Delta|} }\sinh (s\sqrt{|\Delta|} )$ for $\Delta <0$. It can be checked that for any $x \in B_{r_0}(p_0)$, $d(x,\gamma(t)) < \tfrac{\pi}{2\sqrt{\Delta}}$, which in turn leads that 
$$
\frac{\textsf{sn}_{\Delta}'(d(x,\gamma(t))}{\textsf{sn}_{\Delta}(d(x,\gamma(t))} > 0, 
$$
unless $d(x,\gamma(t)) = 0$. 
Therefore,  for $x \in B_{r_0}(p_0) \setminus \gamma$, the angle is non-zero (i.e., $\sin(\alpha) \neq 0$), and by (\ref{eq:HessianComparison}), we have 
$  \tfrac{\partial^2}{\partial t^2} d(x,\gamma(t)) > 0$.
On the other hand, if $x \in \gamma \cap B_{r_0}(p_0)$, but $x \neq \gamma(t)$, then 
$\sin(\alpha) = 0$, leading that 
 $  \tfrac{\partial^2}{\partial t^2} d(x,\gamma(t)) = 0$. 
Recall that  $ \tfrac{\partial^2}{\partial t^2} d^2(x,\gamma(t))> 0$ for any $x \in B_{r_0}(p_0)$ and for any $t \in (t_0,t_1)$ \citep{karcher1977riemannian}.

Observe now that for the given geodesic $\gamma$, and for $t \in (t_0,t_1)$, and $x \in B_{r_0}(p_0)$, 
\begin{align*}
    \tfrac{\partial^2}{\partial t^2}\rho_c\{d(x, \gamma(t))\} = \begin{cases}
          \tfrac{\partial^2}{\partial t^2} d^2(x, \gamma(t)) & \mbox{ if } d(x, \gamma(t)) < c, \\
        \mbox{undefined} & \mbox{ if } d(x, \gamma(t)) = c, \\
        2c  \tfrac{\partial^2}{\partial t^2} d(x, \gamma(t)) & \mbox{ if } d(x, \gamma(t)) > c. 
    \end{cases}
\end{align*}Let $x \in B_{r_0}(p_0)$ be such that $x \neq \gamma(t)$ for any $t \in \mathbb{R}$. 
Then, for $t \in (t_0,t_1)$ such that $d(x, \gamma(t)) \neq c$, we have $\tfrac{\partial^2}{\partial t^2}\rho_c\{d(x, \gamma(t))\} > 0$.
Moreover, the case $d(x, \gamma(t)) = c$ occurs for at most two values of $t \in (t_0,t_1)$. Therefore, the first derivative $t \mapsto \tfrac{\partial}{\partial t}\rho_c\{d(x, \gamma(t))\}$ is strictly increasing for $t \in (t_0,t_1)$, and the Huber loss function 
$t \mapsto \rho_c\{d(x, \gamma(t))\} $ is strictly convex on $(t_0,t_1)$.

Suppose now that $x \in B_{r_0}(p_0)$ be such that $x = \gamma(t')$ for some $t' \in (t_0,t_1)$. 
For $t \in (t_0,t_1)$ such that $d(x, \gamma(t))  < c$, $\tfrac{\partial^2}{\partial t^2}\rho_c\{d(x, \gamma(t))\} = \tfrac{\partial^2}{\partial t^2} d^2(x, \gamma(t)) > 0$. 
For $t \in (t_0,t_1)$ such that $d(x, \gamma(t))  > c$, $\tfrac{\partial^2}{\partial t^2}\rho_c\{d(x, \gamma(t))\} = \tfrac{\partial^2}{\partial t^2} d(x, \gamma(t)) = 0$. Moreover, $\tfrac{\partial^2}{\partial t^2}\rho_c\{d(x, \gamma(t))\}$ is undefined only for at most two values of $t \in (t_0,t_1)$. Therefore, the first derivative $t \mapsto \tfrac{\partial}{\partial t}\rho_c\{d(x, \gamma(t))\}$ is increasing for $t \in (t_0,t_1)$, and the Huber loss function 
$t \mapsto \rho_c\{d(x, \gamma(t))\} $ is convex on $(t_0,t_1)$. 

To sum up, for $x \in B_{r_0}(p_0)$, $a,b \in (t_0,t_1)$ and $s \in (0,1)$, 
$\rho_c\{d(x, \gamma((1-s)a + sb))\} \le (1-s)\rho_c\{d(x, \gamma(a))\} + s\rho_c\{d(x, \gamma(b))\}$ and the inequality is strict as long as $x \in B_{r_0}(p_0) \setminus \gamma$.

Assume that Condition (i) holds. For any $a,b \in (t_0, t_1)$ and for any $s \in (0,1)$, we get
\begin{align}\label{eq:convex_ineq}
    F^c&\{\gamma((1-s)a + sb)\}  = \int \rho_c\{d(X, \gamma((1-s)a + sb))\} dP \nonumber \\ 
    & = \int_{B_{r_0}(p_0)} \rho_c\{d(x, \gamma((1-s)a + sb))\} dP_{X}(x) \nonumber \\
     & = \int_{B_{r_0}(p_0) \setminus \gamma} \rho_c\{d(x, \gamma((1-s)a + sb))\} dP_{X}(x) \nonumber \\ 
     & + \int_\gamma \rho_c\{d(x, \gamma((1-s)a + sb))\} dP_{X}(x) \nonumber \\
     & < (1-s) \int_{B_{r_0}(p_0)} \rho_c\{d(x, \gamma(a))\}dP_{X}(x) 
 +  s \int_{B_{r_0}(p_0)} \rho_c\{d(x, \gamma(b))\}dP_{X}(x) 
 \nonumber \\ 
     & =   (1-s) \int \rho_c\{d(X, \gamma(a))\}dP 
 +  s \int \rho_c\{d(X, \gamma(b))\}dP. 
\end{align}
Thus, the objective function $t \mapsto F^c(\gamma(t))$ is strictly convex.

Still assuming that Condition (i) holds, suppose now that the population Huber mean is not unique, and that $m_{0}$ and $m'_{0}$ are two distinct population Huber means with respect to $P_{X}$; that is, $m_{0}$ and $m'_{0}$ are both minimizers of $F^{c}$. Since $m_0, m'_0 \in B_{r}(p_0)$, a strongly convex set, there exists a geodesic $\gamma$ connecting $m_{0} = \gamma(s_0)$ and $m'_{0} = \gamma(s_1)$ (for some $s_0 < s_1$). 
Under Condition (i), the function $t\mapsto F^c(\gamma(t))$ is strictly convex on $B_{r}(p_0)$, which contradicts to the hypothesis of two distinct Huber means. Therefore, the Huber mean is unique. 
 
Let us now assume that Condition (ii) holds. Let $\gamma$ be the geodesic $\gamma_X$ given in Condition (ii), so that $P(X \in M\setminus \gamma) = 0.$ Then, calculations similar to (\ref{eq:convex_ineq}) leads that 
the function $t \mapsto F^c(\gamma(t))$ is convex (but it may not be strictly convex). 

Suppose again that both $m_0$ and $m_0'$ minimize $F^c$ (while choosing $m_0$ as in Condition (ii)), and let $\gamma$ be the geodesic connecting $m_0 = \gamma(s_0)$ and $m_0' = \gamma(s_1)$ for some $t_0 < s_0 < s_1 < t_1$.  Since $t \mapsto F^c(\gamma(t))$ is convex, every point on the geodesic segment from $m_0$ to $m_0'$ is also a minimizer of $F^c$. That is, for all $t \in [s_0,s_1]$, 
\begin{equation}
    \label{eq:everythingsminimizer}
    F^c(\gamma(t)) = F^c(\gamma(s_0)) = F^c(\gamma(s_1)) = \min_{t \in (t_0,t_1)} F^c(\gamma(t)).
\end{equation}
The equation above requires that the derivative of $F^c(\gamma(t))$ is constant (in fact, it is zero) for $t \in [s_0,s_1]$. We shall show that this can not happen under Condition (ii). 

Assume without loss of generality that $\gamma$ is a unit speed geodesic. For simplicity, let $Y$ be a random variable, having values on $(t_0, t_1)$, defined by $Y = \gamma^{-1}(X)$ if $X \in \gamma$ and is $(t_0 + t_1 ) /2$ if $X \notin \gamma$. (Recall that under Condition (ii), $P(X \notin \gamma) = 0$.) Write $G$ for the cumulative distribution function of the pullback measure $\gamma^{*}P_{X}$ (the pullback of $P_{X}$ by $\gamma$). Then 
\begin{eqnarray*}
F^c(\gamma(t)) &=& \int_{B_{r_0}(p_0)} \rho_c \{d(x, \gamma(t)\} dP_{X}(x) \\
 &=& \int_{\gamma} \rho_c \{d(x, \gamma(t))\} dP_{X}(x) \\ 
 &=& \int_{(t_0,t_1)} \rho_c (|y - t|) dG(y).
\end{eqnarray*}
Moreover, for any $t \in (t_0,t_1)$, we have 
\begin{eqnarray*}
\frac{\partial}{\partial t} F^c(\gamma(t)) &=& \int_{(t_0,t_1)} 
   \frac{\partial}{\partial t} \rho_c (|y - t|) dG(y) \\
    &=& 2\int_{(t_0,t_1)} 
   (t-y)1_{|y-t|\le c} + c \cdot \mbox{sgn}(t-y)1_{|y-t| > c}dG(y),
\end{eqnarray*}
where $\mbox{sgn}(x) = 1 ~ \mbox{for} ~ x \ge 0, ~ -1 ~ \mbox{for} ~ x<0$. Pick an $0< h < \min\{c, s_1-s_0\}$, and let $s = s_0 + h$. 
Condition (ii) leads that, for this choice of $s$, 
$P(d(X, \gamma(s)) \le c) = P(s-c \le Y \le s +c ) >0$. Thus, at least one of the following is true:\begin{align}
    & P(s-c \le Y \le s) > 0, \label{eq:conditionii-case1}\\
    & P(s \le Y \le s+c) > 0.  \label{eq:conditionii-case2}
\end{align}

Suppose that (\ref{eq:conditionii-case1}) is true. It can be checked that 
\begin{align*}
\frac{\partial}{\partial t}\mid_{t = s} & F^c(\gamma(t)) - \frac{\partial}{\partial t}\mid_{t = s-h} F^c(\gamma(t)) \\
  &=  2\int_{[s-c-h,s-c)} y - (s-c-h) dG(y) + 
     2\int_{[s-c,s+c-h)} h dG(y) \\ 
   &  + 2\int_{(s+c-h,s+c)} (s+c) - y  dG(y) \\
   & \ge  2h \cdot P( s-c \le Y \le s+c - h) \\
   & \ge  2h\cdot P(s-c \le Y \le s)  \\ 
   & > 0.
\end{align*} 
Hence, the derivative of $F^c(\gamma(t))$ is not constant, which contradicts to (\ref{eq:everythingsminimizer}). 

Suppose now that (\ref{eq:conditionii-case2}) is true. For this case, we have  
\begin{eqnarray*}
\frac{\partial}{\partial t}\mid_{t = s+h} F^c(\gamma(t)) - \frac{\partial}{\partial t}\mid_{t = s} F^c(\gamma(t)) &\ge& 2h \cdot P( s-c +h \le Y \le s+c ) \\
&\ge& 2h \cdot P(s \le Y \le s+c) \\
&>& 0,
\end{eqnarray*}which also contradicts to (\ref{eq:everythingsminimizer}). In either case, the Huber mean must be unique. \\

\textbf{Case III ($c  = 0$).} 
Let $\gamma$ be any geodesic on $M$ and let $\infty \le t_0 < t_1 \le \infty$ be as specified in the proof for Case II. 
Note that for the case where $c = 0$, $\rho_0 \{d(x, \gamma(t))\} = d(x,\gamma(t))$.

For $x \in B_{r_0}(p_0) \setminus \gamma$,
$\tfrac{\partial^2}{\partial t^2} d(x, \gamma(t)) > 0$ for any $t \in (t_0,t_1)$, thus 
the function $t \mapsto d(x,\gamma(t))$ is strictly convex on $ (t_0,t_1)$. If $x = \gamma(t')$ for some $t' \in (t_0,t_1)$, then 
the second derivative of the distance function is not defined at $t'$, but it can be checked that the function $d(x,\gamma(\cdot))$ is convex on $ (t_0,t_1)$. Therefore, assuming that Condition (i) holds, a calculation similar to  (\ref{eq:convex_ineq}) yields that the objective function is strictly convex along any geodesic. An argument similar to the proof for Case II shows that the geometric median is unique.

Assume now that Condition (ii) holds. Suppose that the population Huber mean (i.e., the geometric median) is not unique, and that $m_0$ and $m_0'$ are two distinct medians. Let $\gamma$ be the geodesic connecting $m_0 = \gamma(s_0)$ and $m_0' = \gamma(s_1)$. Since $d(x,\gamma(\cdot))$ is convex on $ (t_0,t_1)$, $\gamma(t)$ must be a median for all $t \in [s_0,s_1]$. 
This leads that, 
using the notation introduced in the proof of Case II, 
\begin{equation}
    \label{eq:multiple_medians}
    [s_0, s_1] \subset G^{-1}(\tfrac{1}{2}) := \{t \in \mathbb{R}: G(t) = \frac{1}{2}\}.
\end{equation} 
However, Condition (ii) also states that $P(Y \in (s,s+\epsilon)) > 0$ for any $s \in [s_0,s_1]$ and for any $\epsilon>0$, thus $G$ must be an increasing function on $[s_0, s_1]$. This contradicts to (\ref{eq:multiple_medians}). Thus, the geometric median must be unique.\\

Proof for Part(b). Let $\gamma$ be any geodesic on $M$ and let $\infty \le t_0< t_1 \le \infty$ be as specified in the proof for Part(a), Case II.  Since $\tilde{\rho}_c(\cdot)$ is smooth for any $c \in (0,\infty)$, 
 we obtain that, for any $x\in B_{r_0}(p_0)$ and $t \in (t_0,t_1)$
\begin{align*}
    \tilde{\rho}_c\{ d(x,\gamma(t))\} &= 2c^2 \left( \sqrt{1 + \frac{d^2(x,\gamma(t))}{c^2}} - 1\right),\\
    \frac{\partial}{\partial t} \tilde{\rho}_c\{ d(x,\gamma(t))\} & =  \frac{ c \tfrac{\partial}{\partial t} d^2(x,\gamma(t))}{  \sqrt{d^2(x,\gamma(t)) + c^2} }, \\
       \frac{\partial^2}{\partial t^2} \tilde{\rho}_c\{ d(x,\gamma(t))\} &  = c \left[
       \frac{ \tfrac{\partial^2}{\partial t^2} d^2(x,\gamma(t))  \{d^2(x,\gamma(t)) + c^2\} + \tfrac{1}{2} \{ \tfrac{\partial}{\partial t} d^2(x,\gamma(t)) \}^2 }
       { \{ d^2(x,\gamma(t)) + c^2\}^{3/2}}
       \right].
\end{align*}
Since $\tfrac{\partial^2}{\partial t^2} d^2(x, \gamma(t))  > 0$, it holds that $\frac{\partial^2}{\partial t^2} \tilde{\rho}_c\{ d(x,\gamma(t))\} > 0$, which implies that the function $t \mapsto  \tilde{\rho}_c \{d(x,\gamma(t))\}$ is strictly convex. Therefore, an argument similar to Part(a), Case II under Condition (i), is used to yield the desired result.
\end{proof}

Given $n$ deterministic observations, $(x_{1}, x_{2}, ..., x_{n}) \in M^{n}$, by replacing $P_{X}$ with $P_{n}$ in Theorem~\ref{thm:unique:pop}, we obtain sufficient conditions for the uniqueness of sample Huber means as follows. 
\begin{corollary}[Uniqueness of sample Huber means]\label{cor:unique:sam}
\begin{enumerate}
\item[(a)]
Let $c \in [0,\infty]$ be pre-specified. Given $n$ deterministic observations $(x_{1}, x_{2}, ..., x_{n}) \in M^{n}$, suppose that the sample points $x_{1}, x_{2}, ..., x_{n}$ lie in an open ball of radius $r_0$ as specified in (\ref{eq:r_0}), and that either of the following holds: 
\begin{enumerate}
 \item[(i)] The sample points are in general position (i.e., they do not lie on a single geodesic).
  \item[(ii)] The sample points lie in a single geodesic $\gamma_X$ and the sample size $n$ is an odd number. 
  \item[(iii)] The sample points lie in a single geodesic $\gamma_X$, $n = 2k$ for some $k \ge 1$, and for the order statistics $(x_{(1)}, x_{(2)}, ..., x_{(n)})$ of $(x_{1}, x_{2}, ..., x_{n})$ along the geodesic $\gamma_{X}$ it holds that $d(x_{(k)}, x_{(k+1)})\le 2c$. 
\end{enumerate} 
Then, the sample Huber mean is unique. 
\item[(b)]
Let $c \in (0, \infty)$ be pre-specified. Suppose that the sample points $x_{1}, x_{2}, ..., x_{n}$ lie in an open ball of radius $r_0 = \frac{1}{2}\min\{\frac{\pi}{\sqrt{\Delta}}, r_{\mbox{\tiny inj}}(M) \}$. Then, the sample pseudo-Huber mean is unique.
\end{enumerate}

\begin{proof}[Proof of Corollary~\ref{cor:unique:sam}]
We begin with a proof for Part(a). For Conditions (ii) and (iii), we denote $(x_{(1)}, x_{(2)}, ..., x_{(n)})$ by the order statistics of $(x_{1}, x_{2}, ..., x_{n})$ along the geodesic $\gamma_{X}$. It is noteworthy that the geodesic $\gamma_{X}$ containing $x_1, x_2, \ldots, x_n$ is homeomorphic (in fact, isometric) to an interval in $\mathbb{R}$, as the geodesic is contained in a strongly convex ball. Hence, the order statistics $x_{(1)}, x_{(2)}, \ldots, x_{(n)}$ are well-defined (up to reflection). 

\textbf{Case I ($c \in (0, \infty]$).} Suppose that Condition (i) holds. By replacing $P_{X}$ with $P_{n}$ in Theorem~\ref{thm:unique:pop}, the sample Huber mean is unique.

Next, suppose that Condition (ii) holds. Then, $n$ is equal to $2k+1$ for some integer $k \ge 0$. Parametrize $\gamma_{X}$ by arc-length, and denote $\gamma_{X} : [0,\, \ell] \rightarrow M$. There exist $0 = t_{1} \le t_{2} \le ... \le t_{2k+1} = \ell$ such that for all $1 \le i \le (2k+1)$, $\gamma_{X}(t_{i}) = x_{(i)}$. It is obvious that $E^{(c)}_{n}$ lies in $\gamma_{X}$ (cf. Lemma~\ref{lem:conv}). With abusing of notation, we write $F^{c}_{n}(\gamma_{X}(t))$ by $F^{c}_{n}(t)= \sum_{i=1}^{n}\rho_{c}(|t-t_{i}|)$ for $0 \le t \le \ell$. Since $\frac{\partial F^{c}_{n}(t)}{\partial u}|_{t=t_0} < 0$ for any $t \in (0, (t_{k+1}-c)\vee 0)$ and $F^{c}_{n}(\cdot)$ is continuous on $[0, (t_{k+1}-c)\vee 0]$, $F^{c}_{n}(\cdot)$ is strictly decreasing on $[0, (t_{k+1}-c)\vee 0]$. Similarly, since $\frac{\partial F^{c}_{n}(t)}{\partial t}|_{t=t_0} > 0$ for any $t_0 \in ((t_{k+1} + c)\wedge \ell, \ell])$ and $F_{n}(\cdot)$ is continuous on $[(t_{k+1} + c)\wedge \ell, \ell]$, $F^{c}_{n}(\cdot)$ is strictly increasing on $[(t_{k+1} + c)\wedge \ell, \ell]$. Hence, it follows that $E^{(c)}_{n} \subseteq \gamma_{X}[(t_{k+1} - c) \vee 0 , (t_{k+1} + c) \wedge \ell]$. Since $F^{c}_{n}(t)= \sum_{1 \le i \le n, i\neq k+1} \rho_{c}(|t-t_i|) + (t- t_{k+1})^2$ is strictly convex for $t \in [(t_{k+1} - c) \vee 0 , (t_{k+1} + c) \wedge \ell]$ (owing to the fact that $t \mapsto \rho_{c}(|t-t_i|)$ is convex for all $1 \le i \le n, ~ i \neq k+1$ and $t \mapsto (t-t_{k+1})^2$ is strictly convex), the minimizer of $F^{c}_{n}$ is unique. 

Lastly, suppose that Condition (iii) holds. In this case, $n = 2k$ for some integer $k \ge 1$, and $F^{c}_{n}(t)=\sum_{i=1}^{2k} \rho_{c}(|t-t_i|)$ together with $t_{k+1} - t_{k} \le 2c$. It is easy to check that $\frac{\partial F^{c}_{n}(t)}{\partial t}|_{t = t_0} < 0$ for any $t_0 \in (0, (\frac{t_{k} + t_{k+1}}{2} - c)\vee 0)$ and $F^{c}_{n}(\cdot)$ is continuous on $[0, (\frac{t_{k} + t_{k+1}}{2}-c) \vee 0]$. In turn, $F^{c}_{n}(\cdot)$ is strictly decreasing on $[0, (\frac{t_{k} + t_{k+1}}{2}-c)\vee 0]$. Similarly, owing to the fact that $\frac{\partial F^{c}_{n}(t)}{\partial t}|_{t=t_0} > 0$ for any $t_0 \in ((\frac{t_{k} + t_{k+1}}{2} + c)\wedge \ell, \ell])$ and $F^{c}_{n}(\cdot)$ is continuous on $[(\frac{t_{k} + t_{k+1}}{2} + c)\wedge \ell, \ell]$, it follows that $F^{c}_{n}(\cdot)$ is strictly increasing on $[(\frac{t_{k} + t_{k+1}}{2} + c)\wedge \ell, \ell]$. That is, $E^{(c)}_{n} \subseteq \gamma_{X}[(\frac{t_{k} + t_{k+1}}{2} - c) \vee 0 , (\frac{t_{k} + t_{k+1}}{2} + c) \wedge \ell]$. The condition, $t_{k+1} -t_{k} \le 2c$, implies that $t_{k+1} - \frac{t_{k+1} + t_{k}}{2}, \frac{t_{k+1} + t_{k}}{2} - t_{k} \le c$. We note that $F^{c}_{n}(t)=\sum_{i=1}^{2k} \rho_{c}(|t-t_i|)$ is strictly convex for $t \in [(\frac{t_{k} + t_{k+1}}{2} - c) \vee 0, \frac{t_{k} + t_{k+1}}{2}]$ because $\rho_{c}(|t- t_{k}|) = (t- t_{k})^2$ is strictly convex on the closed interval. In a similar manner, we obtain that $F^{c}_{n}(t)=\sum_{i=1}^{2k} \rho_{c}(|t-t_i|)$ is strictly convex for $t \in [\frac{t_{k} + t_{k+1}}{2}, (\frac{t_{k} + t_{k+1}}{2} + c) \wedge \ell]$ as $\rho_{c}(|t- t_{k}|) = (t- t_{k})^2$ is strictly convex on the closed interval. Moreover, $F^{c}_{n}(t)$ is smooth at $t = \frac{t_{k} + t_{k+1}}{2}$. Using these facts, we get $\frac{F^{c}_{n}(u)}{\partial u}|_{u=t}$ is strictly increasing for $t \in ((\frac{t_{k} + t_{k+1}}{2} - c) \vee 0, (\frac{t_{k} + t_{k+1}}{2} + c) \wedge \ell)$. Therefore, $F^{c}_{n}(t)=\sum_{i=1}^{2k} \rho_{c}(|t-t_i|)$ is strictly convex for $t \in [(\frac{t_{k} + t_{k+1}}{2} - c) \vee 0, (\frac{t_{k} + t_{k+1}}{2} + c) \wedge \ell]$, implying that the minimizer of $F^{c}_{n}$ is unique.

\textbf{Case II ($c=0$).} We use the same notation in Case I, i.e., $F^{0}_{n}(t) = \sum_{i=1}^{n} |t-t_i|$. Suppose that Condition (i) holds. By replacing $P_{X}$ with $P_{n}$ in Theorem~\ref{thm:unique:pop}, the sample geometric median is unique. 

Next, suppose that Condition (ii) holds. Denote $n=2k+1$ for some integer $k \ge 0$. It directly follows that $\frac{\partial F^{0}_{n}(t)}{\partial t}|_{t=t_0} < 0$ for any $t_{0} \in (0, t_{k+1}) \setminus \{t_{1}, t_{2}, ..., t_{k}\}$. An application of the Goldowsky-Tonelli theorem \citep{goldowsky1928note} leads that $F^{0}_{n}(t)$ is strictly decreasing for $t \in [0, t_{k+1}]$. Similarly, it is easy to check that $\frac{\partial F^{0}_{n}(t)}{\partial t}|_{t=t_0} > 0$ for any $t_{0} \in (t_{k+1}, t_{2k+1}) \setminus \{t_{k+2}, t_{k+3}, ..., t_{2k-1}, t_{2k}\}$. Applying the Goldowsky-Tonelli theorem again, we obtain that $F^{0}_{n}(t)$ is strictly increasing for $t \in [t_{k+1}, t_{2k+1}]$. Combining these facts, we obtain that $F^{0}_{n}(t)$ is minimized only at $t=t_{k+1}$. 

Finally, suppose that Condition (iii) holds. Denote $n=2k$ for some integer $k \ge 1$. %The geodesic $\gamma_{X}$ containing $x_1, x_2, \ldots, x_n$ is homeomorphic (in fact, isometric) to an interval in $\mathbb{R}$, as the geodesic is contained in a strongly convex ball. Hence, the order statistics $x_{(1)}, x_{(2)}, \ldots, x_{(n)}$ are well-defined (up to reflection). 
Owing to the fact that $c=0$, we have $t_{k} = t_{k+1}$. By the same argument on the case where Condition (iii) holds, it is easy to see that $F^{0}_{n}(t)$ is minimized only at $t=t_{k} ~ (=t_{k+1})$. \\

Proof for Part(b). Replacing $P_{X}$ by $P_{n}$ in Theorem~\ref{thm:unique:pop}(b), we obtain the desired result.
\end{proof}

\end{corollary}

\subsection{Examples of non-unique Huber means}\label{sec:nonuniqueHuber1}

When Assumption (A2) is violated, the Huber mean is not unique in general. 
For instance, when $P_{X}$ is the uniform distribution on $M = S^{k}$, each point on $S^{k}$ is the population Huber mean with respect to $P_{X}$, thereby not ensuring the uniqueness of Huber means. 
One may be strongly tempted to impose the weaker restriction, $0 < r \le \frac{1}{2}\min\{\frac{\pi}{\sqrt{\Delta}}, r_{\mbox{\tiny inj}}(M)\}$, on Assumption (A2). In this case, however, the uniqueness of Huber means is generally not guaranteed. 
For example, we suppose that $M=S^{2} ~ (\mbox{i.e.,} ~ \Delta = 1)$ and three points are symmetrically positioned on the equator of $S^2$. If $P_{X}$ is equally distributed on the three points, the number of Huber means with respect to $P_{X}$ is at least two by symmetry, not ensuring the uniqueness of Huber means. Here, we call $r_0 ~ (=:r_{\mbox{\tiny uni}}(c))$ in Assumption (A2) as the \textit{uniqueness radius} of $c$. That is, $r_{\mbox{\tiny uni}}(c) = \frac{\pi}{2\sqrt{\Delta}}$ when $\frac{\pi}{\sqrt{\Delta}} \le c \le \infty$ and $r_{\mbox{\tiny uni}}(c) = \frac{\pi}{4\sqrt{\Delta}}$ when $0 \le c < \frac{\pi}{\sqrt{\Delta}}$. For the case where $\frac{\pi}{\sqrt{\Delta}} \le c \le \infty$, the uniqueness radius is optimal in the sense that, if the radius is enlarged a little bit, then the uniqueness of Huber means does not hold. (Consider the case where $M=S^{2}$ and three points with equal mass are symmetrically positioned on the equator of $S^2$.) On the other hand, it seems that $\frac{\pi}{4\sqrt{\Delta}}$ may not be optimal values for the uniqueness radius of $c$ for all $0 \le c \le \frac{1}{\sqrt{\Delta}}$. While an accurate investigation on $r_{\mbox{\tiny uni}}(c)$ for $0 \le c < \frac{\pi}{\sqrt{\Delta}}$ is still an open problem, the uniqueness radius for $0 \le c \le \frac{1}{\sqrt{\Delta}}$ is not enlarged to be $\frac{\pi}{2\sqrt{\Delta}}$; see our example in Remark~\ref{rmk:tight} below.

\begin{remark}[Adapted from Remark 2.4 in \cite{afsari2011riemannian}]\label{rmk:tight}
Suppose that $M=S^2$ with the usual metric, i.e., $\Delta = 1$. We now construct an example where $r_{\mbox{\tiny uni}}(1)$ cannot be enlarged to be $\frac{\pi}{2\sqrt{\Delta}}$. For any $c \in [0,\frac{1}{\sqrt{\Delta}}]$, similar examples can be made. (For the case where $c=0$, refer to Remark 2.4 in \cite{afsari2011riemannian}). Let $p$ be the north pole and $r \in (0,\frac{\pi}{2}]$. Three points, $x_{1}, x_{2}, x_{3} \in S^2$, are positioned so that $d(p, x_{1}) = d(p, x_{2})=d(p, x_{3}) = r$ and the three points are symmetrically positioned with respect to $p$ and the equator of $S^2$ (see Figure~\ref{fig:tight}). Suppose that $P_{X}$ is equally distributed on the three points. If the Huber mean for $c=1$ with respect to $P_{X}$ is unique, the possible candidate of the Huber mean is only $p$ by symmetry. However, if we choose $r \simeq 0.4975 \pi$, then $F^{1}(p) = F^{1}(x_{1})=F^{1}(x_{2})=F^{1}(x_{3})$ by elementary calculations of spherical geometry. In this case, the uniqueness of Huber mean for $c=1$ is violated, and $r_{\mbox{\tiny uni}}(1)$ cannot be enlarged to be $\frac{\pi}{2}$.   
\end{remark}

\begin{figure}[!t]
\center
\includegraphics[scale=0.5]{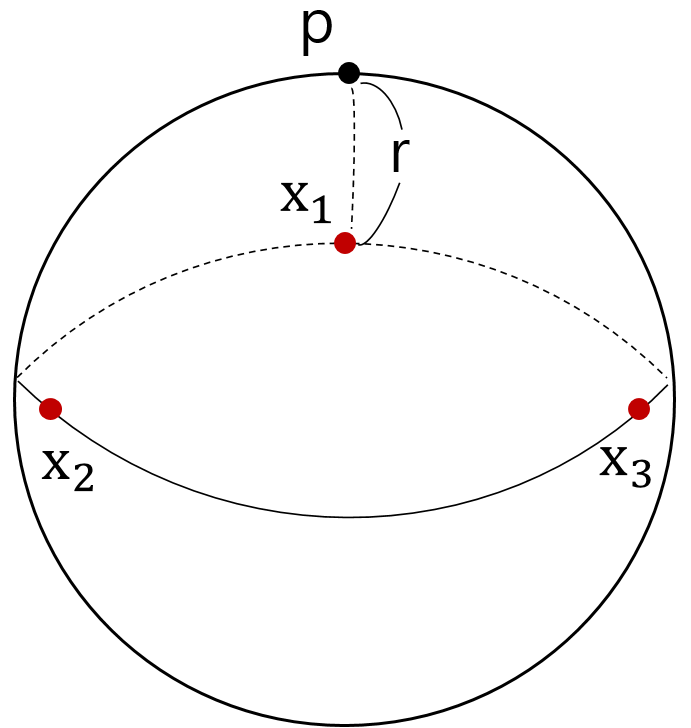}
\caption{Let $p$ be the north pole of $S^2$. Suppose that $\textsf{supp}(P_{X})=\{x_{1}, x_{2}, x_{3}\}$. The three points $x_{1}, x_{2}$ and $x_{3}$ are symmetrically positioned with respect to $p$ and the equator of $S^2$. $P_{X}$ is equally distributed on the three points. If we choose $r \simeq 0.4975\pi$, then the four points, $p$, $x_{1}$, $x_{2}$, and $x_{3}$, may minimize $F^{1}$. Accordingly, the population Huber means for $c=1$ with respect to $P_{X}$ is not unique.}
\label{fig:tight}
\end{figure}

\subsection{Proofs for Section \ref{sec:unbiased}}
%\subsubsection{Definition of asymptotic unbiasedness}
%An $M$-valued estimator $\hat{\theta}_{n}$ is said to be an asymptotically unbiased estimator of parameter $\theta \in M$ if $\textsf{Bias}(\theta_{n}):= d(E(\hat{\theta}_n), \theta)$ goes to zero as $n \to \infty$, where $E(\hat{\theta}_{n})$ denotes the population \Frechet\ mean of $\hat{\theta}_{n}$. In addition, we say that $\hat{\theta}_{n}$ converges to $\theta$ as $n\to\infty$ in $L_{2}$ if $Ed^2(\hat{\theta_{n}}, \theta) ~  (=:\textsf{MSE}(\hat{\theta}_{n}))$ goes to 0 as $n\to\infty$. We define $\textsf{Var}(\hat{\theta}_n) = Ed^2\{\hat{\theta}_n, E(\hat{\theta}_n)\}$ where $E(\hat{\theta}_n)$ refers to the population \Frechet\ mean of $\hat{\theta}_n$. It is known that if $M$ is a Hadamard manifold then the bias-variance inequality \citep[Equation (7),][]{mccormack2022stein} holds as follows:
%\begin{equation}\label{ineq:bias-var}
%\textsf{MSE}(\hat{\theta}_n) \ge \textsf{Bias}^2(\hat{\theta}_n) + \textsf{Var}(\hat{\theta}_n),
%\end{equation}
%where the inequality can be proven by an application of the Toponogov comparison theorem. Hence, $\textsf{MSE}(\hat{\theta}_n) \to 0$ as $n\to \infty$ under certain regularity conditions \citep[Corollary 4,][]{schotz2019convergence}. Due to \ref{ineq:bias-var}, $\hat{\theta}_{n}$ is asymptotically unbiased under the same conditions. 

\subsubsection{Proof of Proposition \ref{prop:unbiased}}

To prove that the sample Huber mean is unbiased under regularity conditions, we need a set of lemmas, as stated next.

\begin{lemma}\label{lem:sam:meas}
Suppose that Condition (C1) is satisfied. For a given $c \in [0, \infty]$ and an integer $n \ge 1$, the sample Huber mean $m^{c}_{n}(X_{1}, X_{2}, \ldots, X_{n})$ is a random variable.
\begin{proof}[Proof of Lemma~\ref{lem:sam:meas}]
Note that $X_{1}, X_{2}, ..., X_{n}$ are already random variables having the distribution $P_{X}$ on $M$. Due to Corollary~\ref{cor:unique:sam}, the function of sample Huber mean 
$$m^{c}_{n}:M^{n} \to M, \quad (x_{1}, x_{2}, \ldots, x_{n}) \mapsto m^{c}_{n}(x_{1}, x_{2}, \ldots, x_{n})$$
is well-defined for any $(x_{1}, x_{2}, \ldots, x_{n})$ for which $x_{1}, x_{2}, \ldots, x_{n}$ do not lie in a single geodesic on $M$. It is easy to see that $m^{c}_{n}(\mathbf{x})$ is continuous at $\mathbf{x} = (x_{1}, x_{2}, \ldots, x_{n}) \in M^{n}$, with respect to the usual product topology in $M^{n}$, for which the sample points $x_{1}, x_{2}, \ldots, x_{n}$ do not lie in a single geodesic on $M$. Since $P_{X}$ does not degenerate to any single geodesic due to Condition (ii) in (C1), $X_{1}, X_{2}, ... X_{n}$ do not lie in a single geodesic with probability 1. As the Riemannian measure is complete, $m^{c}_{n}(X_{1}, X_{2}, \ldots, X_{n})$ is a random variable, i.e., measurable.
\end{proof}
\end{lemma}

From Lemma~\ref{lem:sam:meas}, the sampling distribution of Huber means, $P_{m^{c}_{n}}$, is well-defined. Next lemma characterizes the location of central tendency of data.

\begin{lemma}\label{lem:conv}
For a given $c \in [0, \infty]$, suppose that $P_{X}$ satisfies Assumption (A1). Suppose further that there exists $p \in M$ such that $\textsf{supp}(P_{X}) \subseteq B_{r}(p)$ where $r = \frac{1}{2}\min\{\frac{\pi}{2\sqrt{\Delta}}, r_{\mbox{\tiny inj}}(M)\}$ (which may be infinite). Then, $E^{(c)}$ is contained in the closure of the convex hull of $\textsf{supp}(P_{X})$, and $E^{(c)}_{n}$ is contained the closure of the convex hull of $(X_{i})_{i=1}^{n}$. That is, 
$$
E^{(c)} \subseteq \overline{\mbox{conv}\{\textsf{supp}(P_{X})\}} \quad \mbox{and}\quad E_{n}^{(c)} \subseteq \overline{\mbox{conv}\{(X_{i})_{i=1}^{n}\}},
$$
where $\mbox{conv}\{A\}$ refers to the convex hull of $A ~ (\subset M)$ which is the smallest strongly convex set containing $A$.
\end{lemma}

In Lemma~\ref{lem:conv}, note that $\mbox{conv}\{\textsf{supp}(P_{X})\}$ exists, and $\mbox{conv}\{(X_{i})_{i=1}^{n}\}$ exists with probability 1, since $\textsf{supp}(P_{X})$ is contained in a strongly convex ball $B_{r}(p)$. The following lemma below gives information about the dispersion of $P_{m_n^c}$.

\begin{lemma}\label{lem:sample:finite}
Suppose that Condition (C1) is satisfied. For a given $c \in [0, \infty]$ and a sample size $n \ge 1$, $P_{m^{c}_{n}}$ has a finite variance.
\begin{proof}[Proof of Lemma~\ref{lem:sample:finite}]
Let $c \in [0,\infty]$ and $n \ge 1$ be a fixed integer. When $\Delta > 0$, then the statement holds as $P_{m^{c}_{n}}$ has a bounded support (cf. (iii) of Condition (C1)). For this reason, suppose that $\Delta \le 0$. That is, $M$ is a Hadamard manifold, i.e., simply connected and complete Riemannian manifold with nonpositive curvature. Due to Theorem~\ref{thm:unique:pop} and Corollary~\ref{cor:unique:sam}, $m^{c}_{n}$ exists, and is unique with probability 1. By application of Lemma~\ref{lem:conv}, with probability 1 it holds that $m^{c}_{n} \in \overline{\mbox{conv}\{(X_{i})_{i=1}^{n}\}}$. Hence, we obtain that, with probability 1 
$$
d(\mu, m^{c}_{n}) \le \max_{1 \le i \le n} \{d(\mu, X_{i})\} ~ (=: d(\mu, X_{(n)})),
$$
where the inequality holds due to an application of Proposition 2.3 in \cite{sturm2003probability}. Utilizing the above inequality, we have 
$$
\int d^2(\mu, m^{c}_{n}) dP \le \int d^2(\mu, X_{(n)})dP \le n \int d^2(\mu, X_{1})dP < \infty,
$$
where the second inequality holds by an application of Theorem 3.1 in \cite{papadatos1995maximum}.
\end{proof}
\end{lemma}

By Lemma~\ref{lem:sample:finite}, $P_{m^{c}_{n}}$ satisfies Assumption (A1). In turn, the existence of the population Huber mean for $c$ with respect to $P_{m^{c}_{n}}$ is guaranteed. The next lemma gives information about where the population Huber mean is exactly located.

\begin{lemma}\label{lem:geo:symm}
Suppose that Condition (C1) is satisfied. For any $c \in [0, \infty]$ and any sample size $n \ge 1$, $P_{m^{c}_{n}}$ is geodesically symmetric about $\mu$ that is specified in (C1).
\end{lemma}
By utilizing the isometric equivariance of Huber means (cf. Lemma~\ref{lem:huber:isoequi}), it is not difficult to verify Lemma~\ref{lem:geo:symm}; the proof is omitted. %and it is available upon request. 
We are now prepared to verify the Proposition~\ref{prop:unbiased}.

\begin{proof}[Proof of Proposition~\ref{prop:unbiased}]
We begin with a proof for Part (a). 

Let $c \in [0, \infty]$ be a prespecified constant, and $n \ge 1$ be a fixed integer. Since $P_{X}$ has a finite variance due to Condition (ii) in (C1), $P_{X}$ satisfies Assumption (A1). By an application of Theorems \ref{thm:exist} and \ref{thm:unique:pop}, the population Huber mean for $c$ with respect to $P_{X}$, $m^{c}_{0}$, exists and is unique. We note that $X$ and $s_{\mu}(X)$ have the same distribution on $M$, i.e., $s_{\mu}(X) \overset{d}{=} X$. Since the population Huber mean with respect to $P_{X}$ is isometric-equivariant (cf. Lemma~\ref{lem:huber:isoequi}), $s_{\mu}(m^{c}_{0})$ is a unique population Huber mean for $c$ with respect to $P_{s_{\mu}(X)} ~(= P_{X})$, leading to
\begin{equation}\label{eq:inv:same}
s_{\mu}(m^{c}_{0}) = m^{c}_{0}.
\end{equation}
Since $E^{(c)} \subseteq B_{r}(\mu)$ due to Lemma~\ref{lem:contain:ball}, $E^{(c)} ~(=\{m^{c}_{0}\})$ lies in the strongly convex ball $B_{r}(\mu)$. Assume that $m^{c}_{0} \neq \mu$. Then, there is a unique length-minimizing (minimal) geodesic $\gamma$ joining $\mu$ and $m^{c}_{0}$. Denote $\tilde{\gamma}$ by the reverse geodesic of $\gamma$ through $\mu$. Since the geodesic symmetry $s_{\mu}$ reverses $\gamma$, the geodesic $\tilde{\gamma}$ connects the $\mu$ and $s_{\mu}(m^{c}_{0})$. Consider a combined geodesic of $\gamma$ and $\tilde{\gamma}$, denoted by $\tilde{\gamma} \cup \gamma$, joining $m_0^{c}$, $\mu$, and $s_{\mu}(m_0^{c})$ in order. In turn, the joined geodesic becomes the length-minimizing geodesic connecting $s_{\mu}(m^{c}_{0})$ and $m^{c}_{0}$ (since the three points $s_{\mu}(m^{c}_{0})$, $\mu$, and $m^{c}_{0}$ lie in the strongly convex ball $B_{r}(\mu)$), which implies that $$d(m^{c}_{0}, s_{\mu}(m^{c}_{0})) = 2 d(\mu, m^{c}_{0}) > 0.$$ It is a contradiction to (\ref{eq:inv:same}), thereby leading to $m^{c}_{0} = \mu$. Therefore, the population Huber means for $c$ with respect to $P_{X}$ equals $\mu$, as desired. \\

Proof for Part (b). Let $c \in [0, \infty]$ be prespecified and $n \ge 1$ be a fixed integer. By an application of Lemma~\ref{lem:sam:meas}, $m^{c}_{n}(X_{1}, X_{2}, \ldots, X_{n})$ is a random variable. Hence, the distribution (law) of $m^{c}_{n}$, denoted by $P_{m^{c}_{n}}$, is well-defined. By an application of Lemma~\ref{lem:sample:finite}, $P_{m^{c}_{n}}$ satisfies Assumption (A1); thus, the population \Frechet\ mean with respect to $P_{m^{c}_{n}}$ exists by an application of Theorem~\ref{thm:exist}. Moreover, since $P_{m^{c}_{n}}$ is geodesically symmetric about $\mu$ owing to Lemma~\ref{lem:geo:symm}, the population Fr\'{e}chet mean with respect to $P_{m^{c}_{n}}$ equals $\mu$ by an application of Proposition~\ref{prop:unbiased}(a). Because $\mu$ is the same as the population Huber mean for $c$ with respect to $P_{X}$ by an application of Proposition~\ref{prop:unbiased}(a) again, the population Fr\'{e}chet mean with respect to $P_{m^{c}_{n}}$ equals the population Huber mean for $c$ with respect to $P_{X}$, meaning that the sample Huber mean for $c$ is unbiased. Furthermore, it directly follows that the sample Huber mean for $c$ is an unbiased estimator of $\mu$. 
\end{proof}

\color{black}
\subsection{Proofs for Section \ref{sec:asymptotics}}

\subsubsection{Proof of Theorem \ref{thm:slln}}

\begin{proof}[Proof of Theorem~\ref{thm:slln}]
In the case of $c=\infty$, Theorem~\ref{thm:slln} has been proved by an application of Theorems A.3 and A.4 in \cite{huckemann2011intrinsic}. For this reason, assume $c\in [0, \infty)$. \\ 
\underline{Step I.} We shall show that, with probability 1
\begin{equation}\label{zie:consi}
\cap_{n=1}^{\infty} \overline{\cup_{k=n}^{\infty} E^{(c)}_{k}} \subseteq E^{(c)}.
\end{equation}
If the left-hand side of (\ref{zie:consi}) is empty, clearly (\ref{zie:consi}) holds. Suppose that $\cap_{n=1}^{\infty} \overline{\cup_{k=n}^{\infty} E^{(c)}_{k}} \neq \phi$. Let $\{p_1, p_2, ... \}$ be a countable dense subset of $M$. (The data space $M$ is assumed to be separable in Section~\ref{sec:setup}.) There exists a sequence of events $(A_{k})_{k \ge 1}$ such that for any $k \ge 1$ $P(A_k)=1$ and if $A_k$ has occurred then it holds that
\[
\frac{1}{n}\sum_{i=1}^{n} \rho_{c}\{d(X_i, p_k)\} \underset{n \rightarrow \infty}{\rightarrow} \int \rho_{c}\{d(X, p_k)\} dP. 
\]
Let $A:= \cap_{k=1}^{\infty} A_k$. Then $P(A)=1$. If $A$ has occurred, for any $k\ge 1$
\[
\frac{1}{n}\sum_{i=1}^{n} \rho_{c}\{d(X_i, p_k)\} \underset{n \rightarrow \infty}{\rightarrow} \int \rho_{c}\{d(X, p_k)\} dP, 
\]
where it holds by the strong law of large numbers. We wish to show that 
\[
P\big(\frac{1}{n}\sum_{i=1}^{n} \rho_{c}\{d(X_i, p)\} \underset{n \rightarrow \infty}{\rightarrow} \int \rho_{c}\{d(X, p)\} dP \quad \mbox{for all}~ p \in M \big) = 1.
\]
For any $p \in M$, there exists a sequence $(p_{k_{u}})_{u \ge 1}$ such that $p_{k_u} \underset{u \rightarrow \infty}{\rightarrow} p$.  
\begin{equation}\label{split}
\frac{1}{n} \sum_{i=1}^{n} \rho_{c}\{d(X_i, p)\} = \frac{1}{n} \sum_{i=1}^{n} \rho_{c}\{d(X_i, p_{k_{u}})\} + \frac{1}{n} \sum_{i=1}^{n} [\rho_{c}\{d(X_i, p)\} - \rho_{c}\{d(X_i, p_{k_{u}})\}]. 
\end{equation}
$\rho_{c}$ has a Lipshitz constant $2c \lor 1$. In view of the triangle inequality, 
\begin{eqnarray*}
|\frac{1}{n} \sum_{i=1}^{n} [\rho_{c}\{d(X_i, p)\} - \rho_{c}\{d(X_i, p_{k_{u}})\}]| &\le& \frac{(2c \lor 1)}{n} \sum_{i=1}^{n} |d(X_i, p) - d(X_i, p_{k_{u}})| \\ 
&\le& (2c \lor 1) d(p, p_{k_{u}}) \\
& & \underset{u \rightarrow \infty}{\rightarrow} 0
\end{eqnarray*}
By taking $\limsup_{n \rightarrow \infty}, \liminf_{n \rightarrow \infty}$ to (\ref{split}) respectively, with probability 1  
\begin{eqnarray*}
\limsup_{n \rightarrow \infty} \frac{1}{n} \sum_{i=1}^{n} \rho_{c}\{d(X_i, p)\} \le \int \rho_{c}\{d(X, p)\}dP + (2c \lor 1) d(p, p_{k_{u}}),  \\
\liminf_{n \rightarrow \infty} \frac{1}{n} \sum_{i=1}^{n} \rho_{c}\{d(X_i, p)\} \ge \int \rho_{c}\{d(X, p)\}dP - (2c \lor 1) d(p, p_{k_{u}}).  
\end{eqnarray*}
Taking $\lim_{u \rightarrow \infty}$ to the above inequalities, we have
\[
P\big(\frac{1}{n}\sum_{i=1}^{n} \rho_{c}\{d(X_i, p)\} \underset{n \rightarrow \infty}{\rightarrow} \int \rho_{c}\{d(X, p)\} dP \quad \mbox{for all}~ p \in M \big) = 1.
\]

Suppose now that $p' \in \cap_{n=1}^{\infty} \overline{\cup_{k=n}^{\infty} E^{(c)}_{k}}$. By an application of Lemma~\ref{lem:subseq}, there exists a sequence of points $p_{k_{\ell}} \in E^{(c)}_{k_{\ell}}$ such that $p_{k_{\ell}} \underset{\ell \rightarrow \infty}{\rightarrow} p'$. For each $\ell$, with probability 1
\begin{eqnarray}\label{eq_final}
\frac{1}{k_{\ell}}\sum_{i=1}^{k_{\ell}}\rho_{c}\{d(X_i, p)\} \ge \frac{1}{k_{\ell}} \sum_{i=1}^{k_{\ell}}\rho_{c}\{d(X_i, p_{k_{\ell}})\} \quad \mbox{for all} ~ p \in M.
\end{eqnarray}
Taking $\liminf_{\ell \rightarrow \infty}$ to (\ref{eq_final}), we obtain that, with probability 1, for all $p \in M$
\begin{eqnarray}\label{eq_final2}
\int \rho_{c}\{d(X, p)\}dP &\ge& \lim_{\ell \rightarrow \infty} \frac{1}{k_{\ell}}\sum_{i=1}^{k_{\ell}}\rho_{c}\{d(X_i, p')\} \nonumber \\
& & + \liminf_{\ell \rightarrow \infty} \frac{1}{k_{\ell}}\sum_{i=1}^{k_{\ell}}[\rho_{c}\{d(X_i, p_{k_{\ell}})\} - \rho_{c}\{d(X_i, p')\}]. \nonumber \\
\end{eqnarray}
In view of the triangle inequality,
\begin{align}
\frac{1}{k_{\ell}}\sum_{i=1}^{k_{\ell}}|\rho_{c}\{d(X_i, p_{k_{\ell}})\} - \rho_{c}\{d(X_i, p')\}| \le& \frac{(2c \lor 1)}{k_{\ell}}\sum_{i=1}^{k_{\ell}}[d(X_{i}, p_{k_{\ell}}) - d(X_{i}, p')] \nonumber \\
\le& (2c \lor 1) d(p_{k_{\ell}}, p') \nonumber \\    &\underset{\ell \rightarrow \infty}{\rightarrow} 0.  \label{eq_final3} 
\end{align}
Combining (\ref{eq_final2}) and (\ref{eq_final3}), we obtain that with probability 1
\begin{eqnarray*}
&& \int \rho_{c}\{d(X, p)\}dP \ge \int \rho_{c}\{d(X, p')\}dP \quad \mbox{for all} ~ p \in M.
\end{eqnarray*}
It means that $p' \in E^{(c)}$ with probability 1. Since $p' \in \cap_{n=1}^{\infty} \overline{\cup_{k=n}^{\infty} E^{(c)}_{k}}$ was arbitrarily chosen,  $\cap_{n=1}^{\infty}\overline{\cup_{k=n}^{\infty} E^{(c)}_{k}} \subseteq E^{(c)}$ with probability 1.

\underline{Step II.} We shall show that, with probability 1, for any $\epsilon > 0$ there exists $n_0 \ge 1$ such that 
\begin{equation}\label{BPC}
\cup_{k=n_0}^{\infty} E^{(c)}_{k} \subseteq \{p \in M : d(p, E^{(c)}) \le \epsilon \}.
\end{equation}
The proof of (\ref{BPC}) follows those of Theorem A.4 in \cite{huckemann2011intrinsic} and Lemma B.18 in \cite{jung2025averaging}. By Proposition~\ref{prop:cpt}, there exists a compact set $K_{c} \subseteq M$ such that $E^{(c)} \subset K_{c}$. There exist $p_0 \in M, r > 0$ such that $K_{c} \subseteq K:=\{p \in M : d(p_0, p) \le r \}$. Under the probability 1 event that (\ref{zie:consi}) occurs, every accumulation point of $\cup_{k=1}^{n} E^{(c)}_{k}$ lies in $E^{(c)}$ (since $E^{(c)}$ is sequentially compact by Proposition~\ref{prop:cpt}). Define a sequence $a_{n}=\sup_{m \in E^{(c)}_{n}} d(m, E^{(c)})$. Under (\ref{zie:consi}) occurs, $a_{n} \underset{n \rightarrow \infty}{\nrightarrow} 0$ yields $a_{n} \underset{n \rightarrow \infty}{\rightarrow} \infty$. In short, there are only two cases either $a_{n} \underset{n \rightarrow \infty}{\rightarrow} 0$ or $a_{n} \underset{n \rightarrow \infty}{\rightarrow} \infty$ with probability 1. We will exclude the case $a_{n} \underset{n \rightarrow \infty}{\rightarrow} \infty$.

Suppose that $a_{n} \underset{n \rightarrow \infty}{\rightarrow} \infty$. Choose a sequence of points $m_{n} \in E^{(c)}_{n}$ such that $a_{n}(1-\frac{1}{n}) \le d(m_{n}, E^{(c)}) \le a_{n}$.
Since $a_{n} \underset{n \rightarrow \infty}{\rightarrow} \infty$, $d(m_{n}, E^{(c)}) \underset{n \rightarrow \infty}{\rightarrow} \infty$. Recall that $E^{(c)} \subset K$. For each $x \in K$,
\begin{eqnarray*}
d(m_{n}, E^{(c)}) &=& d(m_{n}, m') \\
&\le& d(m_{n}, x) + d(x, m_{x}) + d(m_{x}, m') \\
&\le& d(m_{n}, x) + d(x, p_0) + d(p_0, m_{x}) + d(m_{x}, p_0) + d(p_0, m') \\
&\le& d(m_{n}, x) + 4r,
\end{eqnarray*}
where each of $m_{x}, m' \in E^{(c)}$ is a point satisfying $$d(x, m_{x}) = \min_{m \in E^{(c)}}d(x, m) \quad \mbox{and} \quad d(m_{n}, m')= \min_{m \in E^{(c)}}d(m_{n}, m),$$ respectively. (Note that $E^{(c)}$ is compact by Proposition~\ref{prop:cpt}.) That is, $d(m_{n}, x) \ge d(m_{n}, E^{(c)}) - 4r$. Hence, $d(m_{n}, E^{(c)}) \underset{n\rightarrow \infty}{\rightarrow} \infty$ gives rise to 
\[
\exists (C_{n})_{n \ge 1} ~ \mbox{with} ~ C_{n} \underset{n \rightarrow \infty}{\rightarrow} \infty ~ \mbox{such that} ~ d(m_{n}, x) \ge C_{n} ~ \mbox{for any} ~ x \in K.
\]
For each $n$, with probability 1 there exists $k(n) ~ \mbox{such that} ~ X_{n_{j}}'s ~ \mbox{lie in}~ K ~ \mbox{for any} ~ j=1, 2, ..., k(n).$ By the strong law of large numbers, with probability 1
\[
\frac{k(n)}{n} \underset{n \rightarrow \infty}{\rightarrow} P_{X}(K) = P_{X}\big(\{p \in M: d(p_0, p) \le r\}\big) > 0.
\]
With probability 1,
\begin{equation}\label{ell:1}
\ell_{n}:= f_{n}(m_{n}) \ge \frac{1}{n} \sum_{j=1}^{k(n)} d(m_{n}, X_{n_{j}}) \ge \frac{k(n)}{n} C_{n} \underset{n \rightarrow \infty}{\rightarrow} \infty.
\end{equation}
Choose a point $p \in E^{(c)}$. By the strong law of large numbers again, with probability 1 
\begin{eqnarray}\label{ell:2}
\ell_{n} \le F^{c}_{n}(p) \underset{n \rightarrow \infty}{\rightarrow} F^{c}(p) &=& \int \rho_{c}\{d(X, p)\}dP < \infty,
\end{eqnarray}
as assumed in Assumption (A1). It is a contradiction, since (\ref{ell:1}), (\ref{ell:2}) hold at the same time with probability 1. Therefore, we can exclude the case $a_{n} \underset{n \rightarrow \infty}{\rightarrow} \infty$, and we get $a_{n} \underset{n \rightarrow \infty}{\rightarrow} 0$, yielding (\ref{BPC}) holds. 

Note that (\ref{BPC}) is equivalent to
\begin{eqnarray*}
& \forall \epsilon > 0 ~ \exists n_{0} ~ \mbox{s.t.} ~ \forall m \in \cup_{k=n_0}^{\infty} E^{(c)}_{k} \Rightarrow d(m, E^{(c)}) \le \epsilon \\
&\Leftrightarrow \forall \epsilon > 0 ~ \exists n_{0} ~ \mbox{s.t.} ~ 
n \ge n_0 \Rightarrow \forall m \in E^{(c)}_{n} ~ d(m, E^{(c)}) \le \epsilon \\
&\Leftrightarrow \forall \epsilon > 0 ~ \exists n_{0} ~ \mbox{s.t.} ~ 
n \ge n_0 \Rightarrow \sup_{m \in E^{(c)}_{n}} d(m, E^{(c)}) \le \epsilon.
\end{eqnarray*}
It is equivalent to
\begin{equation*}
\lim_{n \rightarrow \infty} \sup_{m \in E^{(c)}_{n}} d(m, E^{(c)}) = 0 \quad \mbox{a.s.}
\end{equation*}
which completes the proof.
\end{proof}

In fact, for $c \in (0,\infty]$, the result of Theorem~\ref{thm:slln} 
can be obtained by an application of Theorems A.3 and A.4 in \cite{huckemann2011intrinsic}.
Moreover, the strong consistency of the geometric median on Riemannian manifolds has been verified in Theorem 5 of \cite{arnaudon2012medians}. For completeness, we have provided our proof of Theorem~\ref{thm:slln}, which largely follows the arguments of \cite{huckemann2011intrinsic}, but is modified to include the case $c=0$.

We note that the convergence in Theorem~\ref{thm:slln} may be called a convergence in \textit{one-sided} Hausdorff distance between $E^{(c)}_{n}$ and $E^{(c)}$ but not in (two-sided) Hausdorff distance, which is defined as 
\begin{eqnarray}\label{Hausdorff}
d_{H}(E^{(c)}_{n}, E^{(c)}) &=& \max\{\sup_{m \in E^{(c)}_n} \inf_{p \in E^{(c)}} d(m, p),~ \sup_{p \in E^{(c)}} \inf_{m \in E^{(c)}_n} d(m, p)\} \nonumber \\ 
&=& \max\{\sup_{m \in E^{(c)}_n} d(m, E^{(c)}),~ \sup_{p \in E^{(c)}} d(E^{(c)}_{n}, p) \}.
\end{eqnarray}
See \cite{schotz2022strong} for these terminologies.
If the population Huber mean is unique, then $E^{(c)}_{n}$ converges to the Huber mean in the two-sided Hausdorff distance, as stated next.

\begin{corollary}\label{a.s.conv}
For a given $c \in [0, \infty]$, suppose that $P_{X}$ satisfies Assumption (A1), and that $E^{(c)} ~ (=\{m^{c}_{0} \})$ is a singleton set. With probability 1,
\begin{eqnarray*}
\lim_{n\rightarrow \infty} d_{H}(E^{(c)}_{n}, E^{(c)})=0.
\end{eqnarray*}
Additionally, if $E^{(c)}_{n} ~ (= \{m^{c}_{n}\})$ is singleton for any $n$ with probability 1, then $m^{c}_{n} \underset{n \rightarrow \infty}{\rightarrow} m^{c}_{0}$ almost surely.
\begin{proof}[Proof of Corollary~\ref{a.s.conv}]
Let $E^{(c)}=\{m^{c}_0\}$. By Theorem~\ref{thm:slln}, with probability 1
\begin{eqnarray*}
\lim_{n\rightarrow \infty} d(E^{(c)}_{n}, m^{c}_0) &=& \lim_{n\rightarrow \infty} \inf_{m \in E^{(c)}_n} d(m, m^{c}_0) \\ 
&\le& \lim_{n\rightarrow \infty} \sup_{m \in E^{(c)}_{n}} d(m, m^{c}_0) \\
&=& \lim_{n\rightarrow \infty} \sup_{m \in E^{(c)}_{n}} d(m, E^{(c)}) \\
&=& 0.
\end{eqnarray*}
By (\ref{Hausdorff}), $\lim_{n\rightarrow \infty} d_{H}(E^{(c)}_n, E^{(c)})=0$ with probability 1. 

Additionally, if $E^{(c)}_{n}(X_{1}, X_{2}, ..., X_{n}) ~ (= \{m^{c}_{n}\})$ is singleton for any $n$ with probability 1, then $m^{c}_{n} \underset{n \rightarrow \infty}{\rightarrow} m^{c}_{0}$ almost surely (since $d(m^{c}_{n}, m^{c}_{0}) \underset{n \rightarrow \infty}{\rightarrow} 0$ almost surely).
\end{proof}
\end{corollary}

\subsubsection{Proof of Theorem \ref{thm:clt}}
\begin{proof}[Proof of Theorem~\ref{thm:clt}]
Proof for Part (a). It is satisfied by an application of Corollary~\ref{a.s.conv}. 

Proof for Part (b). When $c=\infty$, it is known that the consequence has been already proven in \citep{bhattacharya2003large, bhattacharya2005large}. In the case of $c \in (0, \infty)$, the proof closely follows the way paved in \cite{bhattacharya2003large, bhattacharya2005large, huckemann2011inference, jung2025averaging}. Given the local coordinate chart $(\phi_{m^{c}_0}, U)$ centered at $m^{c}_0$, due to Part(a) $m^{c}_{n}$ falls within $U$ eventually almost surely; thus, $\phi_{m^{c}_0}(m^{c}_{n})$ is well-defined with probability 1. Observe now that, for each $x \in M$ and any $i=1, 2, ..., k$, $\frac{\partial}{\partial x_i} \rho_{c}\{d(x, \phi_{m_{0}^{c}}^{-1}(\mathbf{0}))\} \le \frac{\partial}{\partial x_i} d^2(x, \phi_{m_{0}^{c}}^{-1}(\mathbf{0}))$. Hence, (\ref{integrability1}) in turn gives $E[\frac{\partial}{\partial x_i} \rho_{c}\{d(X, \phi_{m_{0}^{c}}^{-1}(\mathbf{0}))\}] < \infty$ for all $i=1,2, ..., k$. That is, it holds that $E[\textsf{grad}\rho_{c}\{d(X, \phi_{m_{0}^{c}}^{-1}(\mathbf{0}))\}]$ exists. For each $x \in M$, recall that $\mathbf{x} \mapsto \rho_{c}\{d(x, \phi^{-1}_{m^{c}_{0}}(\mathbf{x}))\}$ is continuously differentiable at near $\mathbf{x} = \mathbf{0}$ almost surely (as assumed in Assumption (A3)), and that $E\{\textsf{grad}d^2(X, \phi_{m_{0}^{c}}^{-1}(\mathbf{0}))\} ~ \mbox{exists}$ (as assumed in Assumption (A4)). The expectation and differentiation below can be interchanged by Lebesgue's dominated convergence theorem. That is,
\[
E[\textsf{grad}\rho_{c}\{d(X, \phi^{-1}_{m^{c}_{0}}(\mathbf{0}))\}] = \textsf{grad}E[\rho_{c}\{d(X, \phi^{-1}_{m^{c}_{0}}(\mathbf{0}))\}] = \mathbf{0}.
\]
Note further that
\begin{align*}
E[\frac{\partial}{\partial x_i} & \rho_{c}\{d(X, \phi_{m_{0}^{c}}^{-1}(\mathbf{0}))\}\frac{\partial}{\partial x_j} \rho_{c}\{d(X, \phi_{m_{0}^{c}}^{-1}(\mathbf{0}))\}] \\
&\le E\{\frac{\partial}{\partial x_i} d^2(X, \phi_{m_{0}^{c}}^{-1}(\mathbf{0}))\frac{\partial}{\partial x_j} d^2(X, \phi_{m_{0}^{c}}^{-1}(\mathbf{0}))\} \\ 
&< \infty,
\end{align*}
for all $i, j = 1, 2, ..., k$, where it holds that $\textsf{Cov}[\textsf{grad}d^2(X, \phi_{m_{0}^{c}}^{-1}(\mathbf{0}))]$ exists as assumed in (A4). Now let $\mathbf{x}_{n}=(x_{n,1}, x_{n,2}, ..., x_{n,k})^{T}:=\phi_{m^{c}_0}(m^{c}_{n}) \in \bbR^{k}$. Since the local coordinate chart $\phi_{m^{c}_{0}}$ is defined on the neighborhood $U$ near $m^{c}_{0}$, by an application of Part(a), it holds that $m_{n}^{c} \rightarrow m^{c}_{0}$ almost surely as $n\rightarrow \infty$. So, $\mathbf{x}_{n}$ is well-defined for large $n$ with probability 1. (Note that $\phi_{m^{c}_{0}}$ is defined only on the neighborhood $U$ of $m^{c}_{0}$.) Let $g_{n}(\mathbf{x}_{n}):= \sum_{i=1}^{n} \rho_{c}\{d(X_{i}, \phi^{-1}_{m^{c}_{0}}(\mathbf{x}_{n}))\}$. By an application of the multivariate central limit theorem  \citep[cf.][]{anderson1958introduction}, $\frac{1}{\sqrt{n}} \textsf{grad}g_{n}(\mathbf{0}) = \frac{\sqrt{n}}{n}\sum_{i=1}^{n}\textsf{grad}\rho_{c}\{d(X_i, \phi_{m^{c}_{0}}^{-1}(\mathbf{0}))\}$ converges to $N_{k}(\mathbf{0}, \Sigma_{c})$ in distribution as $n \rightarrow \infty$, since the expectation and covariance of $\textsf{grad}\rho_{c}\{d(X, \phi^{-1}_{m^{c}_{0}}(\mathbf{0}))\}$ exist.

Applying the mean-value theorem to each component of $\frac{1}{\sqrt{n}}\textsf{grad} g_{n}(\mathbf{x}_{n})$, we obtain that
\begin{eqnarray}\label{eq:mvt}
0=\frac{1}{\sqrt{n}}\textsf{grad}g_{n}(\mathbf{x}_{n})=\frac{1}{\sqrt{n}}\textsf{grad}g_{n}(\mathbf{0})+\{\frac{1}{n} \textsf{Hess}g_{n}(\mathbf{t}_{n})\} \cdot \sqrt{n} \mathbf{x}_{n}
\end{eqnarray}
for some $\mathbf{t}_{n}=(t_{1}x_{n,1}, t_{2}x_{n,2}, ..., t_{k}x_{n,k})^{T} \in \bbR^{k}$. Here $\cdot$ denotes the matrix multiplication and $0 \le t_{i} \le 1$ for any $1 \le i \le k$. Since $\mathbf{x}_{n} \underset{n \rightarrow \infty}{\rightarrow} \mathbf{0}$ almost surely, it follows that $\mathbf{t}_{n} \underset{n \rightarrow \infty}{\rightarrow} \mathbf{0}$ almost surely. In combination with the strong law of large numbers and the continuity of $H_{c}(\mathbf{x})$ at $\mathbf{x}=\mathbf{0}$, $\frac{1}{n}\textsf{Hess}g_{n}(\mathbf{t}_{n})$ converges to $H_{c} ~ (=H_{c}(\mathbf{0}))$ as $n \rightarrow \infty$ almost surely. Recall that $\mathbf{x} \mapsto \rho_{c}\{d(X, \phi^{-1}_{m^{c}_{0}}(\mathbf{x}))\}$ is twice-continuously differentiable at $\mathbf{x}=\mathbf{0}$ with probability 1 (due to condition (iii) of (A4)), and that $\mathbf{x} \mapsto E[\textsf{Hess}d^2(X, \phi^{-1}_{m^{c}_{0}}(\mathbf{x}))]$ exists for $\mathbf{x}$ near $\mathbf{0}$ (owing to condition (ii) of (A4)). By an application of Lebesgue's dominated convergence theorem, the differentiation and expectation can be interchanged twice. In this regard, we obtain that $\textsf{Hess}E[\rho_{c}\{d(X, \phi^{-1}_{m^{c}_{0}}(\mathbf{x}))\}]=E[\textsf{Hess}\rho_{c}\{d(X, \phi^{-1}_{m^{c}_{0}}(\mathbf{x}))\}]$ at $\mathbf{x}=\mathbf{0}$. Accordingly,
$H_{c}= E[\textsf{Hess}\rho_{c}\{d(X, \phi^{-1}_{m^{c}_{0}}(\mathbf{0}))\}]=\textsf{Hess}E[\rho_{c}\{d(X, \phi^{-1}_{m^{c}_{0}}(\mathbf{0}))\}]$ is non-singular by (\ref{integrability1}); that is, its determinant is non-zero, $\mbox{det}(H_{c}) \neq 0$. Since the determinant function is continuous, for all sufficiently large $n$, $\frac{1}{n} \textsf{Hess}g_{n}(\mathbf{t}_{n})$ is invertible with probability 1. Moreover, $\{\frac{1}{n} \textsf{Hess}g_{n}(\mathbf{t}_{n})\}^{-1}$ converges to $H^{-1}_{c}$ almost surely as $n \rightarrow \infty$. Thanks to (\ref{eq:mvt}), we get, for all sufficiently large $n$, $\sqrt{n} \mathbf{x}_{n} = -\{\frac{1}{n} \textsf{Hess}g_{n}(\mathbf{t}_{n})\}^{-1} \frac{1}{\sqrt{n}}\textsf{grad} g_{n}(\mathbf{0})$ with probability 1. By Slutsky's theorem, therefore, $\sqrt{n} \phi_{m^{c}_{0}}(m^{c}_{n}) ~ (=\sqrt{n}\mathbf{x}_{n})$ converges to $N_{k}(\mathbf{0}, H_{c}^{-1}\Sigma_{c}H^{-1}_{c})$ in distribution as $n \rightarrow \infty$.
\end{proof}

\subsection{Technical details in Section \ref{subsec:est:cov}}

\subsubsection{A novel estimation procedure for limiting covariance matrix}
Let $c \in (0,\infty]$ be prespecified. For notational simplicity, we use $m$ for the population location parameter $m:= m_0^c$. Furthermore, we use the Riemannian logarithmic map to define the normal coordinate chart of $M$ centered at $m$, i.e., $\phi_m =  \mbox{Log}_m$, and assume further that, for an open neighborhood $U$ of $m$, $\phi_m(U) \subset \mathbb{R}^k$. Finally, suppose that $P_X$ is absolutely continuous with respect to $V$. 

We propose a novel moment-based estimation procedure for $\Sigma_c$ and $H_c$ for any $c \in (0,\infty]$. To the best of our knowledge, our approach for estimating the Hessian matrix $H_c$ for both Huber means and Fr\'{e}chet means, has never been considered in the literature. For the case of $c = \infty$, our moment-based estimator of $\Sigma_c$ has been used in some previous works \citep{hotz2015intrinsic, eltzner2019smeary, jung2025averaging}, but it has been misused as an estimator for $H^{-1}_{\infty}\Sigma_{\infty}H^{-1}_{\infty}$.

\textbf{Estimation for $\Sigma_c$}. 
Recall that $\Sigma_c = \textsf{Cov} [\textsf{grad} \rho_c\{ d(X, \mbox{Exp}_{m}(\mathbf{0}))\}]$. Since $m = m_{0}^{c}$ is the minimizer of the expected loss function, 
$E [\textsf{grad} \rho_c\{ d(X, \mbox{Exp}_{m}(\mathbf{0}))\} ] = \mathbf{0}$,  
and 
$$ 
\Sigma_c = 
E [ \textsf{grad} \rho_c\{ d(X, \mbox{Exp}_{m}(\mathbf{0}))\} ] [\textsf{grad} \rho_c\{ d(X, \mbox{Exp}_{m}(\mathbf{0}))\} ]^T.
$$
With $\textsf{grad} \rho_{c}\{ d(x, \mbox{Exp}_{m}(\mathbf{0}))\} = 2 \frac{\| \mbox{Log}_{m}(x) \|_{m} \wedge c}{\| \mbox{Log}_{m}(x) \|_{m}}\cdot \mbox{Log}_{m}(x)$ for any $x \in U$, $\Sigma_{c}$ is naturally estimated by a moment estimator
\begin{equation}
    \label{eq:Sigma_c_est}
    \hat{\Sigma}_{c}:=\frac{4}{n}\sum_{i=1}^n \left( \frac{\| \mbox{Log}_{m}(X_i) \|_{m} \wedge c}{\| \mbox{Log}_{m}(X_i) \|_{m}} \right)^2 \mbox{Log}_{m}(X_i) \mbox{Log}_{m}(X_i)^T.
\end{equation}

The multiplicative factor 4 that appears in (\ref{eq:Sigma_c_est}) is an artifact of using the form of squared distances in $\rho_c$, rather than using half of the squared distance. 

\textbf{Estimation for $H_{c}$}. 
We propose a novel estimation procedure for the Hessian matrix $H_{c}$. Recall that $H_c := H_c(\mathbf{0})$ and $H_c(\xv)  = E \{ \textsf{Hess} g_X(\xv) \}$. 
$H_c(\xv) = E[\textsf{Hess} \rho_c\{ d(X, \mbox{Exp}_{m}(\xv))\}]$.
 
Directly evaluating the Hessian matrix 
$\textsf{Hess} \rho_c\{ d(x, \mbox{Exp}_{m}(\xv))\}$ for a given $x \in M$ is quite challenging, especially because the Huber loss function involves both the distance and squared distance functions. Even the seemingly simple task of deriving the Hessian matrix of the squared distance function is complicated, as demonstrated by \cite{pennec2018barycentric} for spheres and hyperbolic spaces. Instead of directly evaluating the Hessian matrix of the Huber loss function, we propose to evaluate a finite set of second-directional derivatives of the Huber loss function, from which the corresponding Hessian matrix can be identified. In other words, our strategy is to obtain an estimator of ${H}_{c}$ from the numerical value of the quadratic form $\mathbf{v}^{T}{H}_{c}\mathbf{v}$ for a set of tangent vectors $\mathbf{v}$ in $T_m(M)$. The tactic can be explained in a metaphor: If we know the cross-section of an object in each direction, we may infer the overall shape of the object.

Specifically, our estimation procedure for $H_c$ consists of two steps. Write $\ev_i$ for the $i$th coordinate unit vector in $\mathbb{R}^k$, and let $\cV = \{\ev_1, \ev_2, \ldots, \ev_k\} \cup\{ (\ev_i + \ev_j)/\sqrt{2} : 1\le i < j \le k \}$. 
\begin{itemize}
    \item[Step 1.] For the given set  $\cV$ of direction vectors, compute an estimator $\hat{a}_{\vv}$ of $a_{\vv} := \vv^T {H_c} \vv$ for each $\vv \in \cV$.
    \item[Step 2.] Using $\{ (\hat{a}_{\vv}, \vv): \vv \in \cV  \}$, solve the system of equations $\hat{a}_{\vv} = \vv^T \hat{H}_c \vv$, $\vv \in \cV$, 
to obtain $\hat{H}_c$. 
\end{itemize}
 
We now elaborate detailed procedures for each step.

\textbf{Step 1}. 
 For notational simplicity, the loss function for a given data point $x \in M$ is denoted by $g_x: \mathbb{R}^k \to \mathbb{R}$, so that 
$g_x(\xv) = \rho_c\{ d(x, \mbox{Exp}_{m}(\xv))\}$. Let $\vv \in T_m(M) \cong \mathbb{R}^k$ be a direction vector, satisfying $\|\vv\|_{m} = 1$.
The second-directional derivative of $g_x$ along $\vv$ at $\xv$ is 
$\partial_{\vv \vv}^2 g_x(\xv) = \vv^T \textsf{Hess} g_x(\xv) \vv$. 
Let $\gamma_{\vv}$ be the unit speed geodesic on $M$, passing through $m$ with the initial velocity vector $\vv$, which is parameterized by $t$ in some interval including zero, i.e., $\gamma_{\vv}(t) = \mbox{Exp}_m (t\vv)$. Then, 
\begin{align*}
    \vv^T \textsf{Hess}g_x(\mathbf{0}) \vv  = \partial_{\vv \vv}^2 \mid_{t= 0} g_x(\mathbf{0})  = \frac{\partial^2}{\partial t^2 } \mid_{t= 0} \rho_c\{ d(x, \gamma_{\vv}(t))\}.
\end{align*}
We define $\textsf{sn}_{\Delta}(s) = \tfrac{1}{\sqrt{\Delta} }\sin (s\sqrt{\Delta} )$ for $\Delta >0$, 
$s$ for $\Delta =0$, and $\tfrac{1}{\sqrt{|\Delta|} }\sinh (s\sqrt{|\Delta|} )$ for $\Delta <0$. From Equation (2.8) of \cite{afsari2011riemannian} and by an application of the chain rule for differentiation, we obtain, for $(x, m)$ such that $ d(x,m) \notin \{0,c\}$, 
\begin{align*}
    \frac{\partial^2}{\partial t^2 }\mid_{t= 0} \rho_c\{ d(x, \gamma_{\vv}(t))\} & = 
      2\left[ \cos^2(\alpha) + d(x,m)  \frac{\textsf{sn}_{\Delta}'(d(x,m))}{\textsf{sn}_{\Delta}(d(x,m))} \sin^2(\alpha)\right]1_{d(x,m) \le c} \\ 
       &\quad + 
       2c\left[ \frac{\textsf{sn}_{\Delta}'(d(x,m))}{\textsf{sn}_{\Delta}(d(x,m))} \sin^2(\alpha) \right]1_{d(x,m) > c}, 
\end{align*}
where $\alpha = \alpha(x,m,\vv)$ is the angle formed by the shortest geodesic from $x$ to $m$, and $\gamma_{\vv}$. The cosine and sine of $\alpha$ can be expressed using $\yv = \mbox{Log}_{m}(x) \in T_{m}M \cong \mathbb{R}^k$ as follows: 
\begin{align*}
    \cos(\alpha) = \frac{\yv^T \vv  }{ \|\yv\|_{m}} = \frac{\yv^T \vv  }{ d(x,m)}, ~ \sin^2(\alpha) =  1 - \cos^2(\alpha) = 1- \frac{(\yv^T \vv)^2}{d^2(x,m)}.
\end{align*}
Summarizing the arguments so far, for $m \in M$, a tangent vector $\vv \in T_m(M)$, and $x \in M$ such that $d(x, m) \notin \{0,c\}$, we have, with  $\yv = \mbox{Log}_{m}(x)$, 
\begin{equation}
    \label{eq:dir.second.deriv}
    \vv^T \textsf{Hess} g_{x}(\mathbf{0}) \vv 
    = 2\frac{(\yv^T \vv)^2  }{ \|\yv\|_{m}^2}1_{\|\yv\|_{m} \le c}  + 
    2  (\|\yv\|_{m} \wedge c)  \frac{\textsf{sn}_{\Delta}'(\|\yv\|_{m})}{\textsf{sn}_{\Delta}(\|\yv\|_{m})} \left\{ 1 - \frac{(\yv^T \vv)^2}{\|\yv\|_{m}^2}\right\}.
\end{equation}
For a sample $X_{1}, X_{2}, \ldots, X_{n} \in M$, we define 
\begin{equation}\label{eq:est.av}
\hat{a}_{\vv} = 
\frac{1}{n}\sum_{i=1}^n \vv^T \textsf{Hess} g_{X_{i}}(\mathbf{0}) \vv, 
\end{equation}
where $\vv^T \textsf{Hess} g_{X_{i}}(\mathbf{0}) \vv$ is given by (\ref{eq:dir.second.deriv}) with $x$ replaced by $X_i$. If $d(m, X_{i}) = 0$ or $c$ for some $i =1, 2, \ldots, n$, then the corresponding observation is excluded from the definition (\ref{eq:est.av}).

\textbf{Step 2}.
The set of ``basis" vectors $\cV = \{\ev_1, \ev_2, \ldots, \ev_k\} \cup\{ (\ev_i + \ev_j)/\sqrt{2} : 1\le i < j \le k \}$ are chosen so that $\{\vv\vv^T: \vv \in \cV\}$ forms the basis of the vector space  ${\rm Sym}(k)$ consisting of symmetric $k\times k$ real matrices. Define a vectorization operator $\mbox{vecd}: {\rm Sym}(k) \to \mathbb{R}^{(k+1)k/2}$ by 
\begin{align}
%\centering
  \mbox{vecd}(H) &= (h_{11},h_{22}, h_{12})^T  \ \mbox{for}\  k = 2,  \nonumber \\
  \mbox{vecd}(H) &= (h_{11},h_{22}, h_{33}, h_{12}, h_{13}, h_{23})^T  \ \mbox{for}\ k = 3, \label{eq:vecd}\\
  \mbox{vecd}(H) &= (h_{11},h_{22},\ldots,h_{kk}, h_{12},\ldots,h_{1k},h_{23},\ldots,h_{2k},\ldots, h_{k-1,k})^T, \nonumber
\end{align} 
for a $k \times k$ symmetric matrix $H$ whose $(i,j)$th element is $h_{ij} = h_{ji}$. Any ``$\mbox{vecd}(H)$" is an ordered list of ${(k+1)k}/{2}$ scalars. We index the elements of the tuple $\mbox{vecd}(H)$ by the subscript $(i,j)$ of $h_{ij}$. 

The equations obtained from Step 1 are 
$$\hat{a}_{\vv} = \vv^T H \vv = \mbox{trace}(H \vv\vv^T) = \mbox{vecd}(H)^T \mbox{vecd}(\vv\vv^T),$$
for each $\vv \in \cV$, where the symmetric matrix $H$ is the estimator we wish to obtain. There are exactly $(k+1)k/2$ elements in $\cV$, and we choose to index the elements of $\cV$ by the index set $\{ (i,j): 1\le i\le  j \le k\}$ used for indexing the elements of the tuple $\mbox{vecd}(H)$. That is, $\vv^{(i,i)} = \ev_i$ for $i=1, 2, \ldots, k$ and $\vv^{(i,j)} = (\ev_i + \ev_j) / \sqrt{2}$ for $1 \le i < j \le k$. Then, the system of equations we wish to solve can be written as 
\begin{equation}
    \label{eq:eqsystem}
    \hat\av = \bV \mbox{vecd}(H),
\end{equation}
where 
$\hat\av = (
\hat{a}_{\vv^{(1,1)}},\ldots, \hat{a}_{\vv^{(k-1,k)}}
)^T$
and 
$\bV$ is the $\tfrac{(k+1)k}{2} \times \tfrac{(k+1)k}{2}$ matrix, whose rows are given by  $\mbox{vecd}(\vv^{(i,j)} (\vv^{(i,j)})^T)$'s. By construction, the matrix $\bV$ is invertible, and the solution of the equation system is
\begin{equation}
    \label{eq:eqsystem2}
    \hat{H}_c = \mbox{vecd}^{-1}( \bV^{-1} \hat\av).
\end{equation}

\textbf{Estimation for $H_{c}^{-1}\Sigma_cH_c^{-1}$}. Our estimator for the limiting covariance matrix $H_{c}^{-1}\Sigma_cH_c^{-1}$ is given by
\begin{equation}\label{est:covariance}
\hat{H}_c^{-1} \hat\Sigma_c \hat{H}_c^{-1},
\end{equation}
where $\hat\Sigma_c$ and $\hat{H}_c^{-1}$ are defined in Equations (\ref{eq:Sigma_c_est}) and (\ref{eq:eqsystem2}), respectively.

\subsubsection{Proof of Theorem \ref{thm:cov:est}}

To prove Theorem~\ref{thm:cov:est}, the following lemma is needed.
\begin{lemma}\label{lem:Hessian:intermediate:result}
Suppose the assumptions of Theorem~\ref{thm:cov:est} are satisfied. Then, for any $\vv \in T_{m^{c}_0}M$ satisfying $\|\vv\|_{m^{c}_0} = 1$,  

(i) $\vv^T \textsf{Hess}_{g_{X}}(\mathbf{0}) \vv$ in (\ref{eq:dir.second.deriv}) is well-defined with probability 1;

(ii) As $n\to \infty$, $\hat{a}_\vv \to  \vv^T H_c \vv
$
in probability. 
\end{lemma}
\begin{proof}[Proof of Lemma~\ref{lem:Hessian:intermediate:result}]
(i) Since $P_X$ is absolutely continuous, $P(d(X,m) \in \{0,c\}) = 0$. For any $\vv \in T_m(M)$, since $\vv^T \textsf{Hess}_{g_{X}}(\mathbf{0}) \vv$ is well defined for $x$ satisfying $d(x,m) \notin \{0,c\}$, $\vv^T \textsf{Hess}_{g_{X}}(\mathbf{0}) \vv$ is well-defined with probability 1. 

(ii) It is known that the regular convexity radius of $M$,  $r_{\mbox{\tiny cx}}$, is lower bounded by $\frac{1}{2} \min\{\frac{\pi}{\sqrt{\Delta}}, r_{\mbox{\tiny inj}}(M)\}$ \citep[cf.][]{afsari2013convergence}. In particular, for any $x$ in the support of $P_X$, $ d(x,m) < r_{\mbox{\tiny cx}} \le \pi / (2{\sqrt{\Delta}})$, if $\Delta > 0$. 

We note that, for any given $c \in (0,\infty]$, $\frac{\partial^2}{\partial t^2 } \mid_{t= 0} \rho_c\{ d(x, \gamma_{\vv}(t))\}$ is uniformly bounded over any choice of the pair $m, x \in M$ such that $d(x,m) < r_{\mbox{\tiny cx}}$, and for any $\vv \in T_{m^c_0}M$, $\|\vv\|_{m^c_0} = 1$. This is can be checked by the following three arguments, which are true for all cases of $\Delta$.  

(1) There exists a constant $C <\infty$ such that, for any $ 0 \le \ell \le c $, it holds that 
$0 \le \ell \cdot \frac{\textsf{sn}_{\Delta}'(\ell)}{\textsf{sn}_{\Delta}(\ell)} \le C$. 

(2) The function $\ell \mapsto \frac{\textsf{sn}_{\Delta}'(\ell)}{\textsf{sn}_{\Delta}(\ell)}$ is a strictly decreasing on $[0, r_{\mbox{\tiny cx}})$ and is strictly positive on the same domain, thus is it bounded on $c < \ell < r_{\mbox{\tiny cx}}$. 

(3) Both $\cos^2(\alpha)$ and $\sin^2(\alpha)$ are bounded for any $\alpha \in \mathbb{R}$. 

Since the $n$ random variables $\frac{\partial^2}{\partial t^2 } \mid_{t= 0}\rho_c\{ d(X_i, \gamma_{\vv}(t))\}$, indexed by $i =1, 2, \ldots, n$ are $\textit{i.i.d.}$ and bounded, the weak law of large numbers gives that 
$$
\hat{a}_\vv \underset{n \to \infty}{\to} E [\vv^T \textsf{Hess} g_{X_1}(\mathbf{0}) \vv] = \vv^T E [\textsf{Hess} g_{X_1}(\mathbf{0})] \vv = \vv^T H_c \vv,
$$
as $n\to \infty$ in probability.
\end{proof}

\begin{proof}[Proof of Theorem \ref{thm:cov:est}]
For a given $0 < c \le \infty$, we shall show that 
\begin{equation}\label{est1}
\hat{\Sigma}_{c} := \frac{4}{n}\sum_{i=1}^{n} (\frac{\|\mbox{Log}_{m^{c}_{0}}(X_i)\|_{m^c_{0}} \wedge c}{\|\mbox{Log}_{m^{c}_{0}}(X_i)\|_{m^c_{0}}})^2 \mbox{Log}_{m^{c}_{0}}(X_i) \mbox{Log}_{m^{c}_{0}}(X_i)^{T}
\end{equation}
is a consistent estimator of $$\Sigma_{c}=\textsf{Cov}[\textsf{grad} \rho_{c}\{d(X, \mbox{Exp}_{m^{c}_0}(\mathbf{\mathbf{x}}))\}] = \textsf{Cov}[\textsf{grad}\rho_{c}\{d(X, m)\}],$$ where $\textsf{grad}$ stands for the (Riemannian) gradient with respect to the second argument. By Assumption (A4), $m^{c}_0$ is the unique minimum of $m \mapsto E [\rho_{c}\{d(X, m) \}]$, and thus $\textsf{grad}E [\rho_{c}\{d(X, m^{c}_{0})\}] = \mathbf{0}$. Since $E \{\textsf{grad}d^2(X, m^{c}_0)\} < \infty$ due to Assumption (A4), we obtain
\begin{equation}\label{integrable}
E [\textsf{grad}\rho_{c}\{d(X, m^{c}_0)\}] = \mathbf{0},
\end{equation}
by Lebesgue's dominated convergence theorem. By (\ref{integrable}),
\begin{eqnarray*}
\Sigma_{c} &=& \textsf{Cov}[\textsf{grad}\rho_{c}\{d(X, m^{c}_0)\}] = E [\textsf{grad}\rho_{c}\{d(X, m^{c}_0)\} \cdot \textsf{grad}\rho_{c}\{d(X, m^{c}_0)\}^{T}]\\ 
&=& 4E[(\frac{d(X, m^{c}_{0}) \wedge c}{d(X, m^{c}_{0})})^2 \mbox{Log}_{m^{c}_{0}}(X)\cdot \mbox{Log}_{m^{c}_{0}}(X)^T].
\end{eqnarray*}
Hence, (\ref{est1}) is a consistent estimator of $\Sigma_{c}$ by the strong law of large numbers. 

Secondly, we shall show that $\hat{H_{c}}$ is consistent estimator of $\hat{H_{c}}$. As shown in Lemma~\ref{lem:Hessian:intermediate:result}, each $\hat{a}_\vv$ converges to $\vv^TH_c \vv$ in probability as $n \rightarrow \infty$. Since $H_c$ is a linear combination of $a_\vv$'s and $\hat{H}_c$ is the same linear combination of $\hat{a}_\vv$'s, $\hat{H}_c \rightarrow H_c$ in probability as $n \to \infty$.

By the continuous mapping theorem, finally, $\hat{H}^{-1}_{c}$ converges to $H^{-1}_{c}$ in probability as $n \rightarrow \infty$. Due to Slutsky's theorem, it follows that $\hat{H}^{-1}_{c}\Sigma_{c}\hat{H}^{-1}_{c}$ is a consistent estimator of $H^{-1}_{c}\Sigma_{c}H^{-1}_{c}$. 
\end{proof}

\subsection{Proofs for Section \ref{subsec:application:hypo}}

\subsubsection{Proof of Corollary \ref{cor:hypotest}}
\begin{proof}[Proof of Corollary~\ref{cor:hypotest}]
By the continuous mapping theorem, we obtain that $\hat{A}_{c}^{-1/2} \overset{P}{\rightarrow} A_{c}^{-1/2}$ as $n \rightarrow \infty$; that is, $\hat{A}^{-1/2}_{c}= A^{-1/2}_{c} + o_{P}(1)$ as $n \rightarrow \infty$. Under the null hypothesis, $\sqrt{n} \phi_{\tilde{m_0}}(m^{c}_{n}) \overset{d}{\rightarrow} N_{k}(\mathbf{0}, A_{c})$ as $n \rightarrow \infty$ (by the result of Theorem~\ref{thm:clt}). Since any weakly convergent sequence of random vectors should be stochastically bounded, $\sqrt{n} \phi_{\tilde{m_0}}(m^{c}_{n}) = O_{P}(1)$ as $n\rightarrow \infty$. By the result of Theorem~\ref{thm:clt} again, we get $A^{-1/2}_{c} \sqrt{n} \phi_{\tilde{m_0}}(m^{c}_{n}) \overset{d}{\rightarrow} N_{k}(\mathbf{0}, I_{k})$ as $n\rightarrow \infty$. Due to Slutsky's theorem, we get
\begin{eqnarray*}
\hat{A}^{-1/2}_{c} \sqrt{n} \phi_{\tilde{m_0}}(m^{c}_{n}) 
&=& (A^{-1/2}_{c} + o_{P}(1))\sqrt{n} \phi_{\tilde{m_0}}(m^{c}_{n}) \\
&=& A^{-1/2}_{c} \sqrt{n} \phi_{\tilde{m_0}}(m^{c}_{n}) + \underbrace{o_{P}(1) \cdot O_{P}(1)}_{=o_{P}(1)}  \\ 
& \overset{d}{\rightarrow} & N_{k}(\mathbf{0}, I_{k}) \quad \mbox{as} ~ n \rightarrow \infty.
\end{eqnarray*}
Hence, it holds that $n\phi_{\tilde{m}_0}(m^{c}_{n})^{T} \hat{A}^{-1}_{c}\phi_{\tilde{m}_0}(m^{c}_{n}) \overset{d}{\rightarrow} \chi^{2}_{k}$ as $n \rightarrow \infty$. Under the null hypothesis, therefore, $P(n\phi_{\tilde{m}_0}(m^{c}_{n})^{T} \hat{A}^{-1}_{c}\phi_{\tilde{m}_0}(m^{c}_{n}) \ge \chi^{2}_{k, \alpha}) \underset{n \rightarrow \infty}{\rightarrow} \alpha$. It means that the type I error of this statistical test is suitably controlled for large sample sizes. 
\end{proof}

\subsubsection{Proof of Corollary \ref{cor:stat:power}}

\begin{proof}[Proof of Corollary~\ref{cor:stat:power}]

For $m_1 \in M$ and  $m_2,x \in B_{r_{\mbox{\tiny cx}}}({m}_1)$, the inequality (5.7) of \cite{kendall2011limit} gives 
$$
\| \mbox{Log}_{m_1}(x) - \Gamma_{m_2 \to m_1} \mbox{Log}_{m_2}(x) \|_{m_1} \le  
d(m_1,m_2) \cdot \sup_{m' \in B_{r_{\mbox{\tiny cx}}}(m_1) } \| \nabla \mbox{Log}_{m'}(x)\|_{m'}.
$$
As argued in \cite{Lin2022additive} (see the paragraph after Assumption 5), the quantity $\| \nabla \mbox{Log}_{m'}(x) \|_{m'}$ is bounded if the operator norm of the Hessian is bounded.  
Let $m_1 = \tilde{m}_0$. Since $\textsf{supp}(P_X) \subseteq B_{r_{\mbox{\tiny cx}}} (\tilde{m}_0)$, we have that the operator norm of Hessian of the squared distance function is uniformly bounded (which can be deduced from the proof of Theorem 5 in which we have seen that directional derivatives are uniformly bounded for any $x,m$ in the support of $P_X$). 
Since $d(\tilde{m}_0, \tilde{m}_n) = O(n^{-\beta})$ as $n \to \infty$, there exists an $N$ such that for all $n \ge N$, $d(\tilde{m}_0, \tilde{m}_n) < r_{\mbox{\tiny cx}}$.
Thus, replacing $m_2$ by $\tilde{m}_n$, there exists a $C>0$ satisfying 
\begin{equation}
\label{eq:corollary5_part_xxx}
    \| \mbox{Log}_{\tilde{m}_0}(x) - \Gamma_{\tilde{m}_n \to \tilde{m}_0} \mbox{Log}_{\tilde{m}_n}(x) \|_{\tilde{m}_0} \le  C \cdot
d(\tilde{m}_0,\tilde{m}_n), 
\end{equation}
 for all $x \in B_{r_{\mbox{\tiny cx}}}(\tilde{m}_0)$. 

Note that $d(\tilde{m}_0, m_n^c) \le d(\tilde{m}_0, \tilde{m}_n) + d(\tilde{m}_n, m_n^c) = O(n^{-\beta}) + O_P(n^{-1/2}) = O_P(n^{-\beta})$ as $n \to \infty$, which in turn leads that 
$\lim_{n\to \infty}P(m_n^c \in B_{r_{\mbox{\tiny cx}}}(\tilde{m}_0) \ | \ \textbf{H}_{1(n)} ) = 1$. 
Suppose now that the event  $m_n^c \in B_{r_{\mbox{\tiny cx}}}(\tilde{m}_0)$ (this event is called $G_n$ hereafter) has occurred. Then, from (\ref{eq:corollary5_part_xxx}) and the fact that  the parallel transport preserves
the Riemannian norm, we have
\begin{align*}
    \|\mbox{Log}_{\tilde{m}_0} (m_n^c) \|_{\tilde{m}_0} 
    & \le 
    \| \mbox{Log}_{\tilde{m}_0}(m_n^c) - \Gamma_{\tilde{m}_n \to \tilde{m}_0} \mbox{Log}_{\tilde{m}_n}(m_n^c) \|_{\tilde{m}_0} 
    + \|\Gamma_{\tilde{m}_n \to \tilde{m}_0} \mbox{Log}_{\tilde{m}_n} (m_n^c) \|_{\tilde{m}_0}  \\
    & \le C \cdot
d(\tilde{m}_0,\tilde{m}_n)
    + \| \mbox{Log}_{\tilde{m}_0} (m_n^c) \|_{\tilde{m}_0} 
\end{align*}
By the central limit theorem for sample Huber means (Theorem~\ref{thm:clt}), it holds that
$\mbox{Log}_{\tilde{m}_n} (m_n^c) = O_P(n^{-1/2})$ as $n \to \infty$. Thus, as $n \to \infty$
$$\|\mbox{Log}_{\tilde{m}_0} (m_n^c) \|_{\tilde{m}_0} = O_P(n^{-\beta} ) + O_P(n^{-1/2}) = O_P(n^{-\beta}),$$ which in turn leads that 
$$T_{n} = n\mbox{Log}_{\tilde{m}_0}(m^{c}_{n})^{T}\hat{A}^{-1}_{c} \mbox{Log}_{\tilde{m}_0}(m^{c}_{n}) = O_{P}(n^{1-2\beta})$$
as $n \rightarrow \infty$ under the event $G_n$. 
Note that $\frac{T_{n}}{n^{1-2\beta}} \ge 0$ for any sample size $n \ge 1$. 
Therefore, 
\begin{align*}
\liminf_{n \rightarrow \infty} & P(T_{n} \ge \chi^{2}_{k, \alpha} \ | \ \mathbf{H}_{1(n)})  \\ 
 = & \liminf_{n \rightarrow \infty} \left[  P( \{ T_{n} \ge \chi^{2}_{k, \alpha} \} \cap G_n\ | \ \mathbf{H}_{1(n)}) + 
P( \{ T_{n} \ge \chi^{2}_{k, \alpha} \} \cap G_n^c\ | \ \mathbf{H}_{1(n)})\right] \\ 
 \ge & \liminf_{n \rightarrow \infty} P\left(\frac{T_{n}}{n^{1-2\beta}} \ge \frac{\chi^{2}_{k, \alpha}}{n^{1-2\beta}} \ | \ \textbf{H}_{1(n)}\right) \\
=& \liminf_{n \rightarrow \infty} P\left(O_{P}(1) \ge \frac{\chi^{2}_{k, \alpha}}{n^{1-2\beta}} \ | \ \textbf{H}_{1(n)}\right) \\ 
=& 1.
\end{align*}
The inequality above is given by $P(G_n \ | \ \textbf{H}_{1(n)} ) \to 1$ as $n \to \infty$ and the last equality holds due to $\frac{\chi^{2}_{k, \alpha}}{n^{1-2\beta}} \underset{n \to \infty}{\to} 0$.
\end{proof}

\subsection{Technical details in Section \ref{sec:RE}}

\subsubsection{Technical conditions and proof of Equation (\ref{def:ARE})}

In this subsection, a technical condition is introduced to satisfy Definition~\ref{def:ARE}. Suppose the technical condition, i.e., $\sup_{n \ge 1}n^2 E \|\phi(m^{c}_{n}) \|^4 < \infty$ for $c=c_{1}, c_{2} \in (0, \infty]$ holds. Then, Equation (\ref{def:ARE}) is naturally satisfied, as stated next. 

\begin{proposition}\label{prop:tech:supp}
Under the same conditions in Theorem~\ref{thm:clt}. For given constants $c_{1}, c_{2} \in (0, \infty]$, if $\sup_{n \ge 1}n^2 E \|\phi(m^{c}_{n})\|^4 < \infty$ for $c=c_{1}, c_{2}$ holds, then we obtain
\begin{eqnarray*}
\textsf{ARE}_{\phi}(m^{c_{1}}_{n}, m^{c_{2}}_{n}) %= \lim_{n \rightarrow \infty} \textsf{RE}_{\phi}(m^{c_{1}}_{n}, m^{c_{2}}_{n}) 
%&=& \lim_{n \rightarrow \infty} \frac{\textsf{tr}\{\textsf{Cov}(\phi(m^{c_{2}}_{n}))\}}{\textsf{tr}\{\textsf{Cov}(\phi(m^{c_{1}}_{n}))\}}
%= \lim_{n \rightarrow \infty} \frac{\textsf{tr}\{\textsf{Cov}(\sqrt{n}\phi(m^{c_{2}}_{n}))\}}{\textsf{tr}\{\textsf{Cov}(\sqrt{n}\phi(m^{c_{1}}_{n}))\}} \\ 
=\frac{\textsf{tr}(H_{c_{2}}^{-1}\Sigma_{c_{2}} H_{c_{2}}^{-1})}{\textsf{tr}(H_{c_{1}}^{-1}\Sigma_{c_{1}} H_{c_{1}}^{-1})}.
\end{eqnarray*}
The next lemma is required to prove the proposition.
\begin{lemma}\label{lem:moment}
For a sequence of real-valued random variables $(X_n)_{n \ge 1}$, suppose that $\sup_{n \ge 1} EX^{4}_{n} < \infty$ and $X_{n} \to X$ in distribution as $n \to \infty$. Then, $EX^2_{n} \underset{n \to \infty}{\to} EX^2$.
\begin{proof}[Proof of Lemma~\ref{lem:moment}]
By Portmanteau theorem, we have $E(X^{2}_{n} \wedge M) \underset{n \to \infty}{\to} E(X^{2} \wedge M)$ for any $M>0$. We note that, for any $n\ge 1$ 
$$
X^2 = (X^2 \wedge M) + (X^2 - M) \cdot 1_{X^2 \ge M},~ X_{n}^2 = (X_{n}^2 \wedge M) + (X_{n}^2 - M) \cdot 1_{X_{n}^2 \ge M}. 
$$
An application of Fatou's inequality leads that, for a given $M>0$
\begin{eqnarray*}
E(X^2-M) \cdot 1_{X^2 \ge M} \le E(X^2 \cdot 1_{X^2 \ge M}) &\le& E(X^2 \cdot \frac{X^2}{M})
\le \frac{EX^4}{M} \\
&\le& \frac{\liminf_{n\to \infty}EX^{4}_{n}}{M} \\
&\le& \frac{\sup_{n \ge 1} EX^{4}_{n}}{M},
\end{eqnarray*}
and 
\begin{eqnarray*}
E(X_{n}^2-M)\cdot 1_{X^2 \ge M} \le EX_{n}^2 \cdot 1_{X_{n}^2 \ge M} &\le&  E(X_{n}^2 \cdot \frac{X^2}{M}) \\
&\le& \frac{EX_{n}^4}{M} \\
&\le& \frac{\sup_{n \ge 1} EX_{n}^{4}}{M}.
\end{eqnarray*}
With this in mind,
\begin{eqnarray*}
\limsup_{n \to \infty} |EX^2 - EX^2_{n}| &\le& \limsup_{n \to \infty} |E(X^2 \wedge M) - E(X^2_{n}\wedge M)| \\ 
& & + \limsup_{n \to \infty}E|(X^2 - M)\cdot 1_{X^2 \ge M}| \\ 
& & + \limsup_{n \to \infty} E|(X^2_{n} \wedge M)\cdot 1_{X_{n} \ge M}| \\
&\le& \frac{2\sup_{n \ge 1} EX^{4}_{n}}{M} \\ & \underset{M \to \infty}{\to}& 0.
\end{eqnarray*}
Since $M$ was arbitrarily chosen, we get $$\lim_{n \to \infty} |EX^2 - EX^2_{n}|= \limsup_{n \to \infty} |EX^2 - EX^2_{n}| = 0,$$ by letting $M \to \infty$.
\end{proof}
\end{lemma}

The proof of Proposition~\ref{prop:tech:supp} simply involves applying the results from Theorem~\ref{thm:clt} and Lemma~\ref{lem:moment}, as detailed below.
\begin{proof}[Proof of Proposition~\ref{prop:tech:supp}]
Now enter into the proof of Proposition~\ref{prop:tech:supp}. From the result of Theorem~\ref{thm:clt}, we obtain that $\sqrt{n}\phi(m^{c}_{n}) \overset{d}{\to} N_{k}(\mathbf{0}, H^{-1}_{c}\Sigma_{c}H^{-1}_{c})$ as $n\to\infty$ for $c= c_{1}$ and $c_2$. By a coordinate-wise application of Lemma~\ref{lem:moment} to the previous fact, we obtain that $\textsf{tr}\{\textsf{Cov}(\sqrt{n} \phi(m^{c}_{n}))\} \underset{n \to \infty}{\to} \textsf{tr}(H^{-1}_{c}\Sigma_{c}H^{-1}_{c})$ for $c=c_1$ and $c_2$. Therefore, 
\begin{eqnarray*}
\textsf{ARE}_{\phi}(m^{c_{1}}_{n}, m^{c_{2}}_{n}) = \lim_{n \rightarrow \infty} \textsf{RE}_{\phi}(m^{c_{1}}_{n}, m^{c_{2}}_{n}) 
&=& \lim_{n \rightarrow \infty} \frac{\textsf{tr}\{\textsf{Cov}(\phi(m^{c_{2}}_{n}))\}}{\textsf{tr}\{\textsf{Cov}(\phi(m^{c_{1}}_{n}))\}} \\
&=& \lim_{n \rightarrow \infty} \frac{\textsf{tr}\{\textsf{Cov}(\sqrt{n}\phi(m^{c_{2}}_{n}))\}}{\textsf{tr}\{\textsf{Cov}(\sqrt{n}\phi(m^{c_{1}}_{n}))\}} \\ 
&=&\frac{\textsf{tr}(H_{c_{2}}^{-1}\Sigma_{c_{2}} H_{c_{2}}^{-1})}{\textsf{tr}(H_{c_{1}}^{-1}\Sigma_{c_{1}} H_{c_{1}}^{-1})},
\end{eqnarray*}
which completes the proof.
\end{proof}
\end{proposition}

\subsubsection{A rationale for selecting the robustification parameter}

In Section~\ref{sec:RE}, we proposed to select the parameter $c$ of the Huber loss function as the smallest $c$ such that the asymptotic efficiency of $m_n^c$ is at least 95\% when compared to the Fr\'{e}chet mean under the Gaussian-type distribution on $M$. The proposition below supports that such an attempt is worthwhile.

\begin{proposition}\label{prop:c:worthwhile:supp}
Suppose that $P_{X}$ satisfies Assumptions (A1), (A3), and (A4) for all $c \in (0, \infty]$, and $(\ref{def:ARE})$ always holds. Then, there exists a constant $c_{0} \in [0, \infty)$ such $\textsf{ARE}_{\phi}(m^{c_{0}}_{n}, m^{\infty}_{n}) \ge 0.95$.
\begin{proof}[Proof of Proposition~\ref{prop:c:worthwhile:supp}]
It is enough to show that 
\[
\lim_{c \rightarrow \infty}\textsf{ARE}_{\phi}(m^{c}_{n}, m^{\infty}_{n}) = \lim_{c \rightarrow \infty} \frac{\textsf{tr}(H_{\infty}^{-1}\Sigma_{\infty} H_{\infty}^{-1})}{\textsf{tr}(H_{c}^{-1}\Sigma_{c} H_{c}^{-1})}=1.
\]
It is obtained that $\textsf{Hess}E[\rho_{\infty}\{d(X, \phi_{m^{\infty}_{0}}(\mathbf{x}))\}]=E[\textsf{Hess}\rho_{\infty}\{d(X, \phi_{m^{\infty}_{0}}(\mathbf{x}))\}]$ for $\mathbf{x}$ near $\mathbf{0}$ (by Assumption (A3) for $c=\infty$). Note that $H_{\infty}=E[\textsf{Hess}\rho_{\infty}\{d(X, \phi^{-1}_{m^{\infty}_{0}}(\mathbf{0}))\}]$ is positive-definite since $E[\rho_{\infty}\{d(X, \phi^{-1}_{m^{\infty}_{0}}(\mathbf{x}))\}]$ has a global minimum at $\mathbf{x}=\mathbf{0}$ (by Assumption (A3) for $c=\infty$); that is, its determinant is positive, $\mbox{det}(H_{\infty}) > 0$. Since $H^{-1}_{\infty}$ exists, it now suffices to prove that $\Sigma_{c} \rightarrow \Sigma_{\infty}, H_{c} \rightarrow H_{\infty}$ as $c \rightarrow +\infty$, which follows from the combination of Lemma~\ref{lem:unif:conv} and Lebesgue's dominated convergence theorem.
\end{proof}
\end{proposition}

The above proposition states that, for \textit{any} underlying distribution, the robustification parameter $c$ can be tuned so that the sample Huber mean is at least 95\% efficient relative to the Fr\'{e}chet mean.

Using Proposition~\ref{prop:c:worthwhile:supp}, we next provide an example for the choice of $c$ on the real line. %and the unit circle.\footnote{Any connected one-dimensional manifold is homeomorphic to $\bbR$ or $S^1$.} 
For the case $M=\bbR$, the robustification parameter $c = 1.345\sigma$ has been the traditional choice for $c$, ensuring the ARE of 95\% against the sample mean \citep{holland1977robust, huber2004robust}. 
 In the following Example~\ref{0.95efficiency_R}, we confirm that utilizing our expression for the ARE in Definition~\ref{def:ARE}, obtained from the asymptotic results in Theorem~\ref{thm:clt}, also provides the same figure.

 \begin{example}\label{0.95efficiency_R}
 Let $M=\mathbb{R}$, $\mu=0$, and the reference measure be the Lebesgue measure in $\mathbb{R}$. A local coordinate chart $(\phi_{\mu}, U)$ near $\mu$ with $U=(-\infty, \infty)$ and $\phi_{\mu}: x \mapsto x$. For $\sigma>0$, if $X \sim N(0, \sigma^2)$, then it follows that $\Sigma_{c}=4[E\{X^2\cdot 1_{|X| \le c}\} + c^2P(|X| > c)]$, $\Sigma_{\infty}= 4\textsf{Var}(X)=4\sigma^2$, $H_{c} = 2P(|X| \le c)$, and $H_{\infty}=2$. From (\ref{def:ARE}), we have 
 \begin{equation*}
     \textsf{\mbox{ARE}}_{\phi_{\mu}}(m^{c}_{n}, m^{\infty}_{n})  = \frac{\sigma^2 \{P(|X| \le c)\}^2}{E\{X^2 \cdot 1_{|X| \le c}\} + c^2 P(|X| > c)}  
    = \frac{\{P(|Z| \le \kappa)\}^2}{E\{Z^2 \cdot 1_{|Z| \le \kappa}\} + \kappa^2 P(|Z| > \kappa)},
 \end{equation*}
 where $Z \sim N(0, 1)$ and $\kappa = c/\sigma$. As shown in in the top left panel of Figure~\ref{fig:cutoff}, $\textsf{\mbox{ARE}}_{\phi_{\mu}}(m^{c}_{n}, m^{\infty}_{n})$ is an increasing function of $c$, and we have confirmed that the Huber mean for $c=1.345 \sigma$ possesses 95\% ARE when compared to the sample mean, i.e., $\textsf{\mbox{ARE}}_{\phi_{\mu}}(m^{1.345 \sigma}_{n}, m^{\infty}_{n}) = 0.95$. 
 \end{example}

\subsubsection{Asymptotic relative efficiency of Huber means on the real line and spheres}

In this subsection, we provide an example under which Huber means are asymptotically more efficient than \Frechet\ mean on $\mathbb{R}$, and discuss the choice of $c$ on $S^{2}$ and $S^3$. In multiple regression settings where $M = \bbR^{k}$ (as a descriptor space), under heavy-tailed distribution on random noises, it is known that estimators obtained from the Huber loss are more efficient than ones obtained from the $L_{2}$ loss \citep[cf.][]{huber2004robust, zhou2023enveloped}. 

First, in this respect, we show below that Huber means are asymptotically more efficient than the Fréchet mean under Laplace distribution on $\bbR$.

\begin{proposition}\label{prop:M-type}
Suppose that $\frac{dP_{X}}{dV} = f_{\mu, \sigma, 0}$ as specified in (\ref{distr:huber:mfd}) with $M=\bbR$. For a given $c \in (0, \infty)$, suppose that $\rho_{c}$ is used. Then, Huber means are asymptotically more efficient than the Fréchet mean, i.e., $\textsf{ARE}_{\phi}(m^{c}_{n}, m^{\infty}_{n}) > 1$, where $\phi:\bbR \rightarrow \bbR, x \mapsto x$.
\begin{proof}[Proof of Proposition~\ref{prop:M-type}]
The probability density function of the Laplace distribution is given by $f_{\mu, \sigma, 0}(x) = \frac{1}{2\sigma} e^{-|x - \mu|/\sigma}$ for any $x \in \bbR$. By the translation invariance of Laplace distribution, without loss of generality, we may assume that $\mu=0$. It follows that $H_{c}=2P(|X| \le c)$, $\Sigma_{c}=4E\{X^2 \cdot 1_{|X| \le c}\} + 4c^2 P(|X| > c)$, $\Sigma_{\infty}=4\textsf{Var}(X) = 4EX^2$, and $H_{\infty}=2$. Hence, 
\begin{eqnarray*}
\textsf{ARE}_{\phi}(m^{c}_{n}, m^{\infty}_{n}) &=& \frac{(H_{c})^2 \Sigma_{\infty}}{(H_{\infty})^2 \Sigma_{c}} = \frac{EX^2 \cdot \{P(|X| \le c)\}^2}{E\{X^2 \cdot 1_{|X| \le c}\} + c^2P(|X| > c)} \\
&=& \frac{2\sigma^2(1-e^{-c/\sigma})^2}{2\sigma^2-(c^2 + 2\sigma c + 2\sigma^2)e^{-c/\sigma} + c^2e^{-c/\sigma}} \\
&>& 1,
\end{eqnarray*}
which is equivalent to $(\sigma + c)e^{c/\sigma} > \sigma$. Therefore, the inequality holds for any $c > 0$. 
\end{proof}
\end{proposition}

The proposition above holds specifically for the Huber loss.
While we initially thought that this result had already been established elsewhere, to the best of our knowledge, we could not find a proof of Proposition~\ref{prop:M-type} in the existing literature.
We believe that our proof provides the first rigorous derivation of this result.

%%%
Second, we derive the ARE of Huber means on $S^k ~ (k\ge2)$ (compared to the \Frechet\ means with respect to Gaussian-type distributions) and discuss the choice of $c$ for specific cases $k=2, 3$ via numerical evaluations. The formula for ARE on $S^k$ is presented in detail in Example~\ref{0.95efficiency_spheres} below.

\begin{example}[Asymptotic relative efficiency of Huber means on $S^k$]\label{0.95efficiency_spheres}
The $k$-dimensional unit sphere is represented by $S^{k}=\{(x_1,x_2, \ldots, x_{k+1}) \in \bbR^{k+1}: \sum_{i=1}^{k+1} x^2_i = 1\}$ for $k \ge 2$. For a scale parameter $\sigma>0$, we consider a random variable $X$ following the Gaussian-type distribution $f^G_{\mu,\sigma}$ on $S^k$, where the location parameter is set to be $\mu = (1, 0, 0, \ldots, 0)^{T} \in S^k$. In a normal coordinate chart at $\mu$, we denote $Y= \mbox{Log}_{\mu}(X)$. Then, an application of \citep[Lemma 3,][]{kim2020kurtosis} leads that the density function of $Y$ is  
$
f_{Y}(\mathbf{y}) = C_{k, \sigma} \cdot \exp(-\|\mathbf{y}\|^2/(2\sigma^2)) \cdot (\sin(\|\mathbf{y}\|)/\|\mathbf{y}\|)^{k-1},
$
in which $C_{k, \sigma}$ is a normalization constant depending on $k$ and $\sigma$. Using this formula and hyper-spherical coordinate system on $T_{\mu}S^k$, we obtain that, for $X\sim f^{G}_{\mu, \sigma}$, 
\begin{equation}\label{2d3d_hessianandcovariance}
\begin{aligned}
\Sigma_{c} &= 2C\big[\int_{0}^{c} r^2\sin^{k-1}(r)\cdot e^{-r^2/(2\sigma^2)} dr +c^2 \int_{c}^{\pi} \sin^{k-1}(r)\cdot e^{-r^2/(2\sigma^2)} dr\big]I_k, \\ 
\Sigma_{\infty} &= C [\int_{0}^{\pi} r^2\sin^{k-1}(r)\cdot e^{-r^2/(2\sigma^2)} dr]I_{k}, \\
H_{c} &= 2C\big[(\int_{0}^{\pi} \sin^{k-1}(\theta_1)\cos^2(\theta_1)d\theta_1)(\int_{0}^{c} \sin^{k-1}(r)\cdot e^{-r^2/(2\sigma^2)}dr) \\
& + (\int_{0}^{\pi} \sin^{k+1}(\theta_1)d\theta_1)(\int_{0}^{c} r\sin^{k-2}(r)\cos(r)\cdot e^{-r^2/(2\sigma^2)}dr) \\
& + c(\int_{0}^{\pi} \sin^{k+1}(\theta_1)d\theta_1)(\int_{c}^{\pi} \sin^{k-2}(r)\cos(r)\cdot e^{-r^2/(2\sigma^2)}dr)\big]I_{k}, ~ \mbox{and} \\
H_{\infty} &= C\big[(\int_{0}^{\pi} \sin^{k-1}(\theta_1)\cos^2(\theta_1)d\theta_1)(\int_{0}^{\pi} \sin^{k-1}(r)\cdot e^{-r^2/(2\sigma^2)}dr) \\
& + (\int_{0}^{\pi} \sin^{k+1}(\theta_1)d\theta_1)(\int_{0}^{\pi} r\sin^{k-2}(r)\cos(r)\cdot e^{-r^2/(2\sigma^2)}dr)\big]I_{k}
\end{aligned}
\end{equation}
where $C= 4\pi C_{k, \sigma} \cdot \prod_{2\le j \le k-1} \int_{0}^{\pi} \sin^{k-j}(\theta_j)d\theta_j$ and $I_k$ is the $k \times k$ identity matrix. Since these quantities are scaled identity matrices, $\textsf{ARE}(m^{c}_{n}, m^{\infty}_{n})$ can be numerically evaluated for each value of $c$ and $\sigma$. 

%Thanks to the formula $$\textsf{ARE}(m^{c}_{n}, m^{\infty}_{n})= (\frac{H_{c}}{H_{\infty}})^2\cdot (\frac{\Sigma_{\infty}}{\Sigma_{c}})$, under Gaussian-type distributions with scale= 0.1, 0.5, 1 on $S^2$ and $S^3$, ARE curves are drawn in the top-left and bottom-right panels of \Cref{fig:ARE:2sphere:3sphere}.
\end{example} 

Utilizing (\ref{2d3d_hessianandcovariance}), we now investigate the choice of $c$ as the minimum value for which the Huber mean is at least 95\% relatively efficient compared to Gaussian-type distributions (see (\ref{eq:c_hat}) in the main article). For Gaussian-type distributions on $S^2$ and $S^3$ (with scale parameter set as $\sigma = 0.1, 0.5$ or $1$), we have numerically evaluated $\textsf{ARE}(m^{\kappa \sigma}_{n}, m^{\infty}_{n})$ over a fine grid of $\kappa>0$. These numerically evaluated AREs are plotted in Figure \ref{fig:ARE:2sphere:3sphere}. 
The smallest value of $\kappa$ achieving  95\% ARE varies, depending on the value of the scale parameter $\sigma$. These range $\kappa \in [1.21, 1.50]$ for $S^2$ and  $\kappa \in [1.14, 1.62]$ for $S^3$. Overall, our recommend choice of $\kappa = 1.345$ (or $c = 1.345\sigma$) appear to be reasonable.

\begin{figure}[th]
    \centering
    \includegraphics[scale=0.25]{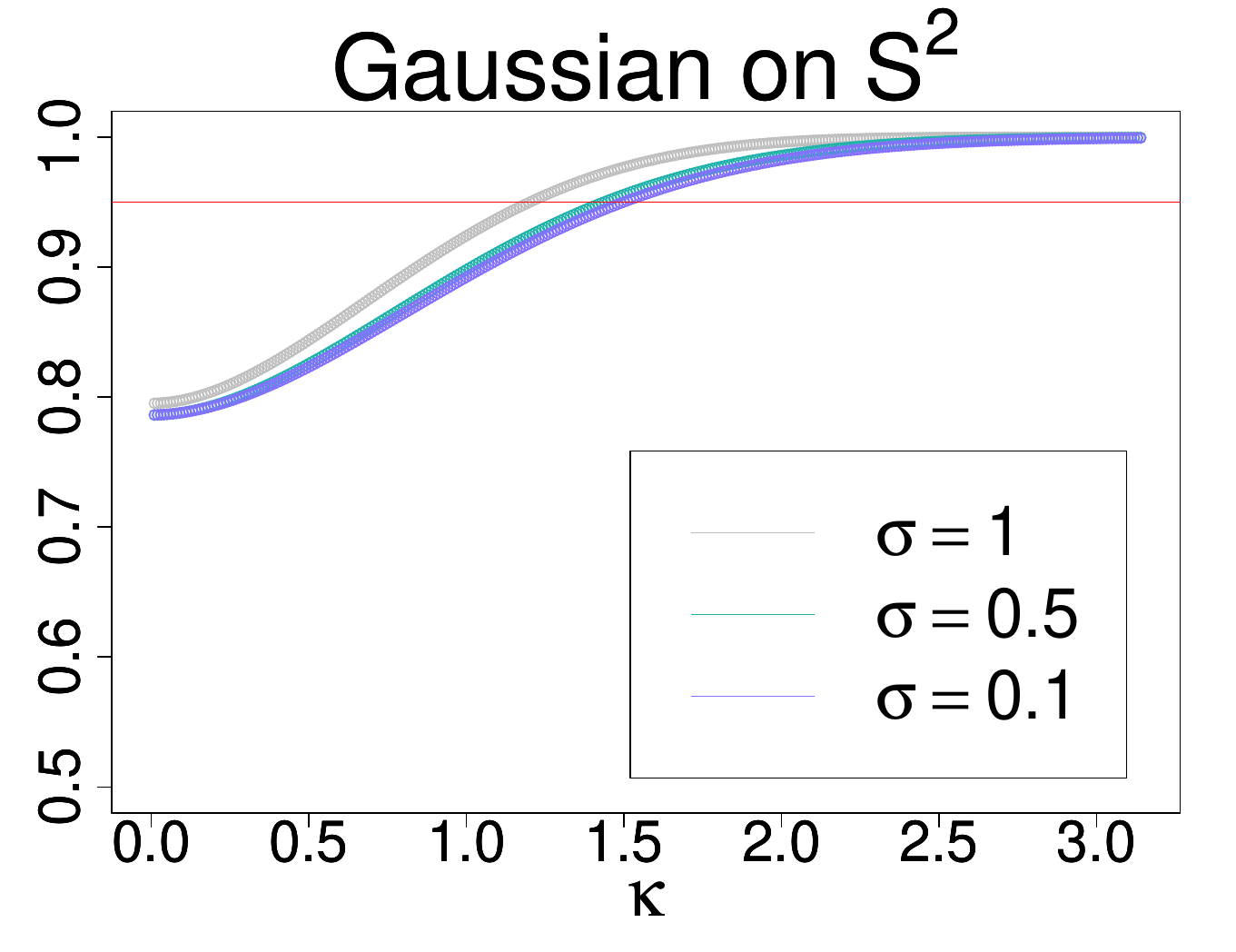}
    \includegraphics[scale=0.25]{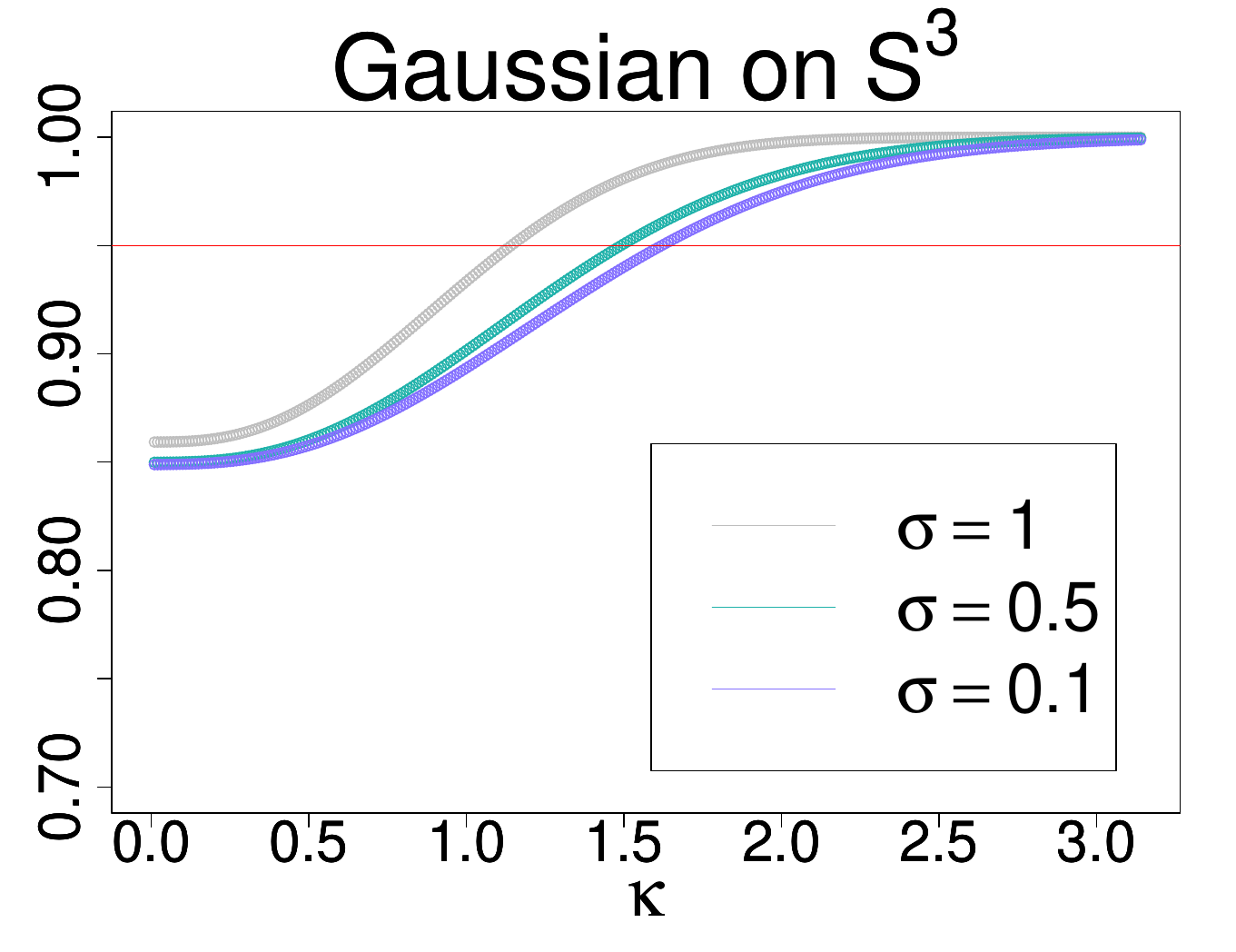}
    \includegraphics[scale=0.25]{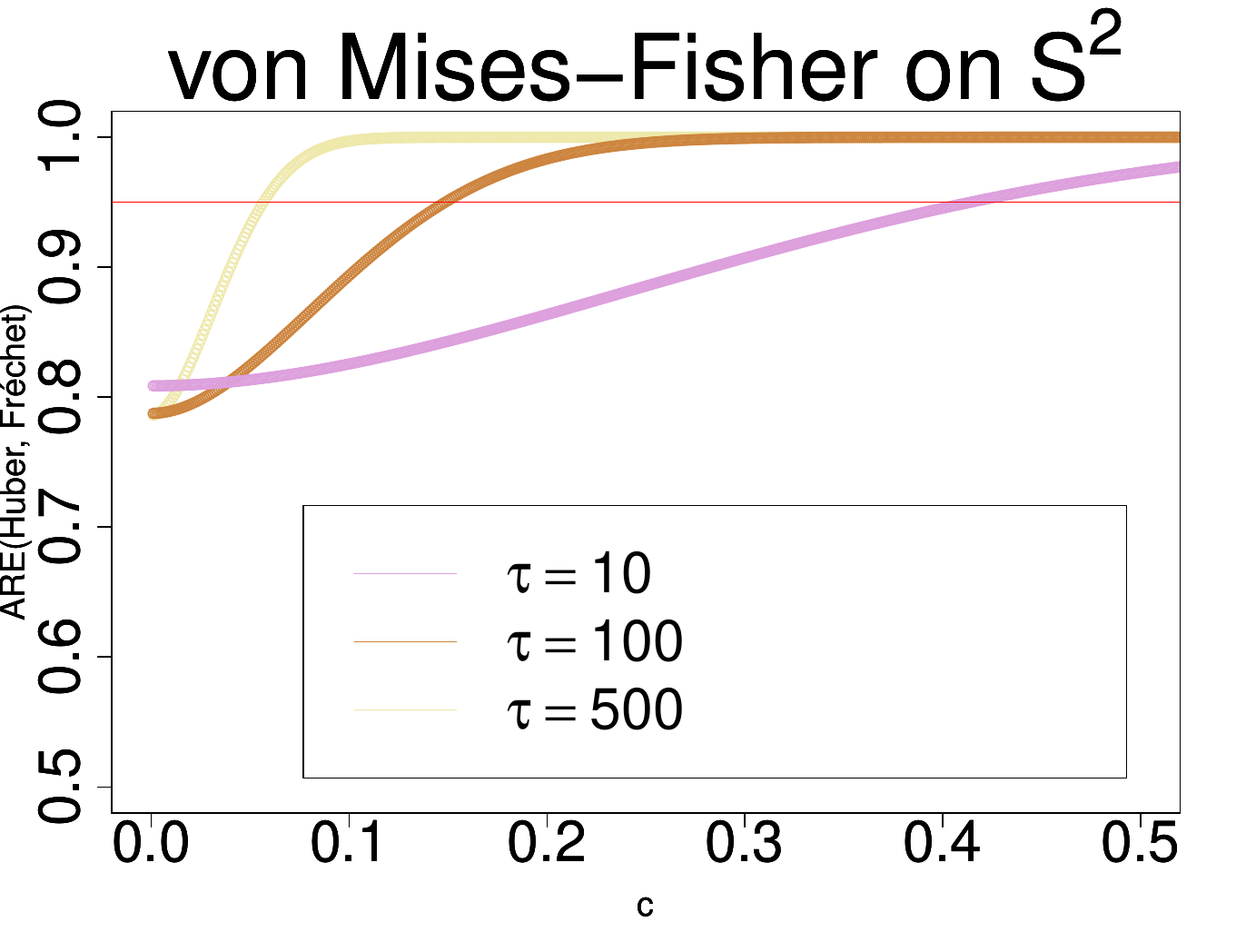}
    \includegraphics[scale=0.25]{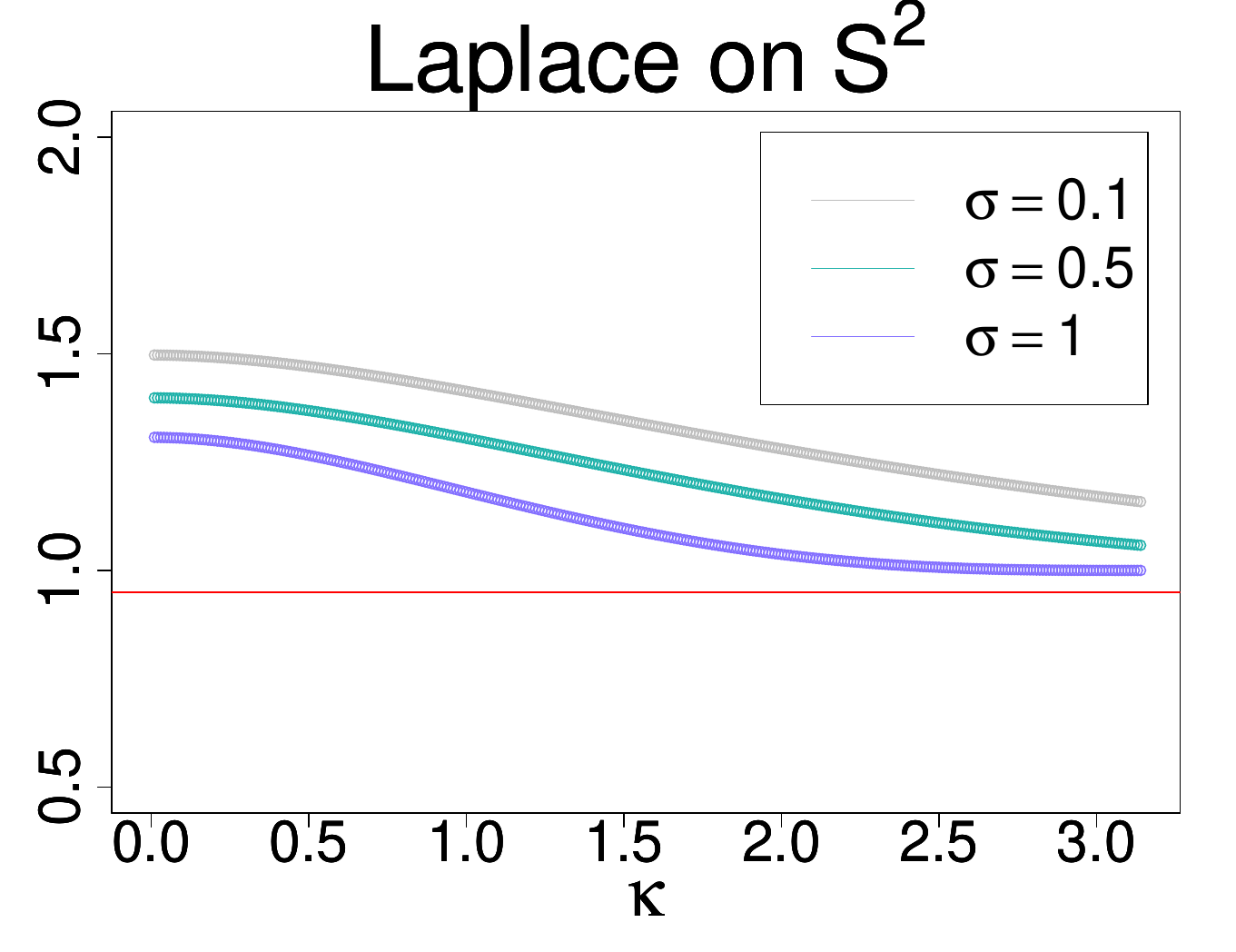}
    \caption{Graphs of $\textsf{ARE}(m^{c}_{n}, m^{\infty}_{n})$ over various choices of robustification parameter $c$. In each panel, the horizontal red line represents the value of 95\% relative efficiency compared to the \Frechet\ mean. 
    (Top) AREs computed over a grid of $\kappa = c/\sigma$ under Gaussian-type distributions with scale parameter $\sigma$ on $S^2$ and $S^3$. (Bottom left) AREs computed over a find grid of $c$ under the von Mises-Fisher distributions with concentration parameter $\tau$ on $S^2$ (Bottom right) AREs over a grid of $\kappa = c/\sigma$ under Laplace-type distributions on $S^2$ with scale parameter $\sigma$. }
    \label{fig:ARE:2sphere:3sphere}
\end{figure}
%    The values of $\textsf{ARE}(m^{c}_{n}, m^{\infty}_{n})$ are numerically evaluated over a fine grid for $\kappa=c/\sigma$ under Gaussian distributions on $S^2$ for $\sigma=0.1, 0.5, 1$ (or von Mises-Fisher distributions for $\tau=10, 100, 500$, respectively), with the $x$-axis representing the $\kappa$ as specified in Examples 1 and 2. (Bottom-left) The ARE of versus $\kappa$ under the Laplace distribution for each choice of $\tau = 0.1, 0.5, 1$ on $S^2$. (Bottom-right) The values of $\textsf{ARE}(m^{c}_{n}, m^{\infty}_{n})$ are numerically evaluated over a fine grid for $\kappa=c/\sigma$ under Gaussian distributions for $\sigma=0.1, 0.5, 1$ on $S^3$.

 In Figure \ref{fig:ARE:2sphere:3sphere}, we have also plotted the AREs of Huber means when the underlying distributions are von Mises-Fisher and Laplace. We point out that the Huber mean is asymptotically more efficient than the \Frechet\ mean under Laplace-type distributions on $S^2$.

\subsection{Proofs for Section \ref{sec:robust} and additional results}

Before verifying Theorems \ref{thm:robust:breakdown} and \ref{thm:breakdown:0.5}, we shall show that the population and sample Huber means for any $c \in [0, \infty]$ are isometric-equivariant, which is a desirable property of location estimators.

Recall that the uniqueness of Huber means is assumed in Section~\ref{sec:robust}.

\begin{lemma}\label{lem:huber:isoequi}
The population and sample (pseudo) Huber means for any $c \in [0, \infty]$ are isometric-equivariant.
\begin{proof}[Proof of Lemma~\ref{lem:huber:isoequi}]
Let $c \in [0, \infty]$ be prespecified, and $\psi$ be an isometry on $M$. 

\textbf{Case I} (Huber means). Suppose that $\rho_{c}$ is used. Since $m^{c}_{n}(\mathbf{X})$ minimizes the map $$m \mapsto \sum_{x \in \mathbf{X}} \rho_{c}\{d(m, x)\} = \sum_{x \in \mathbf{X}} \rho_{c}\{d(\psi(m), \psi(x))\} = \sum_{y \in \psi(\mathbf{X})} \rho_{c}\{d(\psi(m), y)\},$$ it follows that $\psi\{m^{c}_{n}(\bX)\}$ minimizes the map $m \mapsto \sum_{y\in \psi(\mathbf{X})} \rho_{c}\{d(m, y)\}$. We hence obtain that $m^{c}_{n}\{\psi(\bX)\} = \psi\{m^{c}_{n}(\bX)\}$. In other words, the sample Huber mean is isometric-equivariant.

We next verify the isometric-equivariance of \emph{population} Huber means. To this end, define $\psi(P_{X}):=P_{X} \circ \psi^{-1}$ that is a push-forward measure (distribution) of $P_{X}$ on $M$. Denote $m_{0}^{c}(P_{X})$ by the population Huber mean for $c$ with respect to $P_{X}$. Since $m_{0}^{c}(P_{X})$ minimizes the map 
\begin{eqnarray*}
m \mapsto \int \rho_{c}\{d(m, x)\}dP_{X}(x) &=& \int \rho_{c}\{d(\psi(m), \psi(x))\}dP_{X}(x) \\
&=& \int \rho_{c}\{d(\psi(m), y)\}d(P_{X}\circ \psi^{-1})(y),
\end{eqnarray*}
it follows that $\psi\{m_{0}^{c}(P_{X})\}$ minimizes the map $m \mapsto \int \rho_{c}\{d(m, y)\}d(P_{X} \circ \psi^{-1})(y)$. Therefore $m^{c}\{\psi(P_{X})\} = \psi\{m^{c}(P_{X})\}$. In other words, the population Huber mean for $c$ with respect to $P_{X}$ is isometric-equivariant.

\textbf{Case II} (Pseudo Huber means). An argument similar to the above verifies the isometric equivariance of (sample and population) pseudo-Huber means for any $c \in [0, \infty]$ by replacing $\rho_{c}$ with $\tilde{\rho_{c}}$.
\end{proof}
\end{lemma}

%We are now prepared to prove Theorems \ref{thm:robust:breakdown} and \ref{thm:breakdown:0.5}.

\subsubsection{Proof of Theorem \ref{thm:robust:breakdown}}

\begin{proof}[Proof of Theorem \ref{thm:robust:breakdown}]

We use an argument motivated by \citep[Theorem 2.2,][]{lopuhaa1991breakdown}. When $c=0$, this is proven by an application of \citep[Theorem 2,][]{fletcher2009geometric}. For this reason, assume that $c \in (0, \infty)$. Let $\bY_{k}$ be a corrupted set of observations from $\bX$. Specifically, $\bY_{k}$ includes $k$ arbitrary points on $M$ and $(n-k)$ points from $X$. Firstly we shall show that, for a given $k \le \floor{\frac{n-1}{2}}$ and any $\bY_{k}$, $d(m^{c}_{n}(\bX), m^{c}_{n}(\bY_{k}))$ is bounded by a constant that does not depend on the choice of $\bY_{k}$. Without loss of generality, assume that $\bY_{k}=(y_1, y_2, ..., y_n)\in M^n$ and each of $y_i ~ (1 \le i \le k)$ is an arbitrary point on $M$, $y_i=x_i$ for $k+1 \le i \le n$. Let $R := \max_{1 \le i
\le n}d(x_{i}, m^{c}_{n}(\bX))$, $B:=\overline{B_{2R}(m_{n}^{c}(\bX))}=\{p \in M : d(p, m^{c}_{n}(\bX)) \le 2R \}$, and $\delta := \max\{\inf_{p \in B} d(p, m^{c}_{n}(\bY_{k})), c\}$. Note that $0 \le \delta < \infty$. Due to triangle inequality, we also get $d(m^{c}_{n}(\bX), m^{c}_{n}(\bY_k)) \le 2R + \delta$. Here, denote $\rho_{c}(t)=0$ for $t \le 0$. If $y_i \neq x_i$, in view of the triangle inequality and the monotonicity of $\rho_{c}$, then
\begin{eqnarray}\label{eq:different}
d(y_i, m^{c}_{n}(\bY_{k})) 
&\ge& d(y_i, m^{c}_{n}(\bX)) - d(m^{c}_{n}(\bX), m^{c}_{n}(\bY_k)) \nonumber \\ 
&\ge& d(y_i, m^{c}_{n}(\bX)) - (2R + \delta) \nonumber \\
\Rightarrow \rho_{c}\{d(y_i, m^{c}_{n}(\bY_{k}))\} & \ge & \rho_{c}\{d(y_i, m^{c}_{n}(\bX)) - (2R + \delta)\}.  
\end{eqnarray}
If $y_i = x_i$, in view of the monotonicity of $\rho_{c}$, then
\begin{eqnarray}\label{eq:same}
d(y_i, m^{c}_{n}(\bY_k)) 
& = & d(x_i, m^{c}_{n}(\bY_k)) \ge R + \delta \ge d(y_i, m^{c}_{n}(\bX)) + \delta \nonumber \\
\Rightarrow \rho_{c}\{d(y_i, m^{c}_{n}(\bY_k))\} 
& \ge & \rho_{c}\{d(y_i, m^{c}_{n}(\bX)) + \delta\}.  
\end{eqnarray}
Note that $2ct-c^2 \le \rho_{c}(t)$ for any $t\in (-\infty, \infty)$ and that $\rho_{c}(t) < 2ct$ for any $t \in (0, \infty)$. Moreover, $\delta \ge c$ implies that $\rho_{c}(t+\delta) - \rho_{c}(t) \ge c \cdot \delta$ for $t \in [0, \infty)$. With these in mind, assume that $\delta > \frac{2k(c+2R)}{n-2k}$. By (\ref{eq:different}) and (\ref{eq:same}),
\begin{align*}
\sum_{i=1}^{n}  \rho_{c}\{d(y_i,& m^{c}_{n}(\bY_k))\} \ge  \sum_{i=1}^{k} \rho_{c}\{d(y_i, m^{c}_{n}(X))-(2R+\delta)\}  \\ & + \sum_{i=k+1}^{2k} \rho_{c}\{d(y_i, m^{c}_{n}(X))+\delta\} + \sum_{i=2k+1}^{n} \rho_{c}\{d(y_i, m^{c}_{n}(X))+\delta\}  \\
 \ge & -2kc^2 - 4cRk + \sum_{i=1}^{k} [2c \cdot \{d(y_i, m^{c}_{n}(X)) + d(y_{k+i}, m^{c}_{n}(X))\}] \\ & + \sum_{i=2k+1}^{n} [\rho_{c}\{d(y_i, m^{c}_{n}(X))\} +c\delta] \\
 \ge &-2kc^2 -4cRk +(n-2k)c\delta + \sum_{i=1}^{n} \rho_{c}\{d(y_i, m^{c}_{n}(X))\} \\ > & \sum_{i=1}^{n} \rho_{c}\{d(y_i, m^{c}_{n}(X))\},  
\end{align*} 
which contradicts the fact that $m^{c}_{n}(\bY_k)$ is the Huber mean of $\bY_k$. Thus, $\delta \le \frac{2k(c+2R)}{n-2k}$ and $d(m^{c}_{n}(\bX), m^{c}_{n}(\bY_k)) \le 2R + \delta \le 2R + \frac{2k(c+2R)}{n-2k}$ for all $\bY_k$. Therefore, $\epsilon^*(m^{c}_{n},\bX) \ge \floor{\frac{n+1}{2}/n}$.
\end{proof}

Note that Theorem~\ref{thm:robust:breakdown} holds only when the Huber loss is used. When the pseudo-Huber loss is used, we guess that the same result holds, but it has been subtly unproven.

\subsubsection{Proof of Theorem \ref{thm:breakdown:0.5}}

If the isometry group of $M$, denoted by $\textsf{Iso}(M)$, is transitive, then $M$ is said to be a \emph{homogeneous manifold} (of $\textsf{Iso}(M)$). The isometry group $\textsf{Iso}(M)$ is said to be \textit{transitive} if for any $m_{1}, m_{2} \in M$ there exists a $\psi \in \textsf{Iso}(M)$ such that $\psi(m_{1})=m_{2}$.

\begin{proof}[Proof of Theorem~\ref{thm:breakdown:0.5}]

The proof is primarily motivated by an argument used in \citep[Theorem 2.1,][]{lopuhaa1991breakdown}. Let $\bY_{k}$ be a corrupted set of observations from $\bX$. Specifically, $\bY_{k}$ includes $k$ arbitrary points on $M$ and $(n-k)$ points from $X$. Suppose that $\epsilon^{*}(m_n, \bX) > \floor{\frac{n+1}{2}}/n$. To show this, choose an integer $k$ such that $\floor{\frac{n+1}{2}} \le k < n$. Since $\epsilon^{*}(m_n, \bX) > \floor{\frac{n+1}{2}}/n$, there exists $0 < C < \infty$ such that, for every $\bY_{k}$,
\begin{equation}\label{eq:bound}
d(p, m_{n}(\bY_k)) \le C. 
\end{equation}
Here, we denote $p=m_{n}(\bX)$. In other words, $m_{n}(\mathbf{Y}_{k})$ lies in the closed ball of radius $C$ centered at $p$. Choose an isometry, say $\psi$, satisfying $d(p, \psi(p)) > 2C$. (It can be possible because $M$ is unbounded and $\textsf{Iso}(M)$ transitively acts on $M$.) It is known that all isometries on $M$ are bijective when $M$ is connected and complete. As $\psi$ is bijective, $\psi^{-1}$ exists and is also isometric. Note that $2(n-k) \le 2(n-\floor{\frac{n+1}{2}}) \le n$. Let $\bY_{k}$ be a corrupted set of observations (from $\bX$) that contains the points $\{x_{1}, x_{2}, ..., x_{n-k}, \psi(x_{1}), \psi(x_{2}), ..., \psi(x_{n-k})\}$, and define 
\begin{eqnarray*}
\bZ_{k} = \psi^{-1}(\bY_{k}) = \{\psi^{-1}(x_{1}), \psi^{-1}(x_{2}), \ldots, \psi^{-1}(x_{n-k}), \psi^{-1}(\psi(x_{1})), \\  
\qquad \psi^{-1}(\psi(x_{2})), \ldots, \psi^{-1}(\psi(x_{n-k})), \ldots \},
\end{eqnarray*}
which equals $\{x_{1}, x_{2}, ..., x_{n-k}, \psi^{-1}(x_{1}), \psi^{-1}(x_{2}), ..., \psi^{-1}(x_{n-k}), ... \}$. Both $\bY_{k}$ and $\bZ_{k}$ are corrupted sets of points in $\bX_{k}$ with $(n-k)$ fixed points, $x_{1}, x_{2}, ..., x_{n-k}$. The main idea of this proof is to observe that both $M$-valued estimators of $\mathbf{Y}_{k}, \mathbf{Z}_{k}$ cannot simultaneously lie in the closed ball of radius $C$ centered at $p$. From (\ref{eq:bound}), it holds that $d(p, m_{n}(\bY_{k})), d(p, m_{n}(\bZ_{k})) \le C$. Since $m_{n}$ is isometric-equivariant,
\begin{eqnarray*}
2C \ge d(p, m_{n}(\bY_{k})) + d(p, m_{n}(\bZ_{k})) &=& d(p, m_{n}(\bY_{k})) + d\{p, \psi^{-1}(m_{n}(\bY_{k}))\}  \\
&=& d(p, m_{n}(\bY_{k})) + d\{\psi(p), m_{n}(\bY_{k})\} \\
&\ge& d(p, \psi(p)) \\
&>& 2C,
\end{eqnarray*}
which is a contradiction. Therefore, $\epsilon^{*}(m_{n}, \bX) \le \floor{\frac{n+1}{2}}/n$.
\end{proof}

\subsubsection{Additional theorems for robustness of Huber means}

While the proof of Theorem 7 utilizes a fancy argument motivated by \citep[Theorem 2.1,][]{lopuhaa1991breakdown}, it primarily depends on the transitivity of $\mbox{Iso}(M)$. Under a small perturbation of the Riemannian metric (or deformation) of a given homogeneous manifold, its isometry group reduces to $\{\mbox{identity map}\}$. Although it might be possible that the result stated in Theorem 7 continues to hold for deformed manifolds, we are unable to verify it using the argument used in Theorem 7. 

Instead, we assume that $M$ belongs to the class of Hadamard manifolds.\footnote{That is, it is a simply connected and complete manifold with non-positive curvature. It is also written as ``global NPC manifold" \citep[Proposition 3.1,][]{sturm2003probability}.} 
Huber means are unique for data on a Hadamard manifold (see Corollary~\ref{cor:unique:sam}). 

\paragraph{Breakdown point of Huber means on Hadamard manifolds}
We first show that the breakdown point of (pseudo) Huber means for any $c \in [0, \infty]$ is not greater than 0.5.

\begin{theorem}
Suppose $M$ is a Hadamard manifold. For a prespecified $c \in [0, \infty]$ and for an $n$-tuple of observations on $M$, say $\textbf{X}=(x_1, x_2, \ldots, x_n)\in M^n$, it holds that  $\epsilon^{*}(m^{c}_{n},\mathbf{X})\le \floor{\frac{n+1}{2}}/n$ and $\epsilon^{*}(\tilde{m}^{c}_{n},\mathbf{X}) \le \floor{\frac{n+1}{2}}/n$. 
\begin{proof}[Proof of Theorem~\ref{thm:hada:supp}]

The essence of this proof is to note that Hadamard manifolds are diffeomorphic to $\mathbb{R}^{k}$ due to Cartan-Hadamard theorem \citep[e.g., see Theorem 1.10,][]{lee2006riemannian}. In particular, for any $p \in M$, $\mbox{Exp}_{p}$ is a diffeomorphism between $M$ and $T_{p}M ~(\cong \mathbb{R}^{k})$. Naturally, it holds that $r_{\mbox{\tiny inj}}(M)=\infty$ by the definition of injectivity radius, thereby implying that $m^{c}_{n}(\cdot)$ and $\tilde{m}^{c}_{n}(\cdot)$ are well-defined with probability 1 by Corollary B.5.

We begin with a proof for Huber means. If all the observations in $\mathbf{X}=(x_1,x_2,\ldots, x_{n})$ live in a single geodesic in $M$, then the proof easily ends, since the geodesic is isometric to $\mathbb{R}$. For this reason, we may assume that the observations of $\mathbf{X}$ do not lie in a single geodesic in $M$.

\textbf{Case I.} $n=2k$ for some positive integer $k$.

Suppose that $\epsilon^{*}(m^{c}_{n},\mathbf{X}) > \floor{\frac{n+1}{2}}/n = k/2k$. It means that there exists an $n$-tuple of observations $\textbf{X} = (x_1, x_2, \ldots, x_{2k})$ and a large $C>0$ such that $\overline{\textsf{conv}(\mathbf{X})} \subsetneq B_{C}\{m^{c}_{n}(\textbf{X})\}:= \{x\in M : d(x, m^{c}_{n}(\textbf{X})) < C \}$ holds and for any corrupted $\textbf{Y}_{k}=(x_1, \ldots, x_k, y_1, \ldots, y_k)$, it holds that
\begin{equation}\label{eq:contain:ball:supp}
d(m^{c}_{n}(\textbf{X}), m^{c}_{n}(\mathbf{Y}_{k})) < C, 
\end{equation}
which means that $m^{c}_{n}(\mathbf{Y}_{k}) \in B_{C}\{m^{c}_{n}(\mathbf{X}_{k})\}$. Now choose a point $p$ on the boundary of $B_{C}\{m^{c}_{n}(\textbf{X})\}$. For each $i$, define $y_i$ to be a reflection point of $x_i$ with respect to $p$; specifically, $y_{i} = \mbox{Exp}_{p}\{-\mbox{Log}_{p}(x_i)\}$ for any $i=1, 2, \ldots, k$. That is, the point $p$ is the midpoint of a unique geodesic joining $x_i$ and $y_i$. Notice that the Huber functional at $x$, $F^{c}_{n}(x)= \sum_{i=1}^{k}[\rho_{c}\{d(x, x_i)\} + \rho_{c}\{d(x, y_i)\}]$, is differentiable on $M \setminus \{x_1, x_2, \ldots, x_n\}$ and convex on the entire $M$, as the space is Hadamard. Hence, $-\textsf{grad}|_{x=p}F^{c}_{n}(x)$ is an inward-pointing vector across the boundary of $B_{C}\{m^{c}_{n}(\textbf{X})\}$ due to (\ref{eq:contain:ball:supp}). However, we get
\begin{eqnarray*}
-\textsf{grad}|_{x=p}F^{c}_{n}(x) &=& -\sum_{i=1}^{k} [\textsf{grad}|_{x=p}\rho_{c}\{d(x, x_i)\} + \textsf{grad}|_{x=p}\rho_{c}\{d(x, y_i)\}] \\
&=& \sum_{i=1}^{k} \rho'_{c}\{d(p, x_i)\} \cdot \frac{\mbox{Log}_{p}(x_i)}{\| \mbox{Log}_{p}(x_i) \|_{p}} + \rho'_{c}\{d(p, y_i)\} \cdot \frac{\mbox{Log}_{p}(y_i)}{\| \mbox{Log}_{p}(y_i) \|_{p}}\} \\
&=& \sum_{i=1}^{k} \rho'_{c}\{d(p, x_i)\}\{\frac{\mbox{Log}_{p}(x_i)}{\| \mbox{Log}_{p}(y_i) \|_{p}} - \frac{\mbox{Log}_{p}(x_i)}{\| \mbox{Log}_{p}(x_i) \|_{p}} \} \\
&=& \textbf{0} \in T_{p}M,
\end{eqnarray*}
meaning that p is the global minimum of $F^{c}_{n}$. It is a contradiction to (\ref{eq:contain:ball:supp}). 

\textbf{Case II.} $n=2k+1$ for some non-negative integer $k$.
An argument similar to the above easily verifies this case; specifically, for a corrupted $\textbf{Y}_{k}= (x_1, \ldots, x_k, y_1, \ldots, y_k, y_{k+1})$, define $y_i ~ (1 \le i \le k$) as same in Case I, and $y_{k+1} = \mbox{Exp}_{p}\{-\mbox{Log}_p(m^{c}_{n}(\mathbf{X}))\}$. Then, $-\textsf{grad}|_{x=p}F^{c}_{n}(x)$ is not an inward-pointing vector across the boundary of $B_{C}\{m^{c}_{n}(\textbf{X})\}$.

Owing to Cases I and II, we have proven that $\epsilon^{*}(m^{c}_{n},\mathbf{X}) \le \floor{\frac{n+1}{2}}/n$. The same argument is used to prove the result in the case of the pseudo-Huber mean; that is, $\epsilon^{*}(\tilde{m}^{c}_{n},\mathbf{X}) \le \floor{\frac{n+1}{2}}/n$. 
\end{proof}
\label{thm:hada:supp}
\end{theorem}

\begin{remark} In Theorem 7 of the main article, we consider the case where $M$ belongs to the class of (unbounded) Riemannian homogeneous spaces. In contrast, Theorem~\ref{thm:hada:supp} addresses the case where $M$ is a Hadamard manifold. 
Although both classes of manifolds are sufficiently large and practically important, they are not contained within one another.
The distinction between these two classes is illustrated in the following Venn diagram, along with representative examples.
\end{remark}
{ 
\begin{center} 
\begin{tikzpicture}
% Draw the circles for the Venn diagram

\draw[fill=blue!20, opacity=0.5] (4,0) circle (4); % Hadamard Manifolds circle
\draw[fill=red!20, opacity=0.5] (8,0) circle (4); % Unbounded Homogeneous Spaces circle
\centering
% Add labels for the circles
\node at (3.5, 4) {\textbf{Hadamard manifolds}};
\node at (9, 4) {\textbf{Homogeneous manifolds}};

% Add examples inside the circles

\node[align=center] at (2.1, 1.5) {Deformations of $\mathbb{H}^{k}$ \\
and their product spaces};
\node[align=center] at (2.1, -1.0) {Graph manifold \\
(with negative curvature)};

\node[align=center] at (6, 1.2)  {$\mbox{Sym}^{+}(k)$};
\node[align=center] at (6, 2.8) {$\mathbb{R}^k$};
\node[align=center] at (6, -2.5) {$\mathbb{R}^{3,1}$ \\(\small Minkowski space; pseudo-Riemannian)};
 
 \node[align=center] at (6, -1.4) {$\mathbb{C}\mathbb{H}^k$ \\ 
\small(complex hyperbolic space)};

\node[align=center] at (6, 2.0) {$\mathbb{H}^k$\\};
\node[align=center] at (6, 0) {$\mbox{GL}(k,m)$ \\
(\small Grassmannian manifold)};

\node[align=center] at (9.5, 2.8) {$S^1 \times \mathbb{R}^k$};
\node[align=center] at (9.5, 1.5) {$S^k$};
\node[align=center] at (9.5, -0.1) {$SO(k)$};

\node[align=center] at (9.5, -2.0) {$\mathbb{C}P^k$ \\
(complex projective space)};

% Intersection area (Hadamard and Unbounded Homogeneous)
%\draw[fill=purple!20, opacity=0.5] (2,0) circle (1.5); % Intersection circle
%%\label{fig:bendiagram}
\end{tikzpicture}
\end{center} 
}

\paragraph{Breakdown point of \Frechet\ means on Hadamard manifolds}
Motivated by the argument utilized in the proof of Theorem \ref{thm:hada:supp}, we secondly prove that the breakdown point of Fr\'{e}chet mean is asymptotically 0 when $M$ is a Hadamard manifold.

\begin{theorem}
Suppose $M$ is a Hadamard manifold. For an $n$-tuple of observations on $M$, say $\textbf{X}=(x_1, x_2, \ldots, x_n)$, it holds that $\epsilon^{*}(m^{\infty}_{n},\mathbf{X}) = 1/n$. 
\begin{proof}[Proof of Theorem~\ref{thm:frechet:supp}]
assume that $\epsilon^{*}(m^{\infty}_{n},\mathbf{X}) > 1/n$. It means that there exists an $n$-tuple of observations $\textbf{X} = (x_1, x_2, \ldots, x_{n})$ and a sufficiently large constant $C$ such that $\overline{\textsf{conv}(\mathbf{X})} \subsetneq B_{C}\{m^{\infty}_{n}(\textbf{X})\}$ and for any corrupted $\textbf{Y}_{1}=(y_1, x_2, x_3, \ldots, x_n) \in M^n$ it holds that
\begin{equation}\label{eq:contain:ball2:supp}
d(m^{\infty}_{n}(\textbf{X}), m^{c}_{n}(\textbf{Y}_{1})) < C
\end{equation}
Now select a point $p$ on the boundary of $B_{C}\{m^{c}_{n}(\textbf{X})\}$, and define $y_1 = \mbox{Exp}_{p}\{-2\sum_{i=2}^{n}\mbox{Log}_{p}(x_i)\}$. It is noteworthy that the \Frechet\ functional of $x$, $F^{\infty}_{n}(x)$, is differentiable and strictly convex on the entire $M$, as the space is Hadamard. On the other hand, we naturally obtain that  
\begin{eqnarray*}
-\textsf{grad}|_{x=p}F^{\infty}_{n}(x) &=& - \textsf{grad}|_{x=p} d^2(x, y_1) - \sum_{i=2}^{n} \textsf{grad}|_{x=p}d^2(x, x_i) \\
&=& 2\sum_{i=2}^{n}\mbox{Log}_{p}(x_i) - 2\sum_{i=2}^{n}\mbox{Log}_{p}(x_i) \\
&=& \textbf{0} \in T_{p}M,
\end{eqnarray*}
meaning that $p$ is equal to $m^{c}_{n}(\textbf{Y}_1)$. It is a contradiction to (\ref{eq:contain:ball2:supp}).
\end{proof}
\label{thm:frechet:supp}
\end{theorem}

\paragraph{Breakdown point of \Frechet\ mean is 0 when $M$ is unbounded.}

A natural question arising from Theorem~\ref{thm:frechet:supp} is whether Hadamard $M$ is necessary for \Frechet\ mean having the breakdown point of 0. However, the sample Huber means may not be unique in general Riemannian manifolds. To circumvent the problem, we redefine the breakdown point of the sample Huber means (including the \Frechet\ mean) as follows: 
Let $\bX$ be the data at hand, recall that $\mathbf{Y}_{k}$ is $\bX$ with its $k$ elements replaced by arbitrary values. 
We write $m^{c}_{n}$ for a sample Huber mean, which is an arbitrary member of $E^{c}_{n}(\mathbf{X})$, the set of sample Huber means. 
For the situation where $E_n^c(\mathbf{X})$ has multiple elements, we redefine the breakdown point of a potentially set-valued statistic $E_n^c$ at $\textbf{X}$ to be 
\begin{eqnarray}
\epsilon(m^{c}_{n}, \mathbf{X}) = \min_{1\le k\le n}\{\frac{k}{n}: \sup_{m^{c}_{n}(\mathbf{X}) \in E^{c}_{n}(\mathbf{X})} \sup_{\mathbf{Y}_{k}}\sup_{m^{c}_{n}(\mathbf{Y}_{k}) \in E^{c}_{n}(\mathbf{Y}_{k})} d(m^{c}_{n}(\mathbf{X}), m^{c}_{n}(\mathbf{Y}_{k})) = \infty \}.
\label{eq:redefine:break:supp}
\end{eqnarray}
Here, $E^{c}_{n}(\mathbf{X})$ (or $E^{c}_{n}(\mathbf{Y}_{k})$) denotes by the set of sample Huber means at $\mathbf{X}$ (or at $\mathbf{Y}_{k}$, respectively), and the supremum is taken over all possible $m^{c}_{n}(\mathbf{X})$, $\mathbf{Y}_{k}$, and  $m^{c}_{n}(\mathbf{Y}_{k})$. If $m_n^c$ is unique, then the definition (\ref{eq:redefine:break:supp}) reduces to the usual definition of the breakdown point (as defined in Section 3.7 of the main article).

% %After proving Theorem (\ref{thm:hada:supp}) and (\ref{thm:frechet:supp}), 
% Meanwhile, we contend that, if $M$ is merely unbounded then the breakdown point of the \Frechet\ mean exhibits 0, which can be demonstrated through a modified argument used in Theorems (\ref{thm:hada:supp}) and (\ref{thm:frechet:supp}). To this end, we extend the definition of breakdown point for Huber means even when Huber means are not unique. In Section 3.7 of the main paper, we assume that the sample Huber mean is unique. However, the sample Huber means may not be unique in general Riemannian manifolds. To circumvent the problem, we redefine the breakdown point of the sample Huber mean at $\bX$ to be 

% Here, $E^{c}_{n}(\mathbf{X})$ (or $E^{c}_{n}(\mathbf{Y}_{k})$) denotes by the set of sample Huber means at $\mathbf{X}$ (or $\mathbf{Y}_{k}$, respectively), the supremum is taken over all possible $m^{c}_{n}(\mathbf{X}), \mathbf{Y}_{k}, ~ \mbox{and} ~ m^{c}_{n}(\mathbf{Y}_{k})$. If $E^{c}_{n}(\mathbf{X})$ is singleton (i.e., the sample Huber mean is unique), then the definition of (\ref{eq:redefine:break:supp}) coincides with that of breakdown point defined in \Cref{sec:robust}.

Using the extended definition of the breakdown point (\ref{eq:redefine:break:supp}), We additionally show the result below.
\begin{theorem}
Suppose $M$ is unbounded. For any $n$-tuple of observations on $M$, say $\textbf{X}=(x_1, x_2, \ldots, x_n)$, it holds that  $\epsilon(m^{\infty}_{n},\mathbf{X})= 1/n$.  
\begin{proof}[Proof of Theorem~\ref{thm:frechet:zero:supp}]
Suppose $\epsilon(m^{\infty}_{n},\mathbf{X}) > 1/n$. Pick one element of $E^{c}_{n}(\mathbf{X})$, and denote by $m^{c}_{n}(\mathbf{X})$. For any choice of $m^{c}_{n}(\mathbf{Y}_1)$ there exists a sufficiently large $C < \infty$ such that all observations in $\mathbf{X}$ are contained in $B_{C}\{m^{c}_{n}(\mathbf{X})\}$ and for any corrupted $\textbf{Y}_{1}=(y_1, x_2, x_3, \ldots, x_n) \in M^n$ it holds that
\begin{equation}\label{unbounded:ball:supp}
d(m^{\infty}_{n}(\textbf{X}), m^{c}_{n}(\textbf{Y}_{1})) < C.
\end{equation}
It is known that there exists a point $p \in \partial\{B_{C}(m^{c}_{n}(\mathbf{X}))\}$ such that a geodesic ray starting from $m^{c}_{n}(\mathbf{X})$ through $p$ is an escaping geodesic with unit speed $\gamma:[0, \infty) \to M$. The term ``escaping" geodesic means that $d(\gamma(t), p) \underset{t\to \infty}{\to} \infty$ (e.g., see \citep[Exercise 7.6,][]{do1992riemannian}); that is, the escaping geodesic extends infinitely in a one direction, and does not oscillate wildly.

Denote $y_{1}= \mbox{Exp}_{p}[-(10n-1)\cdot \mbox{Log}_{p}\{(m^{c}_{n}(\mathbf{X})\}]$ and $\mathbf{Y}_{1}= (y_1, x_2, \ldots, x_n)$, and denote $q = \mbox{Exp}_{p}[-9\cdot\mbox{Log}_{p}\{(m^{c}_{n}(\mathbf{X})\}]$. Obviously, the point $q$ does not lie within the ball of radius $C$ center at $m^{c}_{n}(\mathbf{X})$. Given an $m^{c}_{n}(\mathbf{Y}_1)$, due to (\ref{unbounded:ball:supp}) 
\begin{eqnarray*}
\{(10n-1)C\}^2 \le d^2(y_1, m^{\infty}_{n}(\mathbf{Y}_{1})) &\le& F^{c}_{n}(m^{\infty}_{n}(\mathbf{Y}_{1})) ~ (\le F^{c}_{n}(q)) \\
&\le& d^2(y_1, q) + \sum_{i=2}^{n}d^2(x_i, q) \\ 
&\le& d^2(y_1, q) + \sum_{i=2}^{n} \{d(x_i, m^{\infty}_{n}(\mathbf{X})) + d(m^{\infty}_{n}(\mathbf{X}), q)\}^2 \\
&\le& \{10(n-1)C\}^2 + (n-1)\cdot (11C)^2.  
\end{eqnarray*}
It is equivalent to
\begin{equation*}
(10n-1)^2 \le 100(n-1)^2 + 121(n-1) \quad
\Leftrightarrow \quad 0 \le -59n - 22,
\end{equation*}
which is a contradiction. Since $m^{c}_{n}(\mathbf{X})$ was arbitrarily chosen from $E^{c}_{n}(\mathbf{X})$, the proof ends.
\end{proof}
\label{thm:frechet:zero:supp}
\end{theorem}

\color{black}
\section{Additional numerical results}

\subsection{Performances of the covariance estimator and the one-sample location test}

This subsection examines the numerical performance of the limiting covariance matrix estimator $\hat{A}_c = \hat{H}_c^{-1} \hat\Sigma_c \hat{H}_c^{-1} $ in Section \ref{subsec:est:cov}, and finite-sample type I error rates and powers of the one-sample location test discussed in Section \ref{subsec:application:hypo}.

To evaluate the performance of the estimator $\hat{A}_c$, we consider the situation where the limiting covariance matrix $A_c$ is estimated for the von Mises-Fisher distribution $\textsf{vMF}(m_0, 30)$ on the unit sphere $S^k$ (for $k = 2$ and $3$). Since $A_c$ itself is challenging to evaluate, we use an empirical covariance matrix of $\sqrt{{\tilde{n}}} \mbox{Log}_{m_0}(m_{\tilde{n}}^c)$, in which $m_{\tilde{n}}^c$ is computed from a large sample of size $ {\tilde{n}} = 1000$. Specifically, the true $A_c$ is approximated by $\textsf{A}_c := \frac{{\tilde{n}}}{R}\sum_{r= 1}^R  \{\mbox{Log}_{m_0}(m_{{\tilde{n}}}^c(r)) \} \cdot\{\mbox{Log}_{m_0}(m_{{\tilde{n}}}^c(r)) \}^T$, where $m_{{\tilde{n}}}^c(r)$ is the sample Huber mean, computed from the $r$th random sample of size ${\tilde{n}}$ for $r= 1, 2, \ldots, R$.

The estimator $\hat{A}_c$ (computed from a sample of size $n$) is then computed and compared with the (approximate) truth $\textsf{A}_c$ in terms of the relative error $\| \hat{A}_c - \textsf{A}_c\|_F /\| \textsf{A}_c\|_F$. Table \ref{tab:comparison} displays the averages (standard deviation) of the relative error, for varying sample sizes $n = 50, 100, 500, 1000$ and for a few representative choices of the parameter $c = 0.1, 0.3, 0.5$. (The proportion of observations whose distance from $m_0$ is greater than $c = 0.1, 0.3, 0.5$ are roughly 90\%, 50\%, 5\% for both $k = 2$ and $3$.) We observe that the relative error decreases as $n$ increases, and the estimator $\hat{A}_c$ appears to be an accurate estimator of $A_c$. 

\begin{table}[t]
\centering
\caption{Comparison of relative errors of $\hat{A}_c$ for $\textsf{vMF}(m_0, 30)$. Shown are the averages of relative errors with standard deviations in parenthesis.}
\label{tab:comparison}
\begin{tabular}{cccccc}
\hline
Manifold & $c$ & \( n = 50 \) & \( n = 100 \) & \( n = 500 \) & \( n = 1000 \) \\ \hline
&  \( 0.1 \) & 0.441 (0.322) & 0.294 (0.156) & 0.119 (0.052) & 0.086 (0.036) \\ 
\( S^2\) &  \( 0.3 \) & 0.275 (0.138) & 0.197 (0.096) & 0.089 (0.042) & 0.066 (0.029)\\
&  \( 0.5 \) & 0.241 (0.116) & 0.173 (0.079) & 0.090 (0.039) & 0.071 (0.030)\\ \hline
& \( 0.1 \) & 0.473 (0.213) & 0.321 (0.129) & 0.140 (0.047) & 0.102 (0.033)\\
\( S^3\)&  \(0.3 \) & 0.320 (0.119) & 0.226 (0.072) & 0.100 (0.028) & 0.073 (0.021)\\
& \( 0.5 \) & 0.290 (0.102) & 0.204 (0.068) & 0.096 (0.031) & 0.074 (0.022) \\ \hline
\end{tabular}
\end{table}

We next evaluate the finite-sample performance of the one-sample location test procedure utilizing the Huber mean. For this, we consider four models. 
\begin{enumerate}
    \item $\textsf{vMF}(m_0,30)$ on the unit sphere $S^2$
    \item $\tfrac{9}{10}\textsf{vMF}(m_0,30) + \tfrac{1}{10}\textsf{vMF}(m_0,1)$ on the unit sphere $S^2$
    \item Lognormal on $\mbox{Sym}^+(2)$, with mean $\delta \cdot 1_3$ and covariance $\tfrac{1}{2}I_3 +\tfrac{1}{2}1_31_3^T$ (in the vectorized coordinate)
    \item Log-Laplace on $\mbox{Sym}^+(2)$, with mean $\delta \cdot 1_3$ and scale parameter $\tfrac{5}{4}$
\end{enumerate}
For the first two models, the null hypothesis is set as $H_0:m_0 = \tilde{m}_0$ for a fixed location $\tilde{m}_0 \in S^2$. For each sample of size $n$, we have applied the proposed Huber mean-based test (in Section 3.5) for $c= 0.3$ at significance level $5\%$, while the true $m_0$ is chosen to satisfy either  
$d(m_0, \tilde{m}_0) = 0, 1, 2, 3, 4, 5$ degrees. We also compare the performance with that utilizing the \Frechet\ mean. Under the first model, we expect that both Huber and \Frechet\ mean-based tests perform similarly. On the other hand, the Huber mean-based test is expected to have higher powers for the second model. The second model is a scale mixture, thus exhibiting longer tails than the first model. 

For to the latter two models, the null hypothesis is set as $H_0: \delta = 0$ (or, equivalently, $m_0 = I_2$). We compare the size of the Huber mean-based test (for $c=3$) and the \Frechet\ mean-based test at significance level $5\%$, while the true mean 
$$m_0 = \delta
 \begin{pmatrix}
     1 & 0.5 \\
     0.5 & 1
 \end{pmatrix}$$
 is given by $\delta = 0, 0.05, 0.1, 0.15, 0.2, 0.25$. Here, Model 4 has longer tails than Model 3.
 
The probabilities of rejection of the test for finite sample sizes $n = 100, 300, 500, 1000$ are approximated using $1,000$ repetitions, and are shown in Tables~\ref{tab:comparison2} and \ref{tab:comparison3}. We observe that the type I error rates of the test procedure are well controlled at the significance level of 0.05, and that the power of the test sharply increases as sample size increases and as the effect size increases. Comparing Model 1 against Model 2 (and also Model 3 against Model 4), we observe that the power of the Huber mean-based test is higher than the \Frechet\ mean-based test, when the distribution has a higher tail. 

% \begin{table}[t]
% \centering
% \caption{Rejection rates of the one-sample location test.}
% \label{tab:comparison2}
% \begin{tabular}{c|cccc}
% \hline
% $d(m_0,\tilde{m}_0)$ (degrees)  & $n = 100$& $n = 300$ & $n = 500$ & $n = 1000 $\\ \hline
% $0^\circ$ & 0.048 & 0.047 & 0.053  & 0.063\\  
% $1^\circ$ & 0.122 & 0.253 & 0.439  & 0.755\\  
% $2^\circ$ & 0.346 & 0.821 & 0.961 & 1.000\\  
% $3^\circ$ & 0.664 & 0.996 & 1.000 & 1.000\\  
% $4^\circ$ & 0.892 & 1.000 & 1.000 & 1.000\\  
% $5^\circ$ & 0.983 & 1.000 & 1.000 & 1.000\\   \hline
% \end{tabular}
% \end{table}
 
\begin{table}[t]
\centering
\caption{Models 1 and 2. Rejection rates of the one-sample location test. Leftmost column shows the values of $d(m_0,\tilde{m}_0)$ (degrees). }
\label{tab:comparison2}
\begin{tabular}{c|cccc|cccc}
\hline
Model 1 & \multicolumn{4}{c|}{Huber} & \multicolumn{4}{c}{\Frechet} \\
\hline
$n$  & $100$& $300$ & $500$ & $1000 $
& $100$& $300$ & $500$ & $1000 $\\ \hline
$0^\circ$ & 0.048 & 0.036 & 0.06 & 0.044 & 0.048 & 0.047 & 0.061 & 0.052 \\ 
$1^\circ$ & 0.133 & 0.255 & 0.458 & 0.755 & 0.141 & 0.268 & 0.468 & 0.771 \\
$2^\circ$ & 0.364 & 0.814 & 0.964 & 1 & 0.401 & 0.823 & 0.977 & 1 \\
$3^\circ$ & 0.674 & 0.998 & 1 & 1 & 0.693 & 0.998 & 1 & 1 \\ 
$4^\circ$ & 0.896 & 1 & 1 & 1 & 0.906 & 1 & 1 & 1 \\
$5^\circ$ & 0.983 & 1 & 1 & 1 & 0.984 & 1 & 1 & 1 \\ 
\hline
\\
\hline
Model 2 & \multicolumn{4}{c|}{Huber} & \multicolumn{4}{c}{\Frechet} \\
\hline
$n$  & $100$& $300$ & $500$ & $1000 $
& $100$& $300$ & $500$ & $1000 $\\ \hline
$0^\circ$ &  0.044 & 0.039 & 0.048 & 0.046 & 0.049 & 0.049 & 0.037 & 0.041 \\
$1^\circ$ &     0.105 & 0.211 & 0.347 & 0.653 & 0.064 & 0.093 & 0.128 & 0.254 \\
$2^\circ$ &     0.263 & 0.708 & 0.912 & 0.999 & 0.131 & 0.304 & 0.465 & 0.783 \\
$3^\circ$ &      0.540 & 0.968 & 1 & 1 & 0.248 & 0.562 & 0.777 & 0.982 \\
$4^\circ$ &    0.805 & 1 & 1 & 1 & 0.396 & 0.850 & 0.967 & 1 \\
$5^\circ$ &    0.950 & 1 & 1 & 1 & 0.571 & 0.964 & 0.997 & 1 \\   \hline
\end{tabular}
\end{table}

\begin{table}[!t]
\centering
\caption{Models 3 and 4. Rejection rates of the one-sample location test. The Leftmost column shows the values of $\delta$. }
\label{tab:comparison3}
\begin{tabular}{c|cccc|cccc}
\hline
Model 3 & \multicolumn{4}{c|}{Huber} & \multicolumn{4}{c}{\Frechet} \\
\hline
$n$  & $100$& $300$ & $500$ & $1000 $
& $100$& $300$ & $500$ & $1000 $\\ \hline
        0.00  & 0.026 & 0.041 & 0.036 & 0.033 & 0.025 & 0.041 & 0.038 & 0.030 \\
        0.05  & 0.042 & 0.087 & 0.129 & 0.310 & 0.041 & 0.089 & 0.125 & 0.301 \\
        0.10  & 0.097 & 0.328 & 0.568 & 0.907 & 0.094 & 0.340 & 0.576 & 0.916 \\
        0.15  & 0.240 & 0.732 & 0.930 & 0.999 & 0.243 & 0.727 & 0.938 & 0.999 \\
        0.20  & 0.427 & 0.940 & 0.994 & 1 & 0.418 & 0.942 & 0.995 & 1 \\
        0.25  & 0.644 & 0.993 & 1 & 1 & 0.653 & 0.995 & 1 & 1 \\
\hline
\\
\hline
Model 4 & \multicolumn{4}{c|}{Huber} & \multicolumn{4}{c}{\Frechet} \\
\hline
$n$  & $100$& $300$ & $500$ & $1000 $
& $100$& $300$ & $500$ & $1000 $\\ \hline
        0.00  & 0.012 & 0.017 & 0.017 & 0.016 & 0.009 & 0.013 & 0.014 & 0.016 \\
        0.05  & 0.015 & 0.044 & 0.073 & 0.141 & 0.019 & 0.039 & 0.061 & 0.103 \\
        0.10  & 0.044 & 0.217 & 0.342 & 0.709 & 0.041 & 0.174 & 0.283 & 0.604 \\
        0.15  & 0.113 & 0.499 & 0.784 & 0.992 & 0.089 & 0.402 & 0.671 & 0.964 \\
        0.20  & 0.231 & 0.802 & 0.973 & 1 & 0.193 & 0.698 & 0.937 & 1 \\
        0.25  & 0.396 & 0.948 & 0.999 & 1 & 0.331 & 0.901 & 0.995 & 1 \\

\hline

\end{tabular}
\end{table}

\subsection{QQ plots for the bootstrap-approximate distribution}

As referenced in Section 4.2 of the main article, we provide the quantile-quantile plots of the bootstrap replicates for the sample Huber mean and the Fr\'{e}chet mean in Figure~\ref{fig:morphometry2}. The bootstrap replicates of Fr\'{e}chet and Huber means are normally distributed (except their tails), which is consistent with the result of Theorem~\ref{thm:clt}.
\begin{figure}[!ht]
\centering
\includegraphics[width = \textwidth]{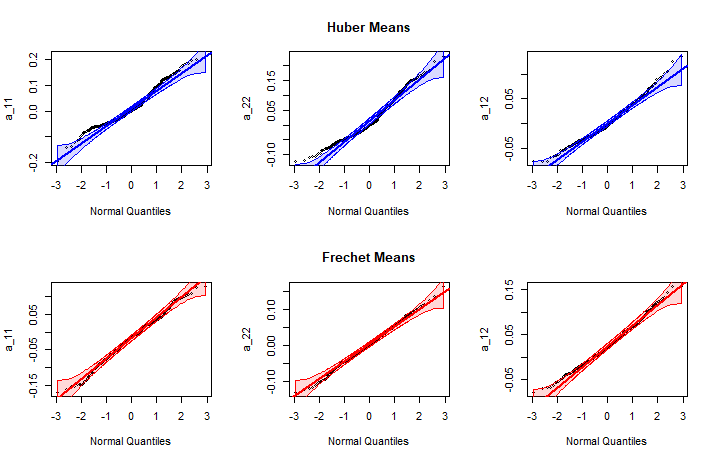}
\caption{Element-wise QQ plots of the bootstrap replicates for the sample Huber mean and the Fr\'{e}chet mean. Approximate 95$\%$ pointwise confidence envelopes are shown as transparent strips.}
 \label{fig:morphometry2}
 \end{figure}

\end{appendix}

\bibliographystyle{imsart-nameyear}
\bibliography{hmr}     

\end{document}